\documentclass[11pt]{amsart}

\usepackage{enumerate}
\usepackage{enumitem}
\usepackage{amsxtra}
\usepackage{amsmath}
\usepackage{amscd}
\usepackage{amssymb}
\usepackage{amsfonts}
\usepackage[all]{xy}
\usepackage{amsthm} 
\usepackage{color}
\usepackage{upgreek}
\usepackage{latexsym}
\usepackage[english]{babel} 
\usepackage{hyperref}
\usepackage{nicefrac}
\usepackage{euscript}
\usepackage[foot]{amsaddr}
\usepackage{amscd}
\usepackage[sort]{cite}           
\usepackage{fancyhdr}          
\usepackage{amsfonts}
\usepackage{color}
\usepackage{upgreek}
\usepackage{graphicx}
\usepackage{tikz}
\usepackage{hyperref}
\usepackage{stmaryrd}
\usepackage{nicefrac}
\usepackage{wasysym}
\xyoption{all}
\theoremstyle{plain}
\newtheorem{theorem}{Theorem}[section]
\newtheorem{lemma}[theorem]{Lemma}
\newtheorem{lemma-definition}[theorem]{Lemma-Definition}
\newtheorem{definition-lemma}[theorem]{Definition-Lemma}

\newtheorem{proposition}[theorem]{Proposition}

\newtheorem{corollary}[theorem]{Corollary}

\newtheorem*{theorem-A}{Theorem A}
\newtheorem*{theorem-B}{Theorem B}
\newtheorem*{theorem-C}{Theorem C}
\newtheorem*{theorem-D}{Theorem D}
\newtheorem*{conjecture-A}{Conjecture A}
\newtheorem*{conjecture-B}{Conjecture B}
\newtheorem*{conjecture-C}{Conjecture C}
\newtheorem*{conjecture-D}{Conjecture D}

\theoremstyle{definition}

\theoremstyle{remark}
\newtheorem{remark}[theorem]{Remark}

\numberwithin{equation}{section}

\def\C{\mathrm{C}}
\def\D{\mathrm{D}}

\def\O{\mathrm{O}}

\def\SS{\mathrm{SS}}

\def\U{\mathrm{U}}

\def\X{\mathrm{X}}

\def\bbA{\mathbb{A}}

\def\bbC{\mathbb{C}}
\def\bbD{\mathbb{D}}

\def\bbN{\mathbb{N}}

\def\bbZ{\mathbb{Z}}

\def\frakg{\mathfrak{G}}

\def\frakL{\mathfrak{L}}
\def\frakM{\mathfrak{M}}

\def\frakP{\mathfrak{P}}

\def\frakZ{\mathfrak{Z}}

\def\calA{\mathcal{A}}

\def\calC{\mathcal{C}}
\def\calD{\mathcal{D}}
\def\calE{\mathcal{E}}
\def\calF{\mathcal{F}}

\def\calH{\mathcal{H}}

\def\calL{\mathcal{L}}
\def\calM{\mathcal{M}}

\def\calO{\mathcal{O}}
\def\calP{\mathcal{P}}

\def\calS{\mathcal{S}}

\def\calU{\mathcal{U}}
\def\calV{\mathcal{V}}
\def\calW{\mathcal{W}}

\def\calZ{\mathcal{Z}}

\def\frakg{\mathfrak{g}}

\def\frakl{\mathfrak{l}}

\def\frakp{\mathfrak{p}}

\def\fraks{\mathfrak{s}}

\def\bfc{\mathbf{c}}

\def\bfv{\mathbf{v}}
\def\bfw{\mathbf{w}}

\def\bfC{\mathbf{C}}

\def\bfK{\mathbf{K}}

\def\bfO{\mathbf{O}}
\def\bfP{\mathbf{P}}

\def\ua{{\underline a}}

\def\ux{{\underline x}}

\def\k{{\operatorname{k}\nolimits}}

\def\Isom{\operatorname{Isom}\nolimits}

\def\tor{\operatorname{tor}\nolimits}

\def\CC{\operatorname{CC}\nolimits}

\def\Rep{\frak{R}}
\def\Coh{\operatorname{Coh}\nolimits}
\def\DCoh{{\operatorname{DCoh}\nolimits}}

\def\Perf{{\operatorname{Perf}\nolimits}}

\def\sg{{\operatorname{sg}\nolimits}}

\def\soc{{\operatorname{soc}\nolimits}}
\def\top{{\operatorname{top}\nolimits}}
\def\alg{{\operatorname{alg}\nolimits}}

\def\top{{\operatorname{top}\nolimits}}
\def\top{{\operatorname{top}\nolimits}}

\def\op{{{\operatorname{op}\nolimits}}}
\def\Ad{{{\operatorname{Ad}\nolimits}}}
\def\-{{\operatorname{-}\!}}

\def\IC{\operatorname{IC}\nolimits}
\def\Im{\operatorname{Im}\nolimits}
\def\Ker{\operatorname{Ker}\nolimits}
\def\Hom{\operatorname{Hom}\nolimits}

\def\End{\operatorname{End}\nolimits}

\def\Hilb{\operatorname{Hilb}\nolimits}

\def\Aut{\operatorname{Aut}\nolimits}
\def\Res{\operatorname{Res}\nolimits}
\def\reg{{\operatorname{reg}\nolimits}}

\def\Ext{\operatorname{Ext}\nolimits}

\def\id{\operatorname{id}\nolimits}

\def\supp{{\operatorname{supp}\nolimits}}

\def\Tr{\operatorname{Tr}\nolimits}
\def\Gr{{\operatorname{Gr}\nolimits}}

\def\crit{\operatorname{crit}\nolimits}

\def\Quot{{{\operatorname{Quot}\nolimits}}}

\def\Db{\operatorname{D^b\!}\nolimits}
\def\D{{\operatorname{D}\nolimits}}
\def\b{{\operatorname{b}\nolimits}}
\def\Acyclic{\operatorname{Acyclic}\nolimits}
\def\boxplus{{\oplus}}

\def\ev{\operatorname{ev}\nolimits}

\def\Tr{\operatorname{Tr}\nolimits}

\def\can{\operatorname{can}\nolimits}

\def\nil{{\operatorname{nil}\nolimits}}
\def\ch{\operatorname{ch}\nolimits}
\def\circlestar{{\Circle\!\!\!\!\star\;}}

\def\SG{{D}}

\def\sgn{{o}}

\numberwithin{itemcounter}{subsection}
\numberwithin{equation}{section}
\usepackage[toc,page]{appendix}

\usepackage{etoolbox}
\appto\appendix{\addtocontents{toc}{\protect\setcounter{tocdepth}{1}}}

\title[Quantum loop groups and critical convolution algebras]
{Quantum loop groups and critical convolution algebras}

\author{M. Varagnolo$^1$} 
\address{\scriptsize{$^1$~CY Cergy Paris Universit\'e,  95302 Cergy-Pontoise, France,
UMR8088 (CNRS), ANR-18-CE40-0024 (ANR).}}
\author{E. Vasserot$^2$} 
\address{\scriptsize{$^2$~Universit\'e Paris Cit\'e, 75013 Paris, France, UMR7586 (CNRS), 
ANR-18-CE40-0024 (ANR), Institut Universitaire de France (IUF).
}}

\begin{document}
\maketitle

\begin{abstract}
We realize geometrically a family of simple modules of 
(shifted) quantum loop groups including Kirillov-Reshetikhin and prefundamental representations.
To do this, we introduce a new family of algebras attached to quivers with potentials, 
using critical K-theory and critical Borel-Moore homology, which
generalizes the convolution algebras attached to quivers defined by Nakajima. 
\end{abstract}

\section{Introduction and notation}
\subsection{Introduction}
Quiver varieties were introduced by Nakajima in \cite{N94}, \cite{N98} and \cite{N00}.
The equivariant K-theory of Steinberg varieties
attached with quiver varieties, equipped with a convolution product,
yields a family of algebras closely related
to symmetric quantum loop groups which
is important for the finite dimensional modules and
their $q$-characters, see \cite{N03}, \cite{N04}.
In this work we introduce a new family of convolution algebras attached to quiver varieties
with potentials. Here the K-theory is replaced by the critical K-theory.
We'll call them K-theoretical critical convolution algebras.
By critical K-theory we mean the Grothendieck group of derived factorization categories, or singularity 
categories, attached to equivariant LG-models. 
We'll also consider a cohomological analog of these convolution algebras, 
where the critical K-theory
is replaced by the cohomology of some vanishing cycle sheaves. The corresponding algebras are called
the cohomological critical convolution algebras.

The main motivation comes from the representation theory of (shifted) quantum loop groups.
Nakajima realized quantum loop groups via convolution algebras of quiver varieties. 
This construction permits to recover the classification of the simple finite dimensional 
modules of quantum loop groups, but it does not give a geometric construction of those. 
More precisely, the cohomology or K-theory of quiver varieties yields a geometric realization 
of the standard modules,
and the simple modules are the Jordan H\"older constituents of the standards.
In a similar way, we are able to realize the quantum loop groups and the shifted quantum loop 
groups of symmetric types via critical convolution algebras.
Remarkably, the critical cohomology or K-theory also gives a realization of the 
simple modules in several settings :
we realize both a family of simple modules of quantum loop groups containing
all Kirillov-Reshetikhin modules and a family of simple modules of shifted quantum loop groups containing
all tensor products of negative prefundamental modules
as the critical cohomology or K-theory of LG-models attached to quivers.
This construction is partly motivated by the work of Liu in \cite{L20} where some representations
of shifted quantum loop groups are constructed via the cohomology of quasi-maps spaces using
some limit procedure similar to the limit procedure of Hernandez-Jimbo in \cite{HJ12}.
This limit procedure admits also a natural interpretation in critical cohomology or K-theory.
In the work \cite{VV23} we give another achievement of critical convolution algebras :
they yield a geometric realization of all
quantum loop groups and shifted quantum loop groups, not necessarily of symmetric types, which generalize Nakajima's construction in \cite{N00}.
Moreover, using critical convolution algebras we get a geometric realization of the Kirillov-Reshetikhin and prefundamental modules of arbitrary types.

Another motivation comes from cluster theory.
Using cluster algebras, Hernandez-Leclerc give in \cite{HL16} 
a $q$-character formula for  
prefundamental and Kirillov-Reshetikhin representations in terms of Euler characteristic of quiver grassmannians. 
Their character formula does not give any geometric realization of the (shifted) quantum loop group action.
It is surprising that our construction yields indeed a representation of the (shifted) quantum loop group 
in the cohomology of the same quiver grassmannians, with coefficients in some constructible sheaves.
The Kirillov-Reshetikhin modules are particular cases of reachable
modules for the cluster algebra structure on the Grothendieck ring of the quantum loop group considered
in \cite{KKOP20}. 
The Euler characteristic description of the $q$-characters extends to all reachable modules.
We expect that all reachable modules admit also 
a realization in critical cohomology or K-theory.

A third motivation comes from the K-theoretical Hall algebras and cohomological Hall algebras.
We'll define an algebra homomorphism from 
K-theoretical Hall algebras to K-theoretical critical convolution algebras  using Hecke correspondences.
As a consequence, the K-theoretical critical convolution algebras may be viewed as some doubles 
of the K-theoretical Hall algebras introduced by Padurariu.
These doubles are the good setting for representation theory. 
Note that, depending on how the potential is chosen, different doubles
of the same K-theoretical Hall algebras can be realized via different 
K-theoretical critical convolution algebras.
We'll consider two examples.
The first one is isomorphic to Nakajima's convolution algebra via dimensional reduction,
and is related to quantum loop groups.
The second one is obtained with a different potential and is related to 
shifted quantum loop groups with antidominant shifts. 
Recall that K-theoretical Hall algebras of preprojective algebras have been introduced  
by Schiffmann-Vasserot in \cite{SV13a}, in the case of the Jordan quiver. The
case of a general quiver was considered by Varagnolo-Vasserot in \cite{VV22} where it
is proved that, modulo twisting the Hall multiplication, 
K-theoretical Hall algebras of preprojective algebras of quivers of finite or affine type are isomorphic to 
affine quantum groups or toroidal quantum groups in the sense of \cite{GKV}. 
The K-theoretical Hall algebras of a quiver with potential was introduced by Padurariu in \cite{P21}. 
It was proved there that Isik's Koszul duality (=dimensional reduction) implies
that the K-theoretical Hall algebras of triple quivers with some particular potential 
$\bfw$ coincides with the K-theoretical Hall algebras of preprojective algebras.
In parallel, cohomological Hall algebras were introduced in two versions.
Cohomological Hall algebras of preprojective algebras have been introduced 
by Schiffmann-Vasserot in \cite{SV13b}, in the case of the Jordan quiver. 
The case of a general quiver was considered in  \cite{SV18} and \cite{YZ20}.
Cohomological Hall algebras of quivers with potential
were introduced by Kontsevich-Soibelman in \cite{KS11}.
The coincidence of cohomological Hall algebras of triple quivers with the potential 
$\bfw$ and of
cohomological Hall algebras of preprojective algebras was established by Davison \cite{RS17} and 
Yang-Zhao \cite{YZ20}.

The contents of the paper may be summarized as follows. 
Section 2 recalls the singularity categories and derived factorization categories associated with
$G$-equivariant LG-models and their algebraic and topological K-theories.
Then, given a $G$-equivariant LG-model $(X,\chi,f)$ with a $G$-equivariant projective morphism
to an affine $G$-variety $X\to X_0$, we define
a monoidal category $\DCoh_G(X^2,f^{(2)})_Z$ 
of Steinberg correspondences supported in the fiber product $Z=X\times_{X_0}X$.
Taking the algebraic or topological Grothendieck groups yields
the K-theoretical critical convolution algebra in Corollary \ref{cor:critalg1}. 
Using vanishing cycles instead we define similarly
the cohomological critical convolution algebra in Proposition \ref{prop:critalg3}.
More precisely, we prove the following :

\begin{enumerate}[label=$\mathrm{(\alph*)}$,leftmargin=8mm,itemsep=1.5mm]
\item[-]
$K_G(X^2,f^{(2)})_Z$ and
$K^\top_G(X^2,f^{(2)})_Z$ are associative $R_G$-algebras,
\item[-]
 $H^\bullet_G(X^2,f^{(2)})_Z$ is an associative $H^\bullet_G$-algebra.
\end{enumerate}

In Section 3 we consider some particular equivariant LG-models associated with quivers
with potentials. We first give a reminder on Nakajima's quiver varieties and on triple quiver varieties.
The corresponding critical convolution algebras depend on the choice of a potential $\bfw$.
This potential also gives rise to a nilpotent K-theoretical Hall algebra $\calU^+$
in the sense of \cite{P21}, \cite{VV22}.
We expect the K-theoretical critical convolution algebra
to be equipped with an algebra homomorphism from a double $\calU$ of $\calU^+$,
i.e., an algebra with a triangular decomposition 
$\calU=\calU^+\otimes\calU^0\otimes\calU^-$ where
 $\calU^-$ is the opposite of $\calU^+$ and
$\calU^0$ is commutative.
Different choices of the potential may yield different doubles
of the same algebra.
We illustrate this for two explicit potentials attached to the same Dynkin quiver in 
Theorems \ref{thm:Nakajima} and \ref{thm: Nakajima shifted}.
In these cases  $\calU$ is 
either the quantum loop group $\U_\zeta(L\frakg)$ or the negatively shifted quantum loop group
$\U_\zeta^{-w}(L\frakg)$. In the first case the construction reduces to Nakajima's one.
The second one is new. In both case, we get new applications to representation theory.
More precisely, we prove the following :

\begin{theorem}\label{thm:A}
Let $Q$ be a Dynkin quiver. 
\hfill
\begin{enumerate}[label=$\mathrm{(\alph*)}$,leftmargin=8mm]
\item
There is an algebra homomorphism
$\U_\zeta(L\frakg)\to 
K\big(\widetilde\frakM^\bullet(W)^2,(\tilde f_1^\bullet)^{(2)}\big)_{\widetilde\calZ^\bullet(W)}.$
\item
There is an algebra homomorphism
$\U_\zeta^{-w}(L\frakg)\to 
K\big(\widetilde\frakM^\bullet(W)^2,(\tilde f^\bullet_2)^{(2)}\big)_{\widetilde\calZ^\bullet(W)}$.
\end{enumerate}
\end{theorem}

The graded triple quiver variety  $\widetilde\frakM^\bullet(W)$ is defined in Section \ref{sec:triple quiver}.
We consider also the case of the Jordan quiver
in Theorem \ref{thm:Jshifted}.

The geometric construction of the quantum loop groups given above gives rise to some some 
geometrically defined representations which are considered in Section 4.
Before to do that, we consider in Theorem \ref{thm:moduleK}
a deformation of the map (a) in Theorem \ref{thm:A} associated with 
a choice of an $\fraks\frakl_2$-triple called an admissible triple, see
\S\ref{sec:admissible}. 
In Theorems \ref{thm:HL1} and Proposition \ref{prop:TPKR}, 
we give a geometric realization of a family of simple finite dimensional 
modules of quantum loops groups containing all Kirillov-Reshetikhin modules using
the critical cohomology and K-theory of graded quiver varieties. 
In the shifted case, we realize in Theorems \ref{thm:PF1}, \ref{thm:PF2}
all tensor products of negative prefundamental modules 
as the critical K-theory or cohomology of a graded triple quiver variety.
A special case of our results yields the following.

\begin{theorem}
\hfill
\begin{enumerate}[label=$\mathrm{(\alph*)}$,leftmargin=8mm]
\item The Kirillov-Reshetikhin modules of the quantum loop group
are realized in the critical K-theory or cohomology
of Nakajima's quiver varieties.
\item The negative prefundamental modules of the shifted quantum loop group 
are realized in the critical K-theory or cohomology
of triple quiver varieties.
\end{enumerate}
\end{theorem}

The proof is based on the following facts.
\begin{enumerate}[label=$\mathrm{(\alph*)}$,leftmargin=8mm,itemsep=1.5mm]
\item[-]
In the non-shifted case,
the critical locus of the potential is identified in Proposition \ref{prop:HL1} with the quiver grassmannian 
used by Hernandez-Leclerc in \cite{HL16} to relate
the $q$-characters of Kirillov-Reshetikhin modules with cluster algebras.
\item[-]
In the shifted case,
the critical locus of the potential is also identified in Proposition \ref{prop:HL3}
with a quiver grassmannian which is
used in \cite{HL16}.
\item[-]
In Theorem \ref{thm:limH} the critical K-theory or cohomology yields a geometric realization 
of the limit procedure in \cite{HJ12}.
\end{enumerate}

\noindent As mentioned above, we expect all reachable simple modules
to admit a realization similar to the one in Theorem \ref{thm:A}.
A different geometric realization of some Kirillov-Reshetikhin modules appears in the work of Liu in \cite{L20},
using critical K-theory of quasi-maps spaces.

The appendix A is a reminder of basic facts on representations of (possibly shifted) quantum loop groups
which are used throughout the paper.
In our theory we mainly consider quiver of finite types.
We might as well have considered general quivers and, e.g., toroidal quantum groups in the sense of
\cite{GKV}. 
In the appendix B we give analogues of Theorems \ref{thm:Nakajima} and \ref{thm: Nakajima shifted}
for the toroidal quantum group of $\frakg\frakl_1$ and its shifted version.
The appendix C contains a second proof of some version of Theorem \ref{thm:HL1} using microlocal geometry.
The appendix D is a reminder on algebraic and topological critical K-theory.

\subsection{Notation and conventions}
All schemes are assumed to be separated schemes, locally of finite type, over the field $\bbC$.
We may allow an infinite number of connected components, but each of them
is assumed to be of finite type.
Given a scheme $X$ with an action of an affine group $G$, let
$\Db\Coh_G(X)$ be the bounded derived category of the category $\Coh_G(X)$ of $G$-equivariant coherent 
sheaves $X$ and let $\Perf_G(X)$ be the full subcategory of perfect complexes.
For each $G$-invariant closed subscheme $Z$ let $\Coh_G(X)_Z$
be the category of coherent sheaves with set-theoretic support in $Z$, and let
$\Db\Coh_G(X)_Z$ be the full triangulated subcategory of $\Db\Coh_G(X)$ 
consisting of the complexes with cohomology supported on $Z$.
We'll say that a $G$-invariant morphism $\phi:Y\to X$ of $G$-schemes is of finite $G$-flat 
dimension if the pull-back functor
$L\phi^*:\D^-\Coh_G(X)\to\D^-\Coh_G(Y)$ 
takes $\Db\Coh_G(X)$ to $\Db\Coh_G(Y)$.
Similarly, a $G$-equivariant quasi-coherent sheaf has a finite $G$-flat dimension if it admits a finite resolution
by $G$-equivariant flat quasi-coherent sheaves.
We'll say that a $\bbC^\times$-action on $X$ is circle compact if
the fixed points locus in each connected component is compact
and the limit $\lim_{t\to 0}\lambda(t)\cdot x$ exists for each closed point $x$.

Let $K_0(\calC)$ be the complexified Grothendieck group of an Abelian or triangulated category $\calC$.
Let $R_G$ be the complexified Grothendieck ring of the group $G$,
and $F_G$ be the fraction field of $R_G$.
We'll abbreviate
$R=R_{\bbC^\times}=\bbC[q,q^{-1}]$ and $F=F_{\bbC^\times}=\bbC(q)$.
We'll also set
$K_G(X)=K_0(\Perf_G(X)),$
$K^G(X)=K^G(\Db\Coh_G(X))$ and
$K^G(X)_Z=K_0(\Db\Coh_G(X)_Z).$
Note that $K^G(X)_Z=K^G(Z)$. 
If $G=\{1\}$ we abbreviate $K(X)=K^G(X)$.
We'll write
$$\Lambda_a(\calE)=\sum_{i\geqslant 0}a^i\Lambda^i(\calE)\in K_G(X)
 ,\quad
\calE\in K_G(X)
 ,\quad
a\in R_G^{\times}.$$
Let $H_G^\bullet(X,\calE)$
denote the equivariant cohomology of a $G$-equivariant sheaf $\calE$ on $X$.
We abbreviate 
$H^\bullet_G=H_G^\bullet(\{\text{pt}\},\bbC).$
Let $H_\bullet^G(X,\bbC)$ denote the $G$-equivariant Borel-Moore homology over $\bbC$.

A derived scheme is a pair $X=(|X|,\calO_X)$ where $|X|$ is a topological space and $\calO_X$
is a sheaf on $|X|$ with values in the $\infty$-category of simplicial commutative rings such that
the ringed space $(|X|,\pi_0\calO_X)$ 
is a scheme and the sheaf $\pi_n\calO_X$ is a quasi-coherent 
$\pi_0\calO_X$-module over this scheme for each $n>0$.
Here, all derived schemes will be defined over $\bbC$, hence derived schemes can be
modeled locally by dg-algebras rather than simplicial ones.
Let $M$ be a smooth quasi-affine $G$-scheme
and $\sigma$ a $G$-invariant section of a $G$-equivariant vector bundle $E$ over $M$.
The derived zero locus of $\sigma$
is the derived $G$-scheme $X=R(E\to M,\sigma)$ given by the derived fiber product
$M\times_E^RM$ relative to the maps $\sigma,0:M\to E$.
The derived scheme $X$ is quasi-smooth, i.e., it is finitely presented and its cotangent 
complex is of cohomological amplitude $[-1,0]$.
For any derived $G$-scheme $X$, let $\Db\Coh_G(X)$ be the derived category
of modules over $\calO_X$ with bounded coherent cohomology.
A $G$-invariant morphism $\phi:Y\to X$ of derived $G$-schemes has finite $G$-flat dimension
if the functor $L\phi^*$ takes bounded complexes to bounded ones.

Given two schemes $X_1$, $X_2$ and functions $f_a:X_a\to\bbC$ with $a=1,2$, 
we define $f_1\oplus f_2:X_1\times X_2\to\bbC$ to be the function
$f_1\oplus f_2=f_1\otimes 1+1\otimes f_2.$ If $X_1=X_2=X$, and $f_1=f_2=f$ we abbreviate
$f^{\oplus 2}=f\oplus f$ and
$f^{(2)}=f\oplus(-f).$

All categories will be assumed to be essentially small, i.e., equivalent to a small category.
Let $\calC^\op$ denote the opposite of a category $\calC$.
Let $\overline\calC$ denote the idempotent completion of an additive category $\calC$.
If $\calC$ is either a triangulated category or a dg category, then so is also $\overline\calC$.
For any dg category $\calC$, let $H^0(\calC)$ denote its homotopy category.
A dg-enhancement of a triangulated category $\calD$ is a dg-category $\calC$ whose homotopy category
$H^0(\calC)$ is equivalent to $\calD$ as a triangulated category. 
All the additive categories we'll encounter are indeed $\bbC$-linear.
By the symbol $\otimes$ we'll mean a tensor product of $\bbC$-linear objects.
If the category is $\bbZ$-graded or $\bbZ/2\bbZ$-graded we'll write $\otimes_\bbZ$ for the tensor product
of $\bbZ$-graded or $\bbZ/2\bbZ$-graded ($\bbC$-linear) objects.

We'll abbreviate KHA, CCA and KCA for K-theoretical Hall algebra,
cohomological critical convolution algebra and K-theoretical critical convolution algebra.

 \smallskip

\textbf{Acknowledgements.} It is a pleasure to thank D. Hernandez and 
H. Nakajima for inspiring discussions concerning this paper.

\section{Critical convolution algebras}

This section contains a reminder on singularity categories and derived factorization categories.
We'll follow \cite{BFK13}, \cite{EP15} (in the non equivariant case), 
\cite{H17a} and \cite{H17b} to which we refer for more details.

\subsection{Singularity categories}\label{sec:SC}

\subsubsection{Definition}\label{sec:defsing}
Let $G$ be an affine group.
Let $Y$ be a quasi-projective $G$-scheme with 
a $G$-equivariant ample line bundle.
The equivariant triangulated category of singularities of $Y$ is the Verdier quotient 
$$\DCoh_G^\sg(Y)=\Db\Coh_G(Y)\,/\,\Perf_G(Y).$$
Given a closed $G$-invariant subset $Z\subset Y$, let 
$\Perf_G(Y)_Z\subset\Perf_G(Y)$ be the full subcategory of perfect complexes
with cohomology sheaves set-theoretically supported in $Z$, and define
\begin{align}\label{sing1}\DCoh_G^\sg(Y)_Z=\Db\Coh_G(Y)_Z\,/\,\Perf_G(Y)_Z.\end{align}
The forgetful functor $\DCoh_G^\sg(Y)_Z\to\DCoh_G^\sg(Y)$
is fully faithful, see, e.g., \cite[lem.~3.1]{EP15}.
This allows us to see $\DCoh_G^\sg(Y)_Z$ as a full triangulated subcategory
of $\DCoh_G^\sg(Y)$.

When dealing with factorizations on singular varieties, 
or for functoriality reasons, one may
need relative categories of singularities. Let us briefly recall this.
Let $i:Y\to X$ be a $G$-invariant closed embedding of finite $G$-flat dimension
of quasi-projective $G$-schemes with $G$-equivariant ample line bundles.
Let
$\Perf_G(Y/X)\subset \Db\Coh_G(Y)$ be the thick subcategory
generated by $Li^*(\Db\Coh_G(X))$ and
$\Perf_G(Y/X)_Z\subset\Perf_G(Y/X)$ be the full subcategory of complexes
with cohomology sheaves set-theoretically supported in $Z$.
Following \cite{EP15}, 
we define the equivariant triangulated category of singularities of $Y$
relative to $X$ and supported on $Z$ to be the Verdier quotient
\begin{align}\label{sing2}\DCoh_G^\sg(Y/X)_Z=\Db\Coh_G(Y)_Z\,/\,\Perf_G(Y/X)_Z
\end{align}
If $X$ is smooth, then $\Perf_G(Y/X)=\Perf_G(Y)$, hence
\begin{align}\label{sgXsm}\DCoh_G^\sg(Y/X)=\DCoh_G^\sg(Y).\end{align}

\subsubsection{Functoriality}
Let $X_1,$ $X_2$ be quasi-projective schemes with actions of an affine group $G$ and
$G$-equivariant ample line bundles.
Let $i_1:Y_1\to X_1$ be a $G$-invariant closed embedding of finite $G$-flat dimension.
Let $\phi:X_2\to X_1$ be a $G$-invariant morphism and
$Y_2$ be the fiber product $Y_2=Y_1\times_{X_1}X_2$.
We have an obvious closed embedding $i_2:Y_2\to X_2$.
Assume that the morphisms $\phi$, $i_2$ and the restriction of $\phi$ to a morphism $Y_2\to Y_1$ 
have a finite $G$-flat dimension.
Let $Z_1,$ $Z_2$ be closed $G$-invariant subsets of $X_1,$ $X_2$.
By \cite[\S 3.4]{EP15} and \cite[\S 3.2]{H17b} the following hold.

Assume that $\phi^{-1}(Z_1)\subset Z_2$.
The pull-back functor 
$L\phi^*:\Db\Coh_G(Y_1)_{Z_1}\to\Db\Coh_G(Y_2)_{Z_2}$
yields a triangulated functor
\begin{align}\label{funct1}
L\phi^*:\DCoh_G^\sg(Y_1/X_1)_{Z_1}\to\DCoh_G^\sg(Y_2/X_2)_{Z_2}.
\end{align}
Assume that $\phi(Z_2)\subset Z_1$ and that the restriction $\phi|_{Z_2}$ is proper.
The pushforward functor 
$R\phi_*:\Db\Coh_G(Y_2)_{Z_2}\to\Db\Coh_G(Y_1)_{Z_1}$
yields a triangulated functor
\begin{align}\label{funct2}
R\phi_*:\DCoh_G^\sg(Y_2/X_2)_{Z_2}\to\DCoh_G^\sg(Y_1/X_1)_{Z_1}.
\end{align}

\begin{remark}\label{rem:derived1}
For any derived $G$-scheme $Y$ with a closed $G$-invariant subset $Z$
we define the singularity category $\DCoh_G^\sg(Y)_Z$ as in \eqref{sing1}.
The pull-back and pushforward functors are defined similarly.
\end{remark}

\subsection{Derived factorization categories}

\subsubsection{Definition}
A $G$-equivariant LG-model is a triple $(X,\chi,f)$ such that
\hfill
\begin{enumerate}[label=$\mathrm{(\alph*)}$,leftmargin=8mm]
\item $X$ is a quasi-projective scheme with a $G$-equivariant ample line bundle and $G$ is an affine group,
\item $\chi:G\to\bbC^\times$ is a character of $G$ and 
$f:X\to\bbC$ is a $\chi$-semi-invariant regular function on $X$,
\item the critical set of $f$ is contained into its zero locus.
\end{enumerate}

A morphism of $G$-equivariant LG-models $\phi:(X_2,\chi,f_2)\to (X_1,\chi,f_1)$ is a $G$-invariant morphism
$\phi:X_2\to X_1$ such that $f_2=\phi^*f_1$.
We'll say that the $G$-equivariant LG-model $(X,\chi,f)$ is smooth if $X$ is smooth.
If $\chi=1$ we'll say that $(X,f)$ is a $G$-invariant LG-model, and if $G=\{1\}$ that $(X,f)$ is an LG-model.

Let $\Coh_G(X,f)$ be the dg-category of all $G$-equivariant coherent factorizations of $f$ on $X$.
An object of $\Coh_G(X,f)$ is called a factorization. It is a sequence
$$\calE=(\xymatrix{\calE_1\ar[r]^-{\phi_1}&\calE_0\ar[r]^-{\phi_0}&\calE_1\otimes\chi})$$
where $\calE_0,\calE_1\in\Coh_G(X)$ and $\phi_0,\phi_1$ are $G$-invariant homomorphisms such that
$\phi_0\circ\phi_1=f\cdot\id_{\calE_1}$ and
$(\phi_1\otimes\chi)\circ\phi_0=f\cdot\id_{\calE_0}.$
The $G$-equivariant coherent sheaves $\calE_0$ and $\calE_1$ are the components of $\calE$, and
the maps $\phi_0$, $\phi_1$ are its differentials.

The homotopy category of $\Coh_G(X,f)$ is a triangulated category. 
The category of acyclic objects is
the thick subcategory  of $H^0(\Coh_G(X,f))$
generated by the totalization of the exact triangles.
The derived factorization category is the Verdier quotient
\begin{align}\label{DFACT}
\DCoh_G(X,f)=H^0(\Coh_G(X,f))\,/\,\Acyclic.
\end{align}

Let $Z\subset X$ be a closed $G$-invariant subset. 
A factorization
in $\Coh_G(X,f)$ is set-theoretically supported on $Z$ if its components
are set-theoretically supported on $Z$. Let  $\Coh_G(X,f)_Z\subset\Coh_G(X,f)$
be the full dg-subcategory  of all factorizations set-theoretically supported on $Z$.
Let $\DCoh_G(X,f)_Z$ be the Verdier quotient of the homotopy category of $\Coh_G(X,f)_Z$ 
by the thick subcategory of acyclic objects.
Forgetting the support yields a triangulated functor
$\DCoh_G(X,f)_Z\to\DCoh_G(X,f).$ 
This functor is fully faithful
and allows us to view $\DCoh_G(X,f)_Z$ as a full triangulated subcategory
of $\DCoh_G(X,f)$. See, e.g., \cite[\S 3.1]{EP15}, \cite[\S 2.4]{H17b}.

To define derived functors of derived factorization categories
we may need injective or locally free $G$-equivariant factorizations or
$G$-equivariant factorizations of finite $G$-flat dimensions.
There are defined similarly as above, with the components beeing injective quasi-coherent sheaves,
or coherent locally free sheaves, or coherent sheaves of finite $G$-flat dimension,
see, e.g., \cite[\S 1]{EP15}, \cite[\S 2.1]{H17b}.

\begin{remark} 
We'll use factorizations over a smooth quasi-projective $G$-scheme $X$.
If $X$ is affine, then the quotient by acyclic objects in \eqref{DFACT}
can be omitted, see, e.g., \cite[prop.~3.4]{BFK13}.
\end{remark}

\subsubsection{Functoriality and tensor product}
Let $\phi:(X_2,\chi,f_2)\to (X_1,\chi,f_1)$ be a morphism of $G$-equivariant LG-models.
Let $Z_1,$ $Z_2$ be closed $G$-invariant subsets of $X_1,$ $X_2$.

Assume that $\phi^{-1}(Z_1)\subset Z_2$.
Then, we have a pull-back dg-functor 
$\phi^*:\Coh_G(X_1,f_1)_{Z_1}\to\Coh_G(X_2,f_2)_{Z_2}$ which takes a factorization $\calE$ to
the factorization  $\phi^*\calE$ with the components
$\phi^*\calE_0$,  $\phi^*\calE_1$ and the differentials
$\phi^*d_0$, $\phi^*d_1$.
Assume further that the map $\phi$ has finite $G$-flat dimension.
By \cite[\S 3.6]{EP15} and \cite[\S 2.3.1]{H17b}, deriving the functor $\phi^*$
with $G$-equivariant factorizations of finite $G$-flat dimension yields a triangulated functor
\begin{align}\label{funct1}
L\phi^*:\DCoh_G(X_1,f_1)_{Z_1}\to\DCoh_G(X_2,f_2)_{Z_2}.
\end{align}

Assume that $\phi(Z_2)\subset Z_1$ and that the restriction $\phi|_{Z_2}$ is proper.
Then, we have a pushforward dg-functor 
$\phi_*:\Coh_G(X_2,f_2)_{Z_2}\to\Coh_G(X_1,f_1)_{Z_1}$ which takes a factorization $\calE$ to
the factorization  $\phi_*\calE$ with the components
$\phi_*\calE_0$,  $\phi_*\calE_1$ and the differentials
$\phi_*d_0$, $\phi_*d_1$.
By \cite[lem.~3.5]{EP15} and \cite[\S 2.3.1]{H17b}, 
deriving this functor with injective $G$-equivariant factorizations yields a triangulated functor
\begin{align}\label{funct3}
R\phi_*:\DCoh_G(X_2,f_2)_{Z_2}\to\DCoh_G(X_1,f_1)_{Z_1}.
\end{align}

Assume that the map $f_1\boxplus f_2$ on $X_1\times X_2$ is regular.
There is a dg-functor
$$\boxtimes:\Coh_G(X_1,f_1)\otimes\Coh_G(X_2,f_2)\to\Coh_G(X_1\times X_2,f_1\boxplus f_2)$$
which takes the pair of factorizations $(\calE,\calF)$ to the factorization
with components
$$(\calE\boxtimes\calF)_0=(\calE_0\boxtimes\calF_0)\oplus(\calE_1\boxtimes\calF_1)
 ,\quad
(\calE\boxtimes\calF)_1=(\calE_0\boxtimes\calF_1)\oplus(\calE_1\boxtimes\calF_0)$$
and the obvious differentials.
The functor $\boxtimes$ yields a triangulated functor
$$\boxtimes:\DCoh_G(X,f_1)\otimes\DCoh_G(X,f_2)\to\DCoh_G(X_1\times X_2,f_1\oplus f_2).$$
Assume that $X_1=X_2=X$ and $f_1+f_2$ is regular.
There is a dg-functor
$$\otimes:\Coh_G(X,f_1)\otimes\Coh_G(X,f_2)\to\Coh_G(X,f_1+f_2)$$
which takes the pair of factorizations $(\calE,\calF)$ to the factorization
with components
$$(\calE\otimes\calF)_0=(\calE_0\otimes\calF_0)\oplus(\calE_1\otimes\calF_1)
 ,\quad
(\calE\otimes\calF)_1=(\calE_0\otimes\calF_1)\oplus(\calE_1\otimes\calF_0)$$
and the obvious differentials. 
Assume further that $X$ is smooth.
The class in $\DCoh_G(X,f_1)$ of any $G$-equivariant factorization 
can be represented by a locally free one by \cite[prop.~3.14]{BFK13}.
Hence, deriving the functor $\otimes$ we get a triangulated functor,
see \cite[\S 2.3.2]{H17b},
$$\otimes^L:\DCoh_G(X,f_1)\otimes\DCoh_G(X,f_2)\to\DCoh_G(X,f_1+f_2).$$

\begin{remark}\label{rem:base change 1} 
\hfill
\begin{enumerate}[label=$\mathrm{(\alph*)}$,leftmargin=8mm]
\item
The derived pushforward and pulback satisfy the projection formula
and the flat base change property, see \cite[prop.~4.32, lem.~4.34]{H17a}.
\item
The triangulated categories
$\Db\Coh_G(Y)_Z$,  $\Perf_G(Y)_Z$, $\DCoh^\sg_G(Y)_Z$, $\DCoh_G(X,f)_Z$
admit compatible dg-enhancements that we will use when needed.
For instance, as a dg-category $\DCoh^\sg_G(Y)_Z$ 
is the Drinfeld quotient of the dg-category $\Db\Coh_G(Y)_Z$
by the dg-subcategory $\Perf_G(Y)_Z$. 
See, e.g., \cite[\S 5]{BFK13} or \cite[\S 1.1]{PS21}.
All derived functors above, between equivariant triangulated category of singularities or
derived factorization categories, admit dg-enhancements.
\end{enumerate}
\end{remark}

\subsubsection{Comparison with singularity categories}\label{sec:fact-sing}
Let $(X,\chi,f)$ be a $G$-equivariant LG-model.
Let $Y$ be the zero locus of $f$,
$i$ be the closed embedding $Y\subset X$,
and $Z\subset Y$ be a closed $G$-invariant subset.
We have a triangulated functor
\begin{align}\label{Upsilon1}\Upsilon:\Db\Coh_G(Y)_Z\to \DCoh_G(X,f)_Z\end{align}
taking a complex $(\calE^\bullet,d)$ to
\begin{align*}
\xymatrix{
{\displaystyle\bigoplus_{m\in\bbZ}i_*\calE^{2m-1}\otimes\chi^{-m}}\ar[r]^-{d}&
{\displaystyle\bigoplus_{m\in\bbZ}i_*\calE^{2m}\otimes\chi^{-m}}\ar[r]^-{d}&
{\displaystyle\bigoplus_{m\in\bbZ}i_*\calE^{2m-1}\otimes\chi^{1-m}}}
\end{align*}
The functor $\Upsilon$ annihilates the image of $Li^*$, 
yielding a commutative triangle
\begin{align}\label{triangle0}
\begin{split}
\xymatrix{\DCoh_G^\sg(Y/X)_Z\ar[r]^-{\Gamma}& \DCoh_G(X,f)_Z\\
\Db\Coh_G(Y)_Z\ar[u]\ar[ur]_-\Upsilon&}
\end{split}
\end{align}
The functor $\Gamma$ is an equivalence of triangulated categories by
\cite[thm.~3.6]{H17b}.
The functor $\Upsilon$ has the following functoriality properties.

\begin{lemma}\label{lem:Upsilon5}
Let $\phi:(X_2,\chi,f_2)\to (X_1,\chi,f_1)$ be a morphism of smooth $G$-equivariant LG-models.
Assume that $\phi$ is of finite $G$-flat dimension. 
\hfill
\begin{enumerate}[label=$\mathrm{(\alph*)}$,leftmargin=8mm]
\item
There is an isomorphism of functors
\begin{align*}
L\phi^*\circ\Upsilon=\Upsilon\circ L\phi^*:\Db\Coh_G(Y_1)\to\DCoh_G(X_2,f_2).
\end{align*}
\item
Assume that the map $\phi$ is proper.
There is an isomorphism of functors
\begin{align*}
R\phi_*\circ\Upsilon=\Upsilon\circ R\phi_*:\Db\Coh_G(Y_2)\to\DCoh_G(X_1,f_1).
\end{align*}

\end{enumerate}
\end{lemma}

\begin{proof}
The functor $\Upsilon$ can be described in the following way.
To simplify assume that $Z=Y$.
Let $\chi_1:\bbC^\times\to\bbC^\times$ be the linear character.
We consider the $G'$-equivariant LG-model $(X',\chi',f')$
such that 
$$X'=X\times\bbC
 ,\quad
G'=G\times\bbC^\times
 ,\quad
\chi'=\chi\boxtimes\chi_1
 ,\quad
f'(x,z)=f(x)z.$$
Let $\pi:X'\to X$ be the projection $(x,z)\mapsto x$ and
$j:X\to X'$ be the embedding $x\mapsto(x,1)$.
By \cite{I12}, the functor
$$\Coh_G(Y)\to \Coh_{G'}(X',f')
 ,\quad
\calE\mapsto 
\big(0\to\pi^*i_*\calE\to0\big)$$
extends to an equivalence of triangulated categories
\begin{align}\label{Koszul}\Phi:\Db\Coh_G(Y)\to\DCoh_{G'}(X',f').\end{align} 
See \cite[thm.~2.3.11]{BFK19} or \cite[thm.~3.3.3]{T19} for a formulation closer to our setting.
Composing the derived pull-back with the forgetful functor, we get a functor
\begin{align}\label{functor2}
Lj^*:\DCoh_{G'}(X',f')\to\DCoh_G(X,f).
\end{align}
Then, we have 
\begin{align}\label{Upsilon4}
\Upsilon=Lj^*\circ\Phi.
\end{align}
The morphism $\phi$ restricts to a morphism $Y_2\to Y_1$. Further, it lifts to a morphism
$(X'_2,\chi',f'_2)\to (X'_1,\chi',f'_1)$. 
Both are denoted by the symbol $\phi$.

We first consider the following diagram of functors
\begin{align*}
\xymatrix{
\DCoh_G(X_1,f_1)\ar[r]^-{L\phi^*}&\DCoh_G(X_2,f_2)\\
\DCoh_{G'}(X'_1,f'_1)\ar[r]^-{L\phi^*}\ar[u]^-{Lj^*}&\DCoh_{G'}(X'_2,f'_2)\ar[u]_-{Lj^*}\\
\Db\Coh_G(Y_1)\ar[r]^-{L\phi^*}\ar[u]^-\Phi&\Db\Coh_G(Y_2)\ar[u]_-\Phi.
}
\end{align*}
To prove (a) we are reduced to give an isomorphism of functors
\begin{align}\label{isom4}
L\phi^*\circ\Phi=\Phi\circ L\phi^*:\Db\Coh_G(Y_1)\to\DCoh_{G'}(X'_2,f'_2),
\end{align}
because \eqref{Upsilon4} yields
$$\Upsilon\circ L\phi^*=Lj^*\circ\Phi\circ L\phi^*=
Lj^*\circ L\phi^*\circ\Phi=L\phi^*\circ Lj^*\circ\Phi=L\phi^*\circ\Upsilon.$$
The isomorphism \eqref{isom4} follows from the contravariant
functoriality properties of the Koszul equivalence $\Phi$
proved in \cite[lem.~2.4.7]{T19}.

Next, assume that the map $\phi$ is proper.
Consider the diagram of functors
\begin{align*}
\xymatrix{
\DCoh_G(X_2,f_2)\ar[r]^{R\phi_*}&\DCoh_G(X_1,f_1)\\
\DCoh_{G'}(X'_2,f'_2)\ar[r]^{R\phi_*}\ar[u]^-{Lj^*}&\DCoh_{G'}(X'_1,f'_1)\ar[u]_-{Lj^*}\\
\Db\Coh_G(Y_2)\ar[r]^{R\phi_*}\ar[u]^-\Phi&\Db\Coh_G(Y_1).\ar[u]_-\Phi
}
\end{align*}
To prove (b) we are reduced to give an isomorphism of functors
\begin{align}\label{isom5}
R\phi_*\circ\Phi=\Phi\circ R\phi_*:\Db\Coh_G(X_2)\to\DCoh_{G'}(X'_1,f'_1).
\end{align}
because \eqref{Upsilon4} and base change yield
$$\Upsilon\circ R\phi_*=Lj^*\circ\Phi\circ R\phi_*=Lj^*\circ R\phi_*\circ\Phi=
R\phi_*\circ Lj^*\circ\Phi=R\phi_*\circ\Upsilon.$$
The isomorphism \eqref{isom5} follows from the functoriality properties of the Koszul equivalence $\Phi$
proved in \cite[lem.~2.4.4]{T19}.
\end{proof}

\begin{remark}\label{rem:derived}
By definition, for any LG-model $(X,\chi,f)$ the function $f$ is regular.
Hence, the closed embedding $j:Y\to RY$ 
into the derived zero locus
$RY=R(X\times\bbC\to X\,,f)$ 
is a quasi-isomorphism and
the functors $Rj_*$ and $Lj^*$
are mutually inverses equivalences of categories 
$\Db\Coh_G(Y)_Z=\Db\Coh_G(RY)_Z.$
One may need triples $(X,\chi,f)$ such that the function $f$ is not regular. 
According to Remark \ref{rem:derived1}, 
we define the singularity category of  the derived scheme $RY$
to be
\begin{align*}\DCoh_G^\sg(RY)_Z=\Db\Coh_G(RY)_Z\,/\,\Perf_G(RY)_Z.\end{align*}
Note that, if $X$ is smooth, then the derived scheme $RY$ is quasi-smooth.
In particular, if $f=0$ then the category of singularities of 
the zero locus of $f$ is $\DCoh_G^\sg(RY)_Z$ and, taking the $K$-theory, we get
$K_G(X,0)_Z=K_G(Z)$.
\end{remark}

\subsection{K-theoretical critical convolution algebras}\label{sec:critalg1}

Fix a $G$-equivariant LG-model $(X,\chi,f)$.
Let $Y\subset X$ be the zero locus of $f$, 
$i$ be the closed embedding $Y\to X$, and
$Z\subset Y$ a closed $G$-invariant subset.
We'll use both the algebraic and topological equivariant K-theory.
See \S\ref{sec:Ktheory} for a reminder on K-theory.
The critical (algebraic) K-theory group is
$$K_G(X,f)_Z=K_0(\DCoh_G(X,f)_Z).$$

\subsubsection{First properties of the critical K-theory}
Assume that $(X,\chi,f)$ is smooth.
The functor $\Upsilon$ in 
\eqref{Upsilon1} yields a map
\begin{align}\label{Upsilon2}
\Upsilon:\xymatrix{K^G(Z)\ar[r]& K_G(X,f)_Z
.}
\end{align}

\begin{proposition}\label{prop:Upsilon3} 
The map $\Upsilon$ is surjective.
\end{proposition}

\begin{proof}
We must check that the obvious functor
$\Db\Coh_G(Y)_Z\to\DCoh_G^\sg(Y/X)_Z$
yields a surjective morphism of Grothendieck groups.
By \eqref{sing2} this surjectivity follows from \cite[prop.~VIII.3.1]{SGA5}.
\end{proof}

\begin{proposition}\label{prop:Thom}
Let $\rho:V\to X$ be a $G$-equivariant vector bundle.
The pull-back yields an isomorphism
$L\rho^*:K_G(X,f)_Z\to K_G(V,f\circ\rho)_{\rho^{-1}(Z)}$.
\end{proposition}

\begin{proof}
The map $L\rho^*$ is well defined because $\rho$ is flat.
Let $i:X\to V$ be the zero section. 
The map $Li^*$ is well defined because $i$ is of finite $G$-flat dimension.
The composed map $Li^*\circ L\rho^*$ is an isomorphism, hence $L\rho^*$ is injective.
Let $U=\rho^{-1}(Z)$. 
The square
\begin{align*}
\xymatrix{
K^G(Z)\ar[r]^-\Upsilon\ar[d]_-{L\rho^*}&K_G(X,f)_Z\ar[d]^-{L\rho^*}\\
K^G(U)\ar[r]^-\Upsilon&K_G(V,f\circ\rho)_U}
\end{align*}
is commutative by Lemma \ref{lem:Upsilon5}.
Thus the surjectivity of $L\rho^*$ follows from the Thom isomorphism
and the surjectivity of $\Upsilon$ proved in Proposition \ref{prop:Upsilon3}.
\end{proof}


\begin{proposition}\label{prop:Thomason} 
Let $j:X^G\to X$ be the inclusion of the fixed points locus.
Assume that $G$ is a torus and that the function $f\circ j$ on $X^G$ is regular.
Then $Rj_*$ and $Lj^*$
are isomorphisms between the $F_G$-vector spaces
$K_G(X^G,f\circ j)_{Z^G}\otimes_{R_G}F_G$ and $K_G(X,f)_Z\otimes_{R_G}F_G$.
The composed map $Lj^*\circ Rj_*$ is the tensor product with the class 
$\Lambda_{-1}(T^*_{X^G}X)$.
\end{proposition}

\begin{proof} By Lemma \ref{lem:Upsilon5} we have the commutative diagram
\begin{align*}
\xymatrix{
K^G(Z^G)\otimes_{R_G}F_G\ar[r]^-\Upsilon\ar[d]_-{Rj_*}&
K_G(X^G,f\circ j)_{Z^G}\otimes_{R_G}F_G\ar[d]^-{Rj_*}\\
K^G(Z)\otimes_{R_G}F_G\ar[r]^-\Upsilon\ar[d]_-{Lj^*}&K_G(X,f)_Z\otimes_{R_G}F_G\ar[d]^-{Lj^*}\\
K^G(Z^G)\otimes_{R_G}F_G\ar[r]^-\Upsilon&K_G(X^G,f\circ j)_{Z^G}\otimes_{R_G}F_G}
\end{align*}
The composed map $Lj^*\circ Rj_*$ is the tensor product with the class 
$\Lambda_{-1}(T^*_{X^G}X)$ in $K_G(X^G,f\circ j)_{Z^G}\otimes_{R_G}F_G$, 
because it is so in $K^G(Z^G)\otimes_{R_G}F_G$ and the map $\Upsilon$ is surjective by
Proposition \ref{prop:Upsilon3}. In particular, the map $Rj_*$ is injective
on  $K_G(X^G,f\circ j)_{Z^G}\otimes_{R_G}F_G$. 
It is also surjective, because the upper square commutes, $\Upsilon$ is surjective, and
$Rj_*$ is surjective onto $K^G(Z)\otimes_{R_G}F_G$.
\end{proof}

\subsubsection{Critical convolution algebras}\label{sec:Kcrit}
Let $(X_a,\chi,f_a)$, $a=1,2,3$,  be smooth $G$-equivariant LG-models.
We abbreviate $X_{123}=X_1\times X_2\times X_3$ and
$X_{ab}=X_a\times X_b$ for $a<b$.
Let $\pi_{ab}:X_{123}\to X_{ab}$ be the obvious projection.
Let $f_{ab}=f_a\boxplus(-f_b)$, $Y_a=f_a^{-1}(0)$ and $Y_{ab}=f_{ab}^{-1}(0)$.
Let $Z_{ab}\subset Y_{ab}$ be a $G$-invariant closed subset for the diagonal $G$-action.
We define 
$\widetilde Z_{13}=\pi_{12}^{-1}(Z_{12})\cap\pi_{23}^{-1}(Z_{23}).$
We'll assume that the function $f_{ab}$ is regular for each $a<b$, 
and that the map $\pi_{13}|_{\widetilde Z_{13}}$
is proper and maps into  $Z_{13}$.
Then, there is a convolution functor 
\begin{align}\label{conv1}
\Db\Coh_G(X_{12})_{Z_{12}}\otimes\Db\Coh_G(X_{23})_{Z_{23}}\to
\Db\Coh_G(X_{13})_{Z_{13}}
\end{align}
such that 
$\calE\star\calF=R(\pi_{13})_*(L(\pi_{12})^*(\calE)\otimes^LL(\pi_{23})^*(\calF)).$
In a similar way,  since we have
$(\pi_{12}\times\pi_{23})^*(f_{12}\oplus f_{23})=(\pi_{13})^*f_{13},$
we can define a convolution functor of derived factorization categories 
\begin{align}\label{conv2}
\DCoh_G(X_{12},f_{12})_{Z_{12}}\otimes\DCoh_G(X_{23},f_{23})_{Z_{23}}\to
\DCoh_G(X_{13},f_{13})_{Z_{13}}
\end{align}
such that 
$\calE\star\calF=R(\pi_{13})_*(L(\pi_{12})^*(\calE)\otimes^LL(\pi_{23})^*(\calF))$.
This functor is compatible with both the triangulated structures and their dg-enhancements.

Now, we consider the following particular case.
Let $(X,\chi,f)$ be a smooth $G$-equivariant LG-model with
a proper $G$-equivariant map $\pi:X\to X_0$ to an affine $G$-scheme.
Let $f=f_0\circ\pi$ where $f_0:X_0\to\bbC$ is a $\chi$-semi-invariant function.
Let $Y$ and $Y_0$ be the zero loci of the maps $f$ and $f_0$.
We define 
$$Z=X\times_{X_0} X
 ,\quad
L=X\times_{X_0}\{x_0\}
 ,\quad
x_0\in(Y_0)^G.
$$
We set $X_a=X$ for each $a$, and $Z_{ab}=Z$ for each $a<b$.
The convolution functor \eqref{conv1} yields a monoidal structure on the triangulated category
$\Db\Coh_G(X^2)_Z$ such that $\Db\Coh_G(X)_L$ and $\Db\Coh_G(X)$ are modules 
over $\Db\Coh_G(X^2)_Z$.
Taking the Grothendieck groups, this yields an associative $R_G$-algebra structure on
$K^G(X^2)_Z=K^G(Z)$ and $K^G(Z)$-representations in $K^G(L)$ and $K^G(X)$.
Now, we set $f_a=f$ for each $a$, $f_{ab}=f^{(2)}$ for each $a<b$, 
and we consider the factorization categories.
Note that $Z\subset Y_{ab}$.
From \eqref{conv2} we get the following.

\begin{proposition}
\hfill
\begin{enumerate}[label=$\mathrm{(\alph*)}$,leftmargin=8mm]
\item 
$\DCoh_G(X^2,f^{(2)})_Z$
is a monoidal category.
\item
$\DCoh_G(X,f)_L$ and $\DCoh_G(X,f)$  are modules over  
$\DCoh_G(X^2,f^{(2)})_Z$.
\end{enumerate}
\end{proposition}

\begin{proof} 
To prove Part (a) we must define an associativity constraint and a unit 
satisfying the pentagon and the unit axioms. 
The associativity constraint follows from the flat base change and the projection formula as in
\cite[prop.~5.13]{BDFIK16}.
The unit is the factorization $\Upsilon\Delta_*\calO_X$.
Note that $\Delta(X)\subset Z$ by hypothesis.
To prove Part (b) we choose $X_1=X_2=X$, $X_3=\{o\}$, $f_1=f_2=f$,
$f_3=0$, $Z_{12}=Z$, and $Z_{23}=Z_{13}=L\times\{o\}$ or $Z_{23}=Z_{13}=X\times\{o\}$
and we apply \eqref{conv2}.
\end{proof}

\begin{corollary}\label{cor:critalg1}
\hfill
\begin{enumerate}[label=$\mathrm{(\alph*)}$,leftmargin=8mm]
\item
$K_G(X^2,f^{(2)})_Z$ is an $R_G$-algebra which acts on
$K_G(X,f)_L$ and $K_G(X,f)$.

\item
The map $\Upsilon$ yields a surjective algebra map 
$K^G(Z)\to K_G(X^2,f^{(2)})_Z$.
\end{enumerate}
\end{corollary}

\begin{proof}
The convolution functors 
\begin{align}\label{star1}
\begin{split}
&\star:\Db\Coh_G(X_{12},f_{12})_{Z_{12}}\otimes\Db\Coh_G(X_{23},f_{23})_{Z_{23}}\to\Db\Coh_G(X_{13},f_{13})_{Z_{13}}\\
&\star:\Db\Coh_G(X_{12})_{Z_{12}}\times\Db\Coh_G(X_{23})_{Z_{23}}\to\Db\Coh_G(X_{13})_{Z_{13}}
\end{split}
\end{align}
are both given by
\begin{align}\label{star}
\calE\star\calF=R(\pi_{13})_*L(\pi_{12}\times\pi_{23})^*(\calE\boxtimes\calF).
\end{align}
We must compare the functors \eqref{star1}.
To do this, we first consider the derived scheme 
$RY_{ab}=R(X_{ab}\times\bbC\to X_{ab}\,,\,f_{ab}).$
We have the following obvious embeddings of derived schemes
$$\xymatrix{Y_{ab}\ar[r]^-j&RY_{ab}\ar[r]^-i& X_{ab}}.$$ 
We consider the following commutative diagram of derived schemes
\begin{align*}
\xymatrix{
X_{12}\times X_{23}&&\ar[ll]_-{\pi_{12}\times \pi_{23}} X_{123}\ar[r]^-{\pi_{13}}&X_{13}\\
RY_{12}\times RY_{23}\ar[u]^-i&&\ar[ll]_-{\pi_{12}\times \pi_{23}} RY_{123}\ar[r]^-{\pi_{13}}\ar[u]_-i&RY_{13}\ar[u]_-i
}
\end{align*}
The left square is Cartesian.
The upper left horizontal map has finite $G$-flat dimension because $X_1$, $X_2$, $X_3$ are smooth. 
The lower one either because it is quasi-smooth, see \cite[lem.~1.15]{K22}.
Thus, we have a convolution functor
\begin{align}\label{star2}
\star:\Db\Coh_G(RY_{12})_{Z_{12}}\times\Db\Coh_G(RY_{23})_{Z_{23}}\to\Db\Coh_G(RY_{13})_{Z_{13}}
\end{align}
given by the formula \eqref{star}.
The left square is Cartesian. 
The base change
$$L(\pi_{12}\times\pi_{23})^*\circ Ri_*\to Ri_*\circ L(\pi_{12}\times\pi_{23})^*$$
is invertible by \cite[cor.~3.4.2.2]{L18}. 
Hence the direct image 
\begin{align*}
Ri_*:\Db\Coh_G(RY_{ab})_{Z_{ab}}\to\Db\Coh_G(X_{ab})_{Z_{ab}}
\end{align*}
intertwines the convolution functors \eqref{star2} and \eqref{star1}.
The morphism $j$ is a quasi-isomorphism
because the function $f_{ab}$ is regular.
Hence, the pushforward and pull-back functors $Rj_*$ and $Lj^*$
are mutually inverse equivalences of categories 
$$\Db\Coh_G(Y_{ab})_{Z_{ab}}=\Db\Coh_G(RY_{ab})_{Z_{ab}}$$
Hence \eqref{star2} yields a convolution functor
\begin{align}\label{star3}
\star:\Db\Coh_G(Y_{12})_{Z_{12}}\times\Db\Coh_G(Y_{23})_{Z_{23}}\to\Db\Coh_G(Y_{13})_{Z_{13}}.
\end{align}
such that the direct image 
\begin{align*}
Rj_*:\Db\Coh_G(Y_{ab})_{Z_{ab}}\to\Db\Coh_G(RY_{ab})_{Z_{ab}}
\end{align*}
intertwines the convolution functors \eqref{star3} and \eqref{star2}.
Since the K-theory satisfies the equivariant d\'evissage, the functor 
\begin{align}\label{functor1}
Ri_*\circ Rj_*:\Db\Coh_G(Y_{ab})_{Z_{ab}}\to\Db\Coh_G(X_{ab})_{Z_{ab}}
\end{align}
yields an isomorphism of Grothendieck groups.
Both Grothendieck groups are canonically identified with $K^G(Z_{ab})$, 
so that \eqref{functor1} induces the identity map of $K^G(Z_{ab})$.
Now, we consider the functor
\begin{align*}
\Upsilon&:\Db\Coh_G(Y_{ab})_{Z_{ab}}\to\Db\Coh_G(X_{ab},f_{ab})_{Z_{ab}}.
\end{align*}
By Lemma \ref{lem:Upsilon5}, it intertwines the functors \eqref{star3} and \eqref{star1}.
It gives a map
\begin{align}\label{map2}
\Upsilon:K^G(Z_{ab})\to K^G(X_{ab},f_{ab})_{Z_{ab}}.
\end{align}
which intertwines the convolution products on both sides.
The surjectivity in Part (b) follows from Proposition \ref{prop:Upsilon3}.
\end{proof}

\subsection{Cohomological critical convolution algebras}\label{sec:critalg2}
\subsubsection{Vanishing cycles and LG-models}
Let $G$ be an affine group acting on a smooth manifold $X$.
Let $\D^\b_G(X)$ be the $G$-equivariant derived category of
constructible complexes with complex coefficients on $X$.
Given a function $f:X\to\bbC$ with zero locus $Y=f^{-1}(0)$, 
we have the vanishing cycle and nearby cycle functors 
$\phi_f,\psi_f:\D^\b_G(X)\to \D^\b_G(Y).$ 
Let $i:Y\to X$ be the obvious embedding. 
Set
$\phi^p_{\!f}=i_*\phi_f[-1]$ and
$\psi^p_{\!f}=i_*\psi_f[-1].$
The functors $\phi^p_{\!f}$, $\psi^p_{\!f}$ commute with the Verdier duality $\bbD$.
They take perverse sheaves to perverse sheaves.
We have a distinguished triangle
\begin{align}\label{triangle}
\xymatrix{\psi^p_{\!f}\calE\ar[r]^{\can}&\phi^p_{\!f}\calE\ar[r]&i_*i^*\calE\ar[r]^{+1}& }
\end{align}
Let $(X,f)$ be a smooth $G$-invariant LG-model.
Let $i:Y\to X$ be the embedding of the zero locus of $f$,
and $j:Z\to X$ the embedding of a closed $G$-invariant subset of $Y$.
For any constructible complex $\calE\in\D^\b_G(X)$ we set
$H^\bullet_{Z}(X,\calE)=H^\bullet_G(Z,j^!\calE).$
Let 
$\calC_X=\bbC_X[\dim X]$
and
$$H^\bullet_G(X,f)_Z=
H^\bullet_{Z}(X,\phi^p_{\!f}\calC_X).$$
Let $\phi:(X_2,f_2)\to (X_1,f_1)$ be a morphism of smooth $G$-invariant LG-models.
Let $Y_1=(f_1)^{-1}(0)$ and $Y_2=(f_2)^{-1}(0)$.
Let $Z_1,$ $Z_2$ be closed $G$-invariant subsets of $Y_1,$ $Y_2$.
By \cite[\S 2.17]{D17} we have the following functoriality maps.
If $\phi^{-1}(Z_1)\subset Z_2$ then we have a pull-back map
$\phi^*:H^\bullet_G(X_1,f_1)_{Z_1}\to H^\bullet_G(X_2,f_2)_{Z_2}$ which
is an isomorphism if $\phi$ is an affine fibration.
If $\phi(Z_2)\subset Z_1$ and $\phi|_{Z_2}$ is proper
then we have a push-forward map
$\phi_*:H^\bullet_G(X_2,f_2)_{Z_2}\to H^\bullet_G(X_1,f_1)_{Z_1}.$

\subsubsection{Cohomological critical convolution algebras}

Let $(X_a,f_a)$ be a smooth $G$-invariant LG-model for $a=1,2,3$.
We define $X_{ab}$, $Y_{ab}$, $Z_{ab}$, $f_{ab}$, $\pi_{ab}$  as in \S\ref{sec:Kcrit}.
There is a Thom-Sebastiani isomorphism
$$\boxtimes:H^\bullet_G(X_{12},f_{12})_{Z_{12}}\otimes H^\bullet_G(X_{23},f_{23})_{Z_{23}}\to
H^\bullet_G(X_{12}\times X_{23},f_{12}\oplus f_{23})_{\widetilde Z_{13}}.$$
We now define a convolution product in critical cohomology
\begin{align}\label{conv3}
\star:H^\bullet_G(X_{12},f_{12})_{Z_{12}}\otimes H^\bullet_G(X_{23},f_{23})_{Z_{23}}\to
H^\bullet_G(X_{13},f_{13})_{Z_{13}}
\end{align}
such that
$\alpha\otimes\beta\mapsto(\pi_{13})_*(\pi_{12}\times\pi_{23})^*(\alpha\boxtimes\beta).$

We consider the following particular setting where
$\pi:X\to X_0$ is a proper morphism of $G$-schemes with 
$X$ smooth quasi-projective and $X_0$ affine,
$f_0:X_0\to\bbC$ is a invariant function, and $f=f_0\circ\pi$ is regular.
Let $Y$, $Y_0$, $Z$, $L$ and $f^{\boxplus 2}$ be as in \S\ref{sec:Kcrit}.
We set $X_a=X$ and $f_a=f$ for each $a=1,2,3$. 
We equip the $H^\bullet_G$-module
$\Ext^\bullet_{\D^\b_G(X_0)}(\phi^p_{\!f_0}\pi_*\calC_X, \phi^p_{\!f_0}\pi_*\calC_X)$
with the Yoneda product.

\begin{proposition}\label{prop:critalg3}
\hfill
\begin{enumerate}[label=$\mathrm{(\alph*)}$,leftmargin=8mm]
\item 
There is an isomorphism
$H^\bullet_G(X^2,f^{(2)})_Z=
\Ext^\bullet_{\D^\b_G(X_0)}(\phi^p_{\!f_0}\pi_*\calC_X, \phi^p_{\!f_0}\pi_*\calC_X)$
which intertwines the convolution product and the Yoneda product.
\item
The convolution product equips $H^\bullet_G(X^2,f^{(2)})_Z$ with 
an $H^\bullet_G$-algebra structure.
\item 
The $H^\bullet_G$-algebra $H^\bullet_G(X^2,f^{(2)})_Z$ acts on the $H^\bullet_G$-modules 
$H^\bullet_G(X,f)_L$ and $H^\bullet_G(X,f)$.
\end{enumerate}
\end{proposition}

\begin{proof}
Parts (b), (c) follow from (a).
The isomorphism in Part (a) is 
\begin{align*}
H^\bullet_G(X^2,f^{(2)})_Z
&=H^\bullet(Z,j^!\phi^p_{\!f^{(2)}}\calC_{X^2})\\
&=H^\bullet(Z,j^!(\phi^p_{\!f}\calC_X\boxtimes\phi^p_{\!f}\calC_X))\\
&=H^\bullet(Z,j^!(\bbD\phi^p_{\!f}\calC_X\boxtimes\phi^p_{\!f}\calC_X))\\
&=\Ext^\bullet_{\D^\b_G(X_0)}(\pi_*\phi^p_{\!f}\calC_X, \pi_*\phi^p_{\!f}\calC_X)\\
&=\Ext^\bullet_{\D^\b_G(X_0)}(\phi^p_{\!f_0}\pi_*\calC_X, \phi^p_{\!f_0}\pi_*\calC_X)
\end{align*}
where the second isomorphism follows from the Thom-Sebastiani theorem and the inclusion
$\crit(f)\subset f^{-1}(0)$, the third one follows from the self-duality of the complex $\phi^p_{\!f}\calC_X$,
the fourth equality is as in \cite[(8.6.4)]{CG}, and the last one is the commutation of 
proper direct image and vanishing cycles.
The compatibility under the isomorphism in (b) of the convolution product in $H^\bullet_G(X^2,f^{(2)})_Z$
and the Yoneda composition in 
$\Ext^\bullet_{\D^\b_G(X_0)}(\phi^p_{\!f_0}\pi_*\calC_X, \phi^p_{\!f_0}\pi_*\calC_X)$
follows from \cite[\S 8.6.27]{CG}, modulo observing that the convolution product \cite[(8.6.27)]{CG}
is the same as the convolution product \eqref{conv3}.
\end{proof}

The functoriality of $\phi^p_{\!f_0}$ yields the following
analog of the algebra homomorphism 
$\Upsilon:K^G(Z)\to K_G(X^2,f^{(2)})_Z$ in Corollary \ref{cor:critalg1}.

\begin{corollary}\label{cor:Upsilon5}
There is an algebra map 
$\Upsilon :H_\bullet^G(Z,\bbC)\to H^\bullet_G(X^2,f^{(2)})_Z$.
\qed
\end{corollary}

\begin{remark}
If $f_{ab}=0$, then there is an $H^\bullet_G$-module isomorphism
$$H^\bullet_G(X_{ab},f_{ab})_{Z_{ab}}
=H^\bullet_G(Z_{ab},\bbD_{Z_{ab}})[-\dim X_{ab}]
=H^G_{-\bullet}(Z_{ab},\bbC)[-\dim X_{ab}]$$
where $\bbD_{Z_{ab}}$ is the dualizing complex.
Under this isomorphism the convolution product \eqref{conv3} is the same as the
convolution product in equivariant Borel-Moore homology used in \cite[\S 2.7]{CG}.
In particular, if $f=0$ then there is an algebra isomorphism
$H^\bullet_G(X^2,f)_Z=H_\bullet^G(Z,\bbC)$, up to a grading renormalization.
The algebra isomorphism in Proposition \ref{prop:critalg3}(b) is the algebra isomorphism
$H_\bullet^G(Z,\bbC)=\Ext^\bullet_{\D^\b_G(X_0)}(\pi_*\calC_X, \pi_*\calC_X)$
in \cite [thm.~8.6.7]{CG}.
\end{remark}

\section{Quiver varieties and critical convolution algebras}\label{sec:quiver varieties}

\subsection{Basics on quiver varieties}\label{sec:quiver basic}

 \subsubsection{Quiver representations}
Let $Q$ be a finite quiver with sets of vertices and of arrows $Q_0$ and $Q_1$.
Let $s,t:Q_1\to Q_0$ be the source and target.
Let $\alpha^*$ be the arrow opposite to the arrow $\alpha\in Q_1$.
Fix a grading $\deg:Q_1\to\bbZ$.
We'll use the auxiliary sets 
$$Q_1^*=\{\alpha^*\,;\,\alpha\in Q_1\}
,\quad
Q'_0=\{i'\,;\,i\in Q_0\}
,\quad
Q'_1=\{a_i:i\to i'\,;\,i\in Q_0\}
,\quad
\Omega=\{\varepsilon_i:i\to i\,;\,i\in Q_0\}.$$
From $Q$ we construct new quivers as follows :
\begin{itemize}[leftmargin=3mm]
\item[-]
$\overline Q$ is the double quiver :
$\overline Q_0=Q_0$, 
$\overline Q_1=Q_1\cup Q_1^*$,
\item[-] 
$\widetilde Q$ is the triple quiver :
$\widetilde Q_0=Q_0$, $\widetilde Q_1=\overline Q_1\cup\Omega$,
\item[-] 
$Q_f$ is the framed quiver :
$Q_{f,0}=Q_0\sqcup Q'_0$,
$Q_{f,1}=Q_1\sqcup Q'_1$,
\item[-]
$\overline Q_f=\overline{(Q_f)}$ is the framed double quiver,
\item[-] 
$\widetilde Q_f$ is the framed triple quiver :
$\widetilde Q_{f,0}=Q_{f,0}$,
$\widetilde Q_{f,1}=\overline{(Q_f)}_1\cup \Omega$,
\item[-] 
$\widehat Q_f=(\widetilde Q)_f$ is the simply framed triple quiver,

\item[-]
$Q^\bullet$ is the graded quiver :
$Q^\bullet_0=Q_0\times\bbZ$,
$Q^\bullet_1=Q_1\times\bbZ$ with $s(\alpha,k)=(s(\alpha),k)$ and $t(\alpha,k)=(t(\alpha),\deg(\alpha)+k).$
\end{itemize}
We abbreviate $I=Q_0$, $I^\bullet=Q_0^\bullet$ and
$\widetilde Q_f^\bullet=(\widetilde Q_f)^\bullet$,
$\overline Q_f^\bullet=(\overline Q_f)^\bullet.$
Let $\bfC$ and $\bfC^\bullet$ be the categories of finite dimensional $I$-graded and 
$I^\bullet$-graded vector spaces.
For any $V$ in $\bfC$ or $\bfC^\bullet$ we write
$V=\bigoplus_{i\in I}V_i$ or
$V=\bigoplus_{(i,k)\in I^\bullet}V_{i,k}$
respectively.
Let $\delta_i$ and $\delta_{i,k}$ be the Dirac functions at $i$ and $(i,k)$.
The dimension vectors are
$v=\sum_{i\in I}v_i\delta_i$
and
$v=\sum_{(i,k)\in I^\bullet}v_{i,k}\delta_{i,k}$
respectively.
Given $V,W\in\bfC$ the representation varieties of $Q$ and $Q_f$ are
$$\X_Q(V)=\prod_{x\in Q_1}\Hom(V_{s(x)},V_{t(x)})
 ,\quad
\X_{Q_f}(V,W)=\prod_{x\in Q_1}\Hom(V_{s(x)},V_{t(x)})\times
\prod_{i\in Q_0}\Hom(V_i,W_i).$$
A representation of $\widetilde Q_f$  is a tuple 
$x=(x_\alpha,x_a,x_{a^*},x_\varepsilon)$ with
$\alpha\in \overline Q_1$, $a\in Q'_1$ and $\varepsilon\in\Omega$.
We'll abbreviate $h=x_h$ for each arrow $h$ and we write
$x=(\alpha\,,\,a\,,\,a^*\,,\,\varepsilon)$.
We'll abbreviate 
$$\overline\X=\X_{\overline Q_f}=\X_{\overline Q}
 ,\quad
\widetilde\X=\X_{\widetilde Q_f}=\X_{\widetilde Q}
 ,\quad
\widehat\X=\X_{\widehat Q_f}=\X_{\widehat Q}
 ,\quad
\overline\X^\bullet=\X_{\overline Q_f^\bullet}=\X_{\overline Q^\bullet}
 ,\quad
\widetilde\X^\bullet=\X_{\widetilde Q_f^\bullet}=\X_{\widetilde Q^\bullet}.$$
 We define
$G_V=\prod_{i\in I}GL(V_i)$ and
$T=(\bbC^\times)^{Q_1}\times\bbC^\times.$
The representation ring of the torus $T$ is
$R_T=\bbC[t_\alpha^{\pm 1}\,,\,q^{\pm 1}]$ where $\alpha$ runs in $Q_1.$
Let $\frakg_V$ be the Lie algebra of $G_V$ and
$\frakg_V^\nil$ be the set of all nilpotent elements in $\frakg_V$.
We'll abbreviate $G_v=G_{\bbC^v}$ and $\frakg_v=\frakg_{\bbC^v}$.

\subsubsection{Nakajima's quiver varieties}\label{sec:quiver variety}

The group $G_V\times G_W\times T$ acts on $\overline\X(V,W)$ in the following way : 
the groups $G_V$, $G_W$ act by conjugaison, and the torus element $(z_\alpha,z)\in T$ 
takes the representation $x$ to
\begin{align}\label{action1}
(zz_\alpha\alpha,zz_{\alpha^*}\alpha^*,za,za^*\,;\,
\alpha\in Q_1\,,\,i\in I)
 ,\quad
z_{\alpha^*}=z_\alpha^{-1}.
\end{align}
In particular, we have $\xi(z)\cdot x=(z\alpha,za,za^*)$, where $\xi$
is the cocharacter
\begin{align}\label{xi}\xi:\bbC^\times\to T,\quad z\mapsto (1,z).\end{align}
We'll abbreviate $\bbC^\times=\xi(\bbC^\times)$. 
The representation variety $\overline\X(V,W)$ is holomorphic symplectic with an Hamiltonian action of
the groups $G_V$ and $G_W$.
The moment maps are
$\mu_V:\overline\X(V,W)\to\frakg^\vee_V$ and
$\mu_W:\overline\X(V,W)\to\frakg^\vee_W.$
A representation in $\overline\X(V,W)$ is stable if it has no non-zero 
subrepresentations supported on $V$.
Set
\begin{align*}
\overline\X(V,W)_s=\{x\in\overline\X(V,W)\,;\,x\ \text{is\ stable}\}
,\quad
\mu_V^{-1}(0)_s=\overline\X(V,W)_s\cap\mu_V^{-1}(0).
\end{align*}
The Nakajima quiver varieties are the categorical quotients 
$$\frakM(v,W)=\mu_V^{-1}(0)_s/G_V
 ,\quad
\frakM_0(v,W)=\mu_V^{-1}(0)/G_V$$
We have an obvious projective map
$\pi:\frakM(v,W)\to\frakM_0(v,W).$
The $G_W\times T$-variety $\frakM(v,W)$ is smooth, quasi-projective,
connected, holomorphic symplectic with Hamiltonian $G_W$-action.
The map $\mu_W$ descends to a moment map
$\mu_W:\frakM(v,W)\to\frakg^\vee_W.$
The map $\mu_W$ factorizes through the morphism $\pi$. 
More precisely, let $\ux$ denote the orbit of $x$ in the categorical quotient
$\frakM(v,W)$ if $x$ is stable, and in the categorical quotient
$\frakM_0(v,W)$ if the $G_V$-orbit of $x$ is closed. We have
$\mu_W=\mu_0\circ\pi$ where
\begin{align*}
\mu_0:\frakM_0(v,W)\to\frakg^\vee_W
 ,\quad
\mu_0(\ux)=aa^*.
\end{align*}
We'll use the following notations
\begin{gather*}
[\alpha,\alpha^*]=\sum_{\alpha\in Q_1}(\alpha\alpha^*-\alpha^*\alpha)
 ,\quad
a^*a=\sum_{i\in I}a_i^*a_i
 ,\quad
aa^*=\sum_{i\in I}a_ia_i^*\\
\varepsilon\, a^* a=\sum_{i\in I}\varepsilon_ia^*_ia_i
 ,\quad
\varepsilon[\alpha,\alpha^*]=\sum_{i\in I}\sum_{\substack{\alpha\in Q_1\\t(\alpha)=i}}\varepsilon_i\alpha\alpha^*-
\sum_{i\in I}\sum_{\substack{\alpha\in Q_1\\s(\alpha)=i}}\varepsilon_i\alpha^*\alpha\\
[\varepsilon,\alpha]=
\sum_{i,j\in I}\sum_{\substack{\alpha\in \overline Q_1\\\alpha:i\to j}}(\varepsilon_j\alpha-\alpha\varepsilon_i)
\end{gather*}
We have
$\mu_V(x)=[\alpha,\alpha^*]+a^*a$
and
$\mu_W(x)=aa^*.$
Given $\gamma\in\frakg_W$ we'll write
$$[\gamma\oplus\varepsilon,a]=\gamma a-a\varepsilon
 ,\quad
[\gamma\oplus\varepsilon,a^*]=\varepsilon a^*-a^*\gamma$$
We'll also write
\begin{align*}
[\varepsilon,x]=0
&\iff[\varepsilon,\alpha]=a\varepsilon=\varepsilon a^*=0,\\
[\gamma\oplus\varepsilon,x]=0
&\iff[\varepsilon,\alpha]=[\gamma\oplus\varepsilon,a]=[\gamma\oplus\varepsilon,a^*]=0.
\end{align*}

\subsubsection{Graded quiver varieties}\label{sec:graded quiver}
Fix $V=\bigoplus_{(i,k)\in I^\bullet}V_{i,k}$ in $\bfC^\bullet$.
Let $G_V$ and $G_V^0$ be the automorphism groups of $V$
in $\bfC$ and $\bfC^\bullet$ respectively.
We write
\begin{align}\label{gdegree}\frakg_V=\bigoplus_{l\in\bbZ}\frakg^l_V
 ,\quad
\frakg^l_V=\bigoplus_{(i,k)\in I^\bullet}\Hom(V_{i,k},V_{i,k+l}).\end{align}
The Lie algebra of $G_V$, $G_V^0$ are $\frakg_V$, $\frakg_V^0$.
For any $v\in\bbN I^\bullet$ let
$\frakg_v$, $\frakg^0_v$, $G_v$, $G_v^0$ be the Lie algebras and groups
associated with the object $\bbC^v\in\bfC^\bullet$.
We consider the graded quiver $\overline Q_f^\bullet$ 
associated with the grading
\begin{align}\label{degree1}
\begin{split}
\deg:\overline Q_{f,1}\to\bbZ
 ,\quad
\alpha_{ij},a_i,a_i^*\mapsto -1.
\end{split}
\end{align}
A representation of $\overline Q^\bullet_f$ is a tuple
$x=(\alpha_k\,,\,a_{i,k}\,,\,a_{i,k}^*)$ with
$\alpha\in Q_1$,
$i\in I$,
$k\in\bbZ$
where
$\alpha_{k}=(\alpha,k)$,
$a_{i,k}=(a_i,k)$ and
$a_{i,k}^*=(a_i^*,k)$.
Similarly, let $\widetilde Q_f^\bullet$ be the graded quiver 
associated with the grading 
\begin{align}\label{degree2}
\begin{split}
\deg:\widetilde Q_{f,1}\to\bbZ
 ,\quad
\alpha_{ij},a_i,a_i^*\mapsto -1
 ,\quad
\varepsilon_i\mapsto 2.
\end{split}
\end{align}
A representation of $\widetilde Q^\bullet_f$ is a tuple
$x=(\alpha_k\,,\,a_{i,k}\,,\,a_{i,k}^*\,,\,\varepsilon_{i,k}).$
Fix $W$ in $\bfC^\bullet$.
We define the graded quiver varieties
$\frakM^\bullet(v,W)$ and $\frakM^\bullet_0(v,W)$ as in \cite[\S 3.1]{N11}. 
The variety $\frakM^\bullet(v,W)$ is $G^0_W$-equivariant, smooth and quasi-projective,
with a projective morphism $\pi^\bullet$ to the affine variety $\frakM_0^\bullet(v,W)$.
We can realize $\frakM^\bullet(v,W)$ and $\frakM_0^\bullet(v,W)$
as some fixed points loci in $\frakM(v,W)$ and $\frakM_0(v,W)$ in the following way.
The $I^\bullet$-grading on $W$ yields the following cocharacter
\begin{align}\label{sigma1}
\sigma:\bbC^\times\to G_W
 ,\quad
\sigma(z)=\bigoplus_{(i,k)\in I^\bullet}z^k\id_{W_{i,k}}.\end{align}
Let $G_W\times\bbC^\times$ denote the subgroup $G_W\times\xi(\bbC^\times)$ of $G_W\times T$.
Let 
\begin{align}\label{a}
a=(\sigma,\xi):\bbC^\times\to G_W\times T
\end{align}
and let $A\subset G_W\times T$
be the one subgroup such that $A=a(\bbC^\times)$.
We have 
\begin{align}\label{GA}
\frakM^\bullet(W)=\frakM(W)^A
 ,\quad
\frakM_0^\bullet(W)=\frakM_0(W)^A.
\end{align}
For any $R_{G_W\times\bbC^\times}$-module $M$ and any $\zeta\in\bbC^\times$ let 
$M|_{a(\zeta)}=M\otimes_{R_{G_W\times\bbC^\times}}\!\!\bbC$
be the specialization at the point $a(\zeta)$.

\subsubsection{Nakajima's quiver varieties and graded quiver varieties of Dynkin type}\label{sec:DQV}
Let $\bfc=(c_{ij})_{i,j\in I}$ be a symmetric Cartan matrix and
$\O\subset I\times I$ be an orientation :
$$(i,j)\in\O\ \text{or}\ (j,i)\in\O\iff c_{ij}<0
 ,\quad
(i,j)\in\O\Rightarrow (j,i)\not\in\O.$$
Set $\sgn_{ij}=1=-\sgn_{ji}$ if $(i,j)\in\O$ and $\sgn_{ij}=0$ if $c_{ij}=0$.
Let $Q$ be the quiver such that
$Q_0=I$ and
$Q_1=\{\alpha_{ij}:j\to i\,;\,(i,j)\in \O\}.$
We abbreviate
$\frakM(W)=\bigsqcup_{v\in\bbN I}\frakM(v,W).$
Note that $\frakM(v,W)=\emptyset$ except for finitely many $v$'s.
Let $v'\leqslant v$ if and only if $v'_i\leqslant v_i$ for all $i\in I.$
If $v'\leqslant v$ there is a closed embedding
$\frakM_0(v,'W)\subset\frakM_0(v,W)$ 
given by extending a representation by 0 to the complementary subspace.
Define
$\frakM_0(W)=\bigcup_{v\in\bbN I}\frakM_0(v,W).$
The colimit stabilizes. 
Let $\frakM_0^\reg(v,W)\subset\frakM_0(v,W)$ be the subset of closed free $G_V$-orbits.
We have a partition into locally closed subsets
$\frakM_0(W)=\bigsqcup_{v\in\bbN I^\bullet}\frakM_0^\reg(v,W)$.
Write
$$
(\alpha_i,w-\bfc v)=w_i-\sum_{j\in I}c_{ij}v_j
 ,\quad
w-\bfc v=\sum_{i\in I}\big(w_i-\sum_{j\in I}c_{ij}v_j\big)\,\delta_i
 ,\quad
v, w\in\bbN I.$$
The tuple $(v,w)$ is called dominant if $w-\bfc v$ lies in $\bbN I$.
If $\frakM_0^\reg(v,W)\neq\emptyset$ then $(v,w)$ is dominant.
The variety $\frakM_0(W)$ is an affine $G_W\times T$-variety 
and the map $\pi$ yields a projective morphism 
$\pi:\frakM(W)\to \frakM_0(W).$
Let $\frakL(W)$ be the zero fiber of $\pi$ and set
$\calZ(W)=\frakM(W)\times_{\frakM_0(W)}\frakM(W).$
In the graded case, we set
$\frakM^\bullet(W)=\bigsqcup_{v\in\bbN I^\bullet}\frakM^\bullet(v,W)$ and
$\frakM^\bullet_0(W)=\bigcup_{v\in\bbN I^\bullet}\frakM_0^\bullet(v,W)$, and
we define $\frakL^\bullet(W)$ and $\calZ^\bullet(W)$ as above.
Write
$$w-\bfc v=\sum_{(i,k)\in I^\bullet}\big(w_{i,k}-v_{i,k+1}-v_{i,k-1}-\sum_{j\neq i}c_{ij}v_{j,k}\big)\,\delta_{i,k}
 ,\quad
v,w\in\bbN I^\bullet.$$
The tuple $(v,w)$ is called $\ell$-dominant if $w-\bfc v$ lies in $\bbN I^\bullet$.
Let $v'\leqslant v$ if and only if $v'_{i,k}\leqslant v_{i,k}$ for all $(i,k)\in I^\bullet.$

\subsubsection{Nakajima's quiver varieties and quantum loop groups}\label{sec:N00}
Let us recall the relation between quantum loop groups and convolution algebras, 
following \cite{N00}.
See \S\ref{sec:QG} for a reminder on quantum loop groups.
Let $Q$ be a quiver of Dynkin type.
We fix some $\zeta\in\bbC^\times$ which is not a root of unity.
Recall that $\U_R(L\frakg)$ is the $R$-form of the quantum loop group of type $Q$ and that
$\U_F(L\frakg)=\U_R(L\frakg)\otimes_{R}F$,
$\U_\zeta(L\frakg)=\U_R(L\frakg)|_\zeta,$
where $(-)|_\zeta$ is the specialization along the map $R\to\bbC$,
$q\mapsto\zeta$.
The $F$-algebra $U_F(L\frakg)$ is generated by $x^\pm_{i,n}$, $\psi^\pm_{i,\pm m}$ with 
$n\in\bbZ$, $m\in\bbN$, and $\U_R(L\frakg)$
is the $R$-subalgebra generated by the elements
$\psi^\pm_{i,0},$
$h_{i,\pm m}/[m]_q,$
$(x^\pm_{i,n})^{[m]}$
with
$i\in I$,
$n\in\bbZ$
and
 $m\in\bbN^\times.$
Recall that $R=R_{\bbC^\times}$.
For any $R_{G_W\times\bbC^\times}$-module $M$, 
let $M/\tor\subset M\otimes_RF$ be the torsion free part over $R$.
By \cite[thm.~12.2.1]{N00}, there are $R_{G_W\times\bbC^\times}$-algebra homomorphisms
\begin{align}\label{nak1}
\U_R(L\frakg)\otimes R_{G_W}\to K^{G_W\times\bbC^\times}(\calZ(W))/\tor
\to K^{G_W\times\bbC^\times}_\top(\calZ(W))/\tor.
\end{align}
For any closed subgroup $A\subset G_W\times\bbC^\times$ there are representations 
of $\U_R(L\frakg)\otimes R_{G_W}$ in 
$K^A_\top(\frakM(W))=K^A(\frakM(W))$ and $K^A_\top(\frakL(W))=K^A_\top(\frakL(W))$.
Let $w$ be the dimension vector of $W$. 
The universal standard module, or global Weyl module,
is the $\U_R(L\frakg)\otimes R_{G_W}$-module 
$M(w)=K^{G_W\times\bbC^\times}(\frakL(W))$.
Fix a cocharacter $\sigma:\bbC^\times\to G_W$
and a compatible $I^\bullet$-grading on $W$.
Let $w$ denote also the dimension vector $(w_{i,k})$ of $W$ in $\bbN I^\bullet$.
Let $A$ be as in \S\ref{sec:graded quiver}.
Note that $R_A=R$
and that $A$ acts trivially on the varieties $\frakM^\bullet(W)$ and $\calZ^\bullet(W)$.
By \cite[(13.2.2)]{N00}, there are $R$-algebra homomorphisms
\begin{align}\label{nak2}
\U_R(L\frakg)\to K^A(\calZ^\bullet(W))\to K_\top^A(\calZ^\bullet(W)),
\end{align} 
and representations of $\U_R(L\frakg)$ in
$K_\top^A(\frakL^\bullet(W))$, $K_\top^A(\frakM^\bullet(W))$.
We have
$$K_\top^A(\frakL^\bullet(W))=K^A(\frakL^\bullet(W))
 ,\quad
K_\top^A(\frakM^\bullet(W))=K^A(\frakM^\bullet(W)).$$
By \cite[(13.4.2)]{N00}, specializing the quantum parameter to $\zeta$ yields the maps
\begin{align}\label{Psi}\U_\zeta(L\frakg)\to 
K(\calZ^\bullet(W))\to
K_\top(\calZ^\bullet(W))\to
H_\bullet(\calZ^\bullet(W),\bbC).\end{align}
The algebra $H_\bullet(\calZ^\bullet(W),\bbC)$ 
acts on  $H_\bullet(\frakL^\bullet(W),\bbC)$ and 
$H_\bullet(\frakM^\bullet(W),\bbC)$.
The standard module, or local Weyl module, with $\ell$-highest weight $\Psi_w$ is
the finite dimensional $\U_\zeta(L\frakg)$-module $M(w)|_{a(\zeta)}$
given by the specialization of $M(w)$ at the point $a(\zeta)$ in $A$. 
The Chern character and the Thomason localization theorem identify 
the standard module with $H_\bullet(\frakL^\bullet(W),\bbC).$
By \cite[thm.~7.4.1]{N00} there is also a perfect pairing
$H_\bullet(\frakM^\bullet(W),\bbC)\times H_\bullet(\frakL^\bullet(W),\bbC)\to\bbC.$
So we can consider the contragredient representation of $\U_\zeta(L\frakg)$ on the vector space
$H_\bullet(\frakM^\bullet(W),\bbC).$
This $\U_\zeta(L\frakg)$module is called the costandard module
with $\ell$-highest weight $\Psi_w$.
The pushforward by the closed embedding $\frakL^\bullet(W)\subset\frakM^\bullet(W)$
yields an homomorphism from the standard to the costandard module whose
image is the simple module $L(w)$ with the Drinfeld polynomial
$(\prod_{k\in\bbZ}(1-\zeta^ku)^{w_{i,k}})_{i\in I}$.
The Jordan-H\"older multiplicity of $L(v)$ in 
$H_\bullet(\frakL^\bullet(W),\bbC)$ is
the Euler characteristic $\chi_0(\IC_{\frakM_0^{\bullet\reg}(v,W)})$ of the stalk at 0 of the intermediate 
extension of the irreducible constant sheaf on the stratum
$\frakM_0^{\bullet\reg}(v,W)$ in $\frakM_0^\bullet(W)$.
The $q$-character of the standard and costandard modules are 
$$\sum_{v\in\bbN I^\bullet}\dim H_\bullet(\frakL^\bullet(v,W)\,,\,\bbC)\,e^{w-\bfc v},\quad
\sum_{v\in\bbN I^\bullet}\dim H_\bullet(\frakM^\bullet(v,W)\,,\,\bbC)\,e^{w-\bfc v}.$$

\subsubsection{Triple quiver varieties}\label{sec:triple quiver}
The $G_V\times G_W\times T$-action on
$\overline\X(V,W)$ lifts to an action on
$\widetilde\X(V,W)$ such that 
$G_V$ and $G_W$ act by conjugation and the element $(z_\alpha,z)\in T$
by multiplication by $z^{-2}$ on $\varepsilon_i$ for each vertex $i$.
A representation in $\widetilde\X(V,W)$ is stable if it has no non-zero 
subrepresentations supported on $V$.
Set 
\begin{align}\label{stable-triple1}
\widetilde\X(V,W)_s=\{x\in\widetilde\X(V,W)\,;\,x\ \text{is\ stable}\}.\end{align}
The triple quiver varieties associated are the categorical quotients 
$$\widetilde\frakM(v,W)=\widetilde\X(V,W)_s/G_V
 ,\quad
\widetilde\frakM_0(v,W)=\widetilde\X(V,W)/G_V.$$
We have an obvious 
$G_W\times T$-invariant projective map
$\tilde\pi:\widetilde\frakM(v,W)\to\widetilde\frakM_0(v,W).$
We abbreviate
$$\widetilde\frakM(W)=\bigsqcup_{v\in\bbN I}\widetilde\frakM(v,W)
 ,\quad
\widetilde\frakM_0(W)=\bigcup_{v\in\bbN I}\widetilde\frakM_0(v,W).$$
The second colimit is the extension of representations by 0 to the complementary subspace.
These colimits may not stabilize.
Thus $\widetilde\frakM_0(W)$ is an ind-scheme,
while $\widetilde\frakM(W)$ is a scheme locally of finite type.
Let $\widetilde\frakL(W)$ be the fiber at 0, and $\widetilde\calZ(W)$ be the
scheme locally of finite type given by the fiber product
$\widetilde\calZ(W)=\widetilde\frakM(W)\times_{\widetilde\frakM_0(W)}\widetilde\frakM(W)$.
Let  $\xi:\bbC^\times\to T$ be as in \eqref{xi}. We have
\begin{align}\label{q2}
\xi(z)\cdot x=(z\alpha,za,za^*,z^{-2}\varepsilon)
 ,\quad
x=(\alpha,a,a^*,\varepsilon).
\end{align}
This yields a $G_V\times G_W\times\bbC^\times$-action 
on the variety $\widetilde\X(V,W)$,
and a $G_W\times\bbC^\times$-action on $\widetilde\frakM(W)$
and $\widetilde\frakM_0(W)$.
We'll consider the open subsets 
\begin{align*}
\widetilde\frakM(W)_\circ=\bigsqcup_{v\in\bbN I}\widetilde\frakM(v,W)_\circ
\subset\widetilde\frakM(W)
,\quad
\widetilde\calZ(W)_\circ=\widetilde\calZ(W)\cap\widetilde\frakM(W)_\circ^2
\end{align*}
such that 
\begin{align*}
\widetilde\frakM(v,W)_\circ=\widetilde\X(v,W)_\circ/G_v
 ,\quad
\widetilde\X(v,W)_\circ=\overline\X(V,W)_s\times\frakg_v.
\end{align*} 
We'll also consider the varieties
\begin{align}\label{hatM}
\widehat\frakM(W)=\{\ux\in\widetilde\frakM(W)\,;\,a^*=0\}=\widehat\X(V,W)_s\,/\,G_V
\end{align}
where
$\widehat\X(V,W)_s=\widetilde\X(V,W)_s\cap\widehat\X(V,W).$
Set
$$\widehat\calZ(W)=\widetilde\calZ(W)\cap \widehat\frakM(W)^2,\quad
\widehat\frakL(W)=\widetilde\frakL(W)\cap\widehat\frakM(W),\quad
\widehat\frakM(W)_\circ=\widetilde\frakM(W)_\circ\cap\widehat\frakM(W).$$
The graded quiver varieties of $\widetilde Q^\bullet_f$ are
$$\widetilde\frakM^\bullet(W)\subset\widetilde\frakM(W)
 ,\quad
\widetilde\frakM^\bullet_0(W)\subset\widetilde\frakM_0(W)
 ,\quad
\widetilde\frakM^\bullet(W)_\circ=\widetilde\frakM(W)_\circ\cap\widetilde\frakM^\bullet(W).$$
We define the varieties
$\widetilde\frakL^\bullet(W)$, $\widetilde\calZ^\bullet(W)$,
$\widehat\frakM^\bullet(W)$, $\widehat\frakL^\bullet(W)$,
$\widehat\calZ^\bullet(W)$, $\widehat\frakM^\bullet(W)_\circ$, etc, similarly.

\subsubsection{Hecke correspondences}\label{sec:Hecke}
Fix $W\in\bfC$. 
The Hecke correspondence $\widetilde\frakP(W)$ is the scheme given by
$$\widetilde\frakP(W)=\{(x,y,\tau)\in\widetilde\frakM(W)^2\times\Hom_{\widetilde Q_f}(x,y)\,;\,
\tau|_W=\id_W\}.$$
For each triple $(x,y,\tau)$ the map $\tau$ is injective, because the representation $x$ is stable.
For the same reason, there is a closed embedding 
$i:\widetilde\frakP(W)\to\widetilde\frakM(W)^2$ such that
$(x,y,\tau)\mapsto(x,y).$
Hence, we may write
$\widetilde\frakP(W)=\{(x,y)\in\widetilde\frakM(W)^2\,;\,x\subset y\}.$
The opposite Hecke correspondence is
$\widetilde\frakP(W)^\op=\{(x,y)\in\widetilde\frakM(W)^2\,;\,y\subset x\}$.
Let $\Rep$ be the moduli stack of representations of $\widetilde Q$.
We have
$\Rep=\bigsqcup_{v\in\bbN I}\Rep_v$
where $\Rep_v$ is the quotient stack
$\Rep_v=\big[\widetilde\X(v)\,/\,G_v\big].$
Let $\pi:\widetilde\frakP(W)\to\Rep$ be the stack homomorphism taking the pair $(x,y)$ to $y/x.$
We define the nilpotent Hecke correspondence to be the fiber product 
$\widetilde\frakP(W)^\nil=\widetilde\frakP(W)\times_{\Rep}\Rep^\nil.$
For  $v_1\leqslant v_2$, we write
\begin{align*}
\widetilde\frakP(v_1,v_2,W)&=
\widetilde\frakP(W)\cap\big(\widetilde\frakM(v_1,W)\times\widetilde\frakM(v_2,W)\big),\\\widetilde\frakP(v_2,v_1,W)&=
\widetilde\frakP(W)^\op\cap\big(\widetilde\frakM(v_2,W)\times\widetilde\frakM(v_1,W)\big).
\end{align*}
We also write
\begin{align*}
\widetilde\frakP(\delta_i,W)=\bigsqcup_{v\in\bbN I}\widetilde\frakP(v,v+\delta_i,W)
 ,\quad
\widetilde\frakP(-\delta_i,W)=\bigsqcup_{v\in\bbN I}\widetilde\frakP(v+\delta_i,v,W).
\end{align*}

\begin{lemma}
\hfill
\begin{enumerate}[label=$\mathrm{(\alph*)}$,leftmargin=8mm]
\item
The scheme $\widetilde\frakP(W)$ is smooth and locally of finite type.
\item
The map $\pi:\widetilde\frakP(W)\to\Rep$ is flat.
 \item
The map $i$ takes $\widetilde\frakP(W)^\nil$ into $\widetilde\calZ(W)$.
\end{enumerate}
\end{lemma}

\begin{proof}
We'll write
$v_1\leqslant v_2$ if and only if $v_2-v_1\in\bbN I.$
Let $P_{v_1,v_2}\subset G_{v_2}$ be the stabilizer of the flag $\bbC^{v_1}\subset\bbC^{v_2}$.
To prove (a), (b), note that 
$$\widetilde\frakP(v_1,v_2,W)=\widetilde\X(v_1,v_2,W)_s\,/\,P_{v_1,v_2}$$
is the categorical quotient of
\begin{align*}
\widetilde\X(v_1,v_2,W)_s=\{y\in\widetilde\X(v_2,W)_s\,;\,y(\bbC^{v_1}\oplus W)
\subseteq\bbC^{v_1}\oplus W\}.
\end{align*}
The $P_{v_1,v_2}$-action is proper and free 
because the point $y$ is stable.
Part (c) follows from the Hilbert-Mumford criterion.
For any pair $(x,y)$ in $\widetilde\X(v_1,W)\times\widetilde\X(v_2,W)$
representing a point in $\widetilde\frakP(W)^\nil$ there is a 1-parameter subgroup
$\lambda$ in $G_{v_2}$ such that 
$\lim_{t\to\infty}\lambda(t)\cdot y=x\oplus 0.$
Hence $\pi(x)=\pi(y)$ in $\widetilde\frakM_0(W)$.
\end{proof}

Considering the quiver $\widehat Q_f$ instead of $\widetilde Q_f$, we have the Hecke correspondence 
 $\widehat\frakP(W)=\widetilde\frakP(W)\cap \widehat\frakM(W)^2$,
and we define $\widehat\frakP(\delta_i,W)$, 
$\widehat\frakP(-\delta_i,W)$ in the obvious way.

\subsubsection{Universal bundles}
Let $\calV=\bigoplus_{i\in I}\calV_i$ and 
$\calW=\bigoplus_{i\in I}\calW_i$ denote both the tautological bundles on
$\widehat\frakM(W)$ and $\widetilde\frakM(W)$
and their classes in $K_{G_W\times\bbC^\times}\big(\widehat\frakM(W)\big)$ and
$K_{G_W\times\bbC^\times}\big(\widetilde\frakM(W)\big)$.
Given an orientation as in \S\ref{sec:DQV}, we define
\begin{align}\label{Vcirc}
\calV_{\circ i}=\bigoplus_{c_{ij}<0}\calV_j=\calV_{+i}\oplus\calV_{-i}
,\quad
\calV_{- i}=\bigoplus_{c_{ij},\sgn_{ij}<0}\calV_j
,\quad
v_{\circ i}=\sum_{c_{ij}<0}v_j=v_{+i}+v_{-i}.
\end{align}
Let $\calV^-_i$ and $\calV^+_i$ be the pull-back of the tautological vector bundle 
$\calV_i$ on $\widetilde\frakM(W)$
by the first and second projection 
$\widetilde\frakP(\delta_i,W)\to\widetilde\frakM(W)$.
Switching both components of $\widetilde\frakM(W)^2$,
we define similarly the vector bundles
$\calV^-_i$, $\calV^+_i$ on the Hecke correspondence $\widetilde\frakP(-\delta_i,W)$.
Let $\calL_i$ denote the invertible sheaf $\calV^+_i/\,\calV^-_i$ on the Hecke correspondence
$\widetilde\frakP(\pm\delta_i,W)$,
and its pushforward by the closed embedding into $\widetilde\calZ(W)$.
We define the bundles $\calV^-_i$, $\calV^+_i$, $\calL_i$
on $\widehat\frakP(\pm\delta_i,W)$ in the obvious way.

\subsubsection{Potentials}\label{sec:potential}
Fix an homogeneous potential $\bfw\!_f$ on $\widetilde Q_f$ of degree 0
relatively to the grading \eqref{degree2}.
Let $\bfw$ be the restriction of $\bfw\!_f$ to $\widetilde Q$.
All potential will be assumed to be algebraic, i.e., they are finite linear combinations
of cyclic words of the quiver.
If $Q$ is a Dynkin quiver as in \S\ref{sec:DQV},
we'll assume that either
$\bfw\!_f=\bfw_1$ or $\bfw\!_f=\bfw_2$ with
\begin{align}\label{w12}
\begin{split}
\bfw_1=\varepsilon\,[\alpha,\alpha^*]
+\varepsilon\, a^*a
 ,\quad
\bfw_2=\varepsilon\,[\alpha,\alpha^*]
\end{split}
\end{align}
In both cases we have $\bfw=\bfw_2$.
We equip the quiver $\widetilde Q^\bullet_f$ with the following potentials
\begin{align*}
\bfw^\bullet_2&=\sum_{(i,j,k)\in\O\times\bbZ}
\big(\varepsilon_{i,k-2}\,\alpha_{ij,k-1}\,\alpha_{ji,k}-\varepsilon_{j,k-2}\,\alpha_{ji,k-1}\,\alpha_{ij,k}\big),\\
\bfw^\bullet_1&=
\bfw^\bullet_2+\sum_{(i,k)\in I^\bullet}\varepsilon_{i,k-2}\,a^*_{i,k-1}\,a_{i,k}.
\end{align*}
Let $\tilde f_1,\tilde f_2:\widetilde\frakM(W)\to\bbC$ be the traces of $\bfw_1$, $\bfw_2$
and $\tilde f_{1,\circ},\tilde f_{2,\circ}:\widetilde\frakM(W)_\circ\to\bbC$ be their restriction to
$\widetilde\frakM(W)_\circ$.
We'll abbreviate $\tilde f$ for either $\tilde f_1$ or $\tilde f_2$.
Let $\tilde f_1^\bullet,\tilde f_2^\bullet:\widetilde\frakM^\bullet(W)\to\bbC$ 
be the traces of $\bfw_1^\bullet$, $\bfw_2^\bullet$ and $\tilde f^\bullet_{1,\circ}$,
$\tilde f^\bullet_{2,\circ}$ be
their restrictions to $\widetilde\frakM^\bullet(W)_\circ$.
Similarly, let $h:\Rep\to\bbC$ be the trace of  $\bfw_2$ of $\tilde Q$.
Recall the following diagram introduced in \S\ref{sec:Hecke}
$$\xymatrix{\Rep&\ar[l]_\pi\widetilde\frakP(W)\ar[r]^i&\widetilde\frakM(W)^2}$$

\begin{lemma}\label{lem:potential1}
We have $i^*(\tilde f^{(2)})=\pi^*(h)$.
 \end{lemma}

\begin{proof}
Fix point $(x,y)\in\widetilde\frakP(v_1,v_2,W)$. We have 
$$y=(\alpha,\varepsilon, a,a^*)\in\widetilde\frakM(v_2,W)
,\quad
x=(\alpha|_{V_1},\varepsilon|_{V_1}, a,a^*)\in\widetilde\frakM(v_1,W)
,\quad
a^*(W)\subset V_1\subset V_2.$$ 
Then, we have $\pi(x,y)=y/x=(\alpha|_{V_2/V_1},\varepsilon|_{V_2/V_1})$.
Further, either $\widetilde f^{(2)}(x,y)=\tilde f_1(y)-\tilde f_1(x)=\tilde f_2(y/x)=h(y/x)$ or 
$\widetilde f^{(2)}(x,y)=\tilde f_2(y)-\tilde f_2(x)=\tilde f_2(y/x)=h(y/x)$.
\end{proof}

\subsection{The KCA of a triple quiver with potential}\label{sec:CAT}
In this section we compute some KCA's of triple quivers with potentials.
To do that, we must relate KCA's to KHA's.

\subsubsection{The KHA of a triple quiver with potential}\label{sec:KHA}
We first recall the definition of the KHA of the quiver with potential
$(\widetilde Q,\bfw)$, following \cite{P21}, \cite{VV22}.
A representation in $\Rep$ is nilpotent if its image in the categorical quotient  
$\bigsqcup_{v\in\bbN I}\widetilde\X(v)\,/\,G_v$ is zero.
Let $\Rep^\nil\subset\Rep$ be the closed substack parametrizing the nilpotent representations.
Let $\Rep'$ be the  stack of pairs of representations $(x,y)$ 
with an inclusion $x\subset y$. 
The stacks $\Rep$ and $\Rep'$ are smooth and locally of finite type.
Consider the diagram 
\begin{align*}
\xymatrix{\Rep\times\Rep&\ar[l]_-q\Rep'\ar[r]^-p&\Rep}
 ,\quad
q(x,y)=(x,y/x)
 ,\quad
p(x,y)=y.
\end{align*}
The map $p$ is proper, the map $q$ is smooth.
We equip the stack $\Rep$ with the $T$-action in \S\ref{sec:triple quiver}.
Let $\SG\subset T$ be a closed subgroup.
Since $h^{\oplus 2}\circ q=h\circ p$, the following functor is well-defined
\begin{align}\label{Hall1}
\star:\DCoh_{\SG}(\Rep,h)_{\Rep^\nil}
\times\DCoh_{\SG}(\Rep,h)_{\Rep^\nil}
\to\DCoh_{\SG}(\Rep,h)_{\Rep^\nil}
,\quad
(\calE,\calF)\mapsto Rp_*Lq^*(\calE\boxtimes\calF).
\end{align}
It yields a monoidal structure on the triangulated category 
$\DCoh_{\SG}(\Rep,h)_{\Rep^\nil}$.
We'll abbreviate $R=R_{\SG}$ and $F=F_{\SG}$.
Taking the Grothendieck groups, we get the $R$-algebra
$K_{\SG}(\Rep,h)_{\Rep^\nil}$, whose opposite is denoted by $K_{\SG}(\Rep,h)_{\Rep^\nil}^\op$.
This $R$-algebra is the nilpotent KHA of the pair $(\widetilde Q,\bfw)$.
From now on we'll omit the word nilpotent.
Let $\Rep^0\subset\Rep^\nil$ be the zero locus of the function $h$ in \S\ref{sec:potential}. 
By \eqref{Upsilon2} there is an $R$-linear map
\begin{align}\label{UPM}\Upsilon:K^{\SG}(\Rep^0)\to K_{\SG}(\Rep,h)_{\Rep^\nil}.
\end{align}
Note that $\Rep^0_{\delta_i}$ is the classifying stack of the group $G_{\delta_i}$.
Let $\calL_i$ be the line bundle on $\Rep^0_{\delta_i}$ associated with
the linear character of $G_{\delta_i}$.
We consider the $F$-subalgebra 
$\calU_F^+$ of $K_{\SG}(\Rep,h)_{\Rep^\nil}\otimes_RF$
generated by the elements
$x^+_{i,n}=\Upsilon(\calL_i^{\otimes n})$
with $i\in I$,
$n\in\bbZ$.
Let $\calU_F^-$ be the $F$-algebra opposite to $\calU_F^+$ and
$x^-_{i,n}$ be the image of $x^+_{i,n}$ in $\calU_F^-$.

Now, we fix $\SG=\bbC^\times$ as in \eqref{xi}.
Hence $R=\bbC[q,q^{-1}]$.
Let $\calU_R^\pm$ be the $R$-subalgebra of $\calU_F^\pm$ 
generated by the elements $(x^\pm_{i,n})^{[m]}$
with $i\in I$, $n\in\bbZ$ and $m\in\bbN$.
For each $v\in\bbN I$, let $\calU^\pm_{R,\pm v}$ 
be the $R$-submodule of $\calU_R^\pm$ spanned by the classes of the
$v$-dimensional representations. 
It is equipped with its
obvious $R_{G_v\times\bbC^\times}$-module structure.
Let $\calV\in R_{G_v\times\bbC^\times}$ denote the class of the vectorial representation of $G_v$,
with its obvious decomposition $\calV=\bigoplus_{i\in I}\calV_i.$
The twisted Hall multiplication $\circlestar$ on $\calU_R^+$  in \cite[\S 2.3.8]{VV22} 
is the composition  of $\star$ and the linear endomorphism of
$\calU_R^+\otimes \calU_R^+$ given by
\begin{align}\label{twist1}
x_1\otimes x_2\mapsto(x_1\otimes x_2)\cdot\prod_{\substack{\alpha\in Q_1\\\alpha:i\to j}}
(-1)^{v_iv_j}\det(\calV_i\otimes\calV_j^\vee).
\end{align}
The twisted Hall multiplication $\circlestar$ on $\calU_R^-$ is opposite to the
multiplication $\circlestar$ on  $\calU_R^+$.

\begin{proposition}[\cite{VV22}]\label{prop:Omega1}
Let $Q$ be a Dynkin quiver.
\begin{enumerate}[label=$\mathrm{(\alph*)}$,leftmargin=8mm]
\item
There is an $R$-algebra isomorphism
$(\,\calU_R^\pm\,,\,\circlestar)=U_R(L\frakg)^\pm$
taking $(x^\pm_{i,n})^{[m]}$ to $(x^\pm_{i,n})^{[m]}$.
\item
We have $\calU_R^+=K_{\SG}(\Rep,h)_{\Rep^\nil}$.
\qed
\end{enumerate}
\end{proposition}

\subsubsection{From KHA's to KCA's}\label{sec:double}
Let $\SG\subset T$ be any subgroup.
The pair $(\widetilde\frakM(W),\tilde f)$ is a smooth $G_W\times \SG$-invariant LG-model.
Applying the results of \S\ref{sec:critalg1} with
$$G=G_W\times \SG
 ,\quad
X=\widetilde\frakM(W)
 ,\quad
X_0=\widetilde\frakM_0(W)
 ,\quad
Z=\widetilde\calZ(W)
 ,\quad
f=\tilde f,$$
 we get a monoidal category and an associative $R_{G_W\times \SG}$-algebra 
$$\DCoh_{G_W\times \SG}(\widetilde\frakM(W)^2,\tilde f^{(2)})_{\widetilde\calZ(W)}
 ,\quad
K_{G_W\times \SG}(\widetilde\frakM(W)^2,\tilde f^{(2)})_{\widetilde\calZ(W)}.$$

\begin{proposition}\label{prop:double}
 There is a commutative diagram of $R_{G_W\times \SG}$-algebras
 \begin{align*}
\begin{split}
\xymatrix{
K^{\SG}(\Rep^0)\otimes R_{G_W}\ar[d]_-{\varpi^+}\ar[r]^-{\Upsilon\otimes 1}&
K_{\SG}(\Rep,h)_{\Rep^\nil}\ar[d]^-{\omega^+}\otimes R_{G_W}\\
K^{G_W\times \SG}(\widetilde\calZ(W))\ar[r]^-\Upsilon&
K_{G_W\times \SG}(\widetilde\frakM(W)^2,\tilde f^{(2)})_{\widetilde\calZ(W)}
}
\end{split}
\end{align*}
\end{proposition}

\begin{proof} 
We first define
a monoidal triangulated functor
$$\omega^+:\DCoh_{G_W\times \SG}(\Rep,h)_{\Rep^\nil}\to \DCoh_{G_W\times \SG}(\widetilde\frakM(W)^2,\tilde f^{(2)})_{\widetilde\calZ(W)}.$$
Taking the Grothendieck groups it yields the $R_{G_W\times \SG}$-algebra homomorphism
\begin{align*}
\omega^+:K_{\SG}(\Rep,h)_{\Rep^\nil}\otimes R_{G_W}\to 
K_{G_W\times \SG}(\widetilde\frakM(W)^2,\tilde f^{(2)})_{\widetilde\calZ(W)}.
\end{align*}
To do so, we consider the  commutative diagram of stacks with a Cartesian right square
\begin{align}\label{diagram1}
\begin{split}
\xymatrix{
\widetilde\calZ(W)\ar[d]&\ar[l]_-i\widetilde\frakP(W)^\nil\ar[r]^-\pi\ar[d]&\Rep^\nil\ar[d]\\
\widetilde\frakM(W)^2&\ar[l]_-i\widetilde\frakP(W)\ar[r]^-\pi&\Rep}
\end{split}
\end{align}
It yields the functors
\begin{align*}
&Ri_*:\DCoh_{G_W\times \SG}(\widetilde\frakP(W),i^*\tilde f^{(2)})_{\widetilde\frakP(W)^\nil}\to
\DCoh_S(\widetilde\frakM(W)^2,\tilde f^{(2)})_{\widetilde\calZ(W)},\\
&L\pi^*:\DCoh_{G_W\times \SG}(\Rep,h)_{\Rep^\nil}\to
\DCoh_{G_W\times \SG}(\widetilde\frakP(W),\pi^*h)_{\widetilde\frakP(W)^\nil}.
\end{align*}
By Lemma \ref{lem:potential1} we have $i^*\tilde f^{(2)}=\pi^*h$.
Thus, composing $Ri_*$ and $L\pi^*$ we get a functor
\begin{align*}
\omega^+:\DCoh_{G_W\times \SG}(\Rep,h)_{\Rep^\nil}\to 
\DCoh_{G_W\times \SG}(\widetilde\frakM(W)^2,\tilde f^{(2)})_{\widetilde\calZ(W)}.
\end{align*}
We claim that the functor $\omega^+$ has a monoidal structure.
Indeed, set
$$\widetilde\frakP'(W)=\{(x,y,z)\in\widetilde\frakM(W)^3\,;\,x\subset y\subset z\}.$$
We have the following commutative diagram
\begin{align}
\begin{split}
\xymatrix{
\widetilde\frakM(W)^2\times \widetilde\frakM(W)^2&&\ar[ll]_-{\pi_{12}\times\pi_{23}}\widetilde\frakM(W)^3
\ar[r]^-{\pi_{13}}&\widetilde\frakM(W)^2\\
\widetilde\frakP(W)\times \widetilde\frakP(W)\ar[d]_-{\pi\times\pi}\ar[u]^-{i\times i}&&
\ar[ll]\widetilde\frakP'(W)\ar[u]^{i'}\ar[r]\ar[d]_-{\pi'}&
\widetilde\frakP(W)\ar[d]_-\pi\ar[u]^-i\\
\Rep\times\Rep&&\ar[ll]_-q\Rep'\ar[r]^-p&\Rep
}
\end{split}
\end{align}
The left upper square is Cartesian. 
The right lower one either because the set of stable representations of the quiver $\widetilde Q_f$
is preserved by subobjects.
By base change, we get an isomorphism of functors
\begin{align*}
\omega^+\circ\star&=Ri_*\circ L\pi^*\circ Rp_*\circ Lq^*\\
&=R(\pi_{13})_*\circ L(\pi_{12}\times\pi_{23})^*\circ R(i\times i)_*\circ L(\pi\times\pi)^*\\
&=\star\circ R(i\times i)_*\circ L(\pi\times\pi)^*
\end{align*}
where the convolution functors $\star$ are as in \eqref{Hall1} and \eqref{conv2},
proving the claim.
More precisely, for the right lower square
we use the flat base change, and for the
left upper square the fact that
$\widetilde\frakP(W)\times \widetilde\frakP(W)$ and $\widetilde\frakM(W)^3$
intersect transversally in $\widetilde\frakM(W)^2\times \widetilde\frakM(W)^2$.

Next, we define the map $\varpi^+$.
For $\flat=0$ or $\nil$, we have by \eqref{Hall1} a functor
\begin{align}\label{Hall3}
\begin{split}
\star:\Db\Coh_{\SG}(\Rep)_{\Rep^\flat}
\times\Db\Coh_{\SG}(\Rep)_{\Rep^\flat}
\to\Db\Coh_{\SG}(\Rep)_{\Rep^\flat}
\end{split}
\end{align}
This functor yields an $R_{\SG}$-algebra structure on $K^{\SG}(\Rep')$.
The pushforward $K^{\SG}(\Rep^0)\to K^{\SG}(\Rep^\nil)$ is an algebra homomorphism.
Composing it with
$$Ri_*\circ L\pi^*:K^{\SG}(\Rep^\nil)\to K^{G_W\times \SG}(\widetilde\calZ(W))$$
we get the map $\varpi^+:K^{\SG}(\Rep^0)\to K^{G_W\times \SG}(\widetilde\calZ(W)).$

Finally, the map $\Upsilon$ in \eqref{UPM} is an algebra homomorphism by Lemma \ref{lem:Upsilon5}.
The same argument as for $\omega^+$ 
proves that $\varpi^+$ is an algebra homomorphism.
The diagram in the proposition commutes by
Lemma \ref{lem:Upsilon5} and \eqref{diagram1}.
\end{proof}

Taking the opposite algebras and Hecke correspondences, 
we get in a similar way  the commutative diagram of 
$R_{G_W\times \SG}$-algebras
 \begin{align*}
\begin{split}
\xymatrix{
K^{\SG}(\Rep^0)^\op\otimes R_{G_W}\ar[d]_-{\varpi^-}\ar[r]^-{\Upsilon\otimes 1}&
K_{\SG}(\Rep,h)_{\Rep^\nil}^\op\ar[d]^-{\omega^-}\otimes R_{G_W}\\
K^{G_W\times \SG}(\widetilde\calZ(W))\ar[r]^-\Upsilon&
K_{G_W\times \SG}(\widetilde\frakM(W)^2,\tilde f^{(2)})_{\widetilde\calZ(W)}
}
\end{split}
\end{align*}
Now, we fix $\SG=\bbC^\times$ as in \eqref{xi}.
We'll need a twisted version of the maps $\omega^\pm$. 
To define them, we consider the decomposition
$$K_{G_W\times\bbC^\times}(\widetilde\frakM(W)^2,\tilde f^{(2)})_{\widetilde\calZ(W)}=
\bigoplus_{v_1,v_2\in\bbN I}K_{G_W\times\bbC^\times}\big(\widetilde\frakM(v_1,W)\times\widetilde\frakM(v_2,W),\tilde f^{(2)}\big)_{\widetilde\calZ(W)}$$
and we write
$\omega^\pm=\bigoplus_{v_1,v_2}\omega^\pm_{v_1,v_2}$ with
$$\omega^\pm_{v_1,v_2}:\calU_R^\pm\to 
K_{G_W\times\bbC^\times}\big(\widetilde\frakM(v_1,W)\times\widetilde\frakM(v_2,W),\tilde f^{(2)}\big)_{\widetilde\calZ(W)}.$$
Let $\calV_1,$ $ \calV_2$ be the classes in
$R_{G_{v_1}\times G_{v_2}\times\bbC^\times}$
of the vectorial representations of the groups $G_{v_1}$, $G_{v_2}$,
with their obvious decompositions
$\calV_1=\bigoplus_{i\in I}\calV_{1,i}$ and
$\calV_2=\bigoplus_{i\in I}\calV_{2,i}$.
We define
\begin{gather*}
(v_1 \,|\,v_2)=\sum_{\substack{\alpha\in Q_1\\\alpha:i\to j}}v_{1i} \,v_{2j}
 ,\quad
\alpha_{v_1,v_2}=
\prod_{\substack{\alpha\in Q_1\\\alpha:i\to j}}\det(\calV_{1,i})^{v_{2j}}
 ,\quad
\beta_{v_1,v_2}=
\prod_{\substack{\alpha\in Q_1\\\alpha:i\to j}}\det(\calV_{1,j})^{-v_{2i}}.
\end{gather*} 
Given $x_1\in\calU^+_{R,v_1}$ and $x_2\in\calU^+_{R,v_2}$
the formula \eqref{twist1} yields
$$x_1\circlestar x_2=
(-1)^{(v_1\,|\,v_2)}(\alpha_{v_1,v_2}x_1)\star (\beta_{v_2,v_1}x_2).$$
For $x_1\in\calU^-_{R,-v_1}$ and $x_2\in\calU^-_{R,-v_2}$ we have instead
$$x_1\circlestar x_2=
(-1)^{(v_2\,|\,v_1)}(\beta_{v_1,v_2}x_1)\star (\alpha_{v_2,v_1}x_2).$$
We'll abbreviate
$e_{v_1,v_2}=e_{-v_2,-v_1}=(-1)^{(v_1\,|\,v_2)}\,\alpha_{v_1,v_2}\,\beta_{v_2,v_1}.$
Choose some elements $r_{v_1,v_2}\in R_{G_{v_1}\times G_{v_2}\times G_W\times\bbC^\times}$ 
for each $v_1,v_2\in\bbN I$ such that 
$$
v_3\geqslant v_2\geqslant v_1\ \text{or}\ 
v_3\leqslant v_2\leqslant v_1\Rightarrow
e_{v_2-v_1,v_3-v_2}=r_{v_1,v_3}^{-1}\,
r_{v_1,v_2}\,r_{v_2,v_3}$$
Finally, we consider the following map 
\begin{align}\label{twist2}
\Omega^\pm:\calU_R^\pm\to 
K_{G_W\times\bbC^\times}(\widetilde\frakM(W)^2,\tilde f^{(2)})_{\widetilde\calZ(W)}
,\quad
\Omega^\pm=\bigoplus_{v_1,v_2}r_{v_1,v_2}\,\omega^\pm_{v_1,v_2}.
\end{align}
From now on  we'll equip $\calU_R^\pm$ 
with the twisted Hall multiplication $\circlestar$ and
we'll omit the symbol $\circlestar$.
Further, we'll normalize the twist $r_{v_1,v_2}$ such that
$r_{v_1,v_2}=1$ whenever
$v_2-v_1=\pm\delta_i$.

\begin{proposition}\label{prop:Omega2}
The map $\Omega^\pm$ is an $R$-algebra homomorphism 
$$\Omega^\pm:\calU_R^\pm\to 
K_{G_W\times\bbC^\times}(\widetilde\frakM(W)^2,\tilde f^{(2)})_{\widetilde\calZ(W)}.$$
\qed
\end{proposition}

Proposition \ref{prop:double} holds with  $\widetilde\frakP(W)$, $\widetilde\frakM(W)$, $\widetilde\frakZ(W)$
replaced by $\widehat\frakP(W)$, $\widehat\frakM(W)$, $\widehat\frakZ(W)$.
We define as in \eqref{twist2} the map
\begin{align}\label{twist3}
\Omega^\pm:\calU_R^\pm\to 
K_{G_W\times\bbC^\times}(\widehat\frakM(W)^2,\tilde f^{(2)})_{\widehat\calZ(W)}
\end{align}

\subsubsection{The KCA associated with the potential $\bfw_1$}
\label{sec:Jordan1}

\begin{proposition}\label{prop:crit1}
\hfill
\begin{enumerate}[label=$\mathrm{(\alph*)}$,leftmargin=8mm]
\item The extension by zero yields an isomorphism
$\frakM(W)=\crit(\tilde f_1).$

\item
We have the following algebra and module isomorphisms 
\begin{align*}
K_{G_W\times\bbC^\times}\big(\widetilde\frakM(W)^2,(\tilde f_1)^{(2)}\big)_{\widetilde\calZ(W)}=
K^{G_W\times\bbC^\times}\big(\calZ(W)\big),\quad
K_{G_W\times\bbC^\times}\big(\widetilde\frakM(W),\tilde f_1\big)
=K^{G_W\times\bbC^\times}\big(\frakM(W)\big).
\end{align*}
\end{enumerate}
\end{proposition}

\begin{proof}
We have 
$$\{(x,\varepsilon)\in\widetilde\X(V,W)_s\,;\,
[\varepsilon,x]=\mu_V(x)=0\}\,/\,G_V=\crit(\tilde f_1)\cap\widetilde\frakM(v,W).$$
For any tuple $(x,\varepsilon)$ as above,
the subspace $\Im(\varepsilon)$ of $V$ is preserved by the
action of the path algebra $\bbC\widetilde Q$ of $\widetilde Q$ and is contained in the kernel of $a$.
Hence, we have $\varepsilon=0$ and $x\in\mu_V^{-1}(0)_s$.
Thus, the assignment $x\mapsto (x,0)$ yields an isomorphism
$\frakM(W)=\crit(\tilde f_1)$, proving Part (a). 
To prove the part (b) observe that by (a) we have
$\crit(\tilde f_1)\subset\widetilde\frakM(W)_\circ$.
Since any matrix factorization is supported on the critical set of the potential by
\cite[cor.~3.18]{PV11}, we have
$$K_{G_W\times\bbC^\times}\big(\widetilde\frakM(W),\tilde f_1\big)=
K_{G_W\times\bbC^\times}\big(\widetilde\frakM(W)_\circ,\tilde f_{1,\circ}\big).$$
Next, we use the dimensional reduction in K-theory. 
More precisely, we apply \cite{I12} or \cite[thm.~1.2]{H17b}
to the vector bundle
$$\rho_1:\widetilde\frakM(W)_\circ\to\big\{\ux\in\widetilde\frakM(W)_\circ\,;\,
\varepsilon=0\big\}$$
given by forgetting the variable $\varepsilon$. Using the isomorphism
\begin{align*}
\frakM(W)&=\big\{\ux\in\widetilde\frakM(W)_\circ\,;\,
\varepsilon=0\,,\,\partial\tilde f_{1,\circ}/\partial\varepsilon(\ux)=0\big\}
\end{align*}
we deduce that
$$K_{G_W\times\bbC^\times}\big(\widetilde\frakM(W)_\circ,\tilde f_{1,\circ}\big)
=K^{G_W\times\bbC^\times}\big(\frakM(W)\big).$$
In a similar way we prove that
\begin{align*}
K_{G_W\times\bbC^\times}\big(\widetilde\frakM(W)^2,(\tilde f_1)^{(2)}\big)_{\widetilde\calZ(W)}
&=
K_{G_W\times\bbC^\times}\big(\widetilde\frakM(W)_\circ^2,(\tilde f_{1,\circ})^{(2)}
\big)_{\widetilde\calZ(W)_\circ}\\
&=
K^{G_W\times\bbC^\times}\big(\calZ(W)\big).
\end{align*}
\end{proof}

The Nakajima's construction recalled in \S\ref{sec:N00} yields the following.

\begin{theorem}\label{thm:Nakajima}
Let $Q$ be a Dynkin quiver. 
\hfill
\begin{enumerate}[label=$\mathrm{(\alph*)}$,leftmargin=8mm]
\item
There is an $R$-algebra map
$\U_R(L\frakg)\to 
K_{G_W\times\bbC^\times}\big(\widetilde\frakM(W)^2,(\tilde f_1)^{(2)}\big)_{\widetilde\calZ(W)}.$
\item
The $R$-algebra $\U_R(L\frakg)$ acts on 
$K_{G_W\times \bbC^\times}(\widetilde\frakM(W),\tilde f_1)$.
\end{enumerate}
\end{theorem}

\begin{proof}
Part (b) follows from (a), and (a) from \eqref{nak1} and Proposition \ref{prop:crit1}. 
\end{proof}

\begin{remark}\label{rem:VV}
\hfill
\begin{enumerate}[label=$\mathrm{(\alph*)}$,leftmargin=8mm]
\item
The theorem holds for any quiver without edge loops, as well as for the Jordan quiver,
see \S\ref{sec:toroidal}.
\item
The same proof as in Proposition \ref{prop:crit1} implies that
the extension by zero is an isomorphism $\frakM^\bullet(W)=\crit(\tilde f^\bullet_{\!1})$
and that
$K\big(\widetilde\frakM^\bullet(W),\tilde f^\bullet_1\big)=
K\big(\frakM^\bullet(W)\big)$ and
$H^\bullet\big(\widetilde\frakM^\bullet(W),\tilde f^\bullet_1\big)=
H^\bullet\big(\frakM^\bullet(W)\big).$
\item 
The relation between $\crit(\tilde f_1)$ and Nakajima's quiver varieties
is not new. 
It appears already in the literature in several forms, see, e.g., \cite{D16}, \cite{L20}.

\end{enumerate}
\end{remark}

\subsubsection{The KCA associated with the potential $\bfw_2$}\label{sec:Jordan2}
Let $A\subset G_W\times\bbC^\times$ be as in $\S\ref{sec:graded quiver}$.
Recall that $R=R_A$ and $F=F_A$.
For any $R$-module $M$, 
let $M/\tor\subset M\otimes_RF$ be the torsion free part.
The stability in \eqref{stable-triple1} 
does not depend on the variable $a^*$. Forgetting $a^*$ yields a vector bundle
\begin{align}\label{rho}\rho_2:\widetilde\frakM(W)\to\widehat\frakM(W).\end{align}
Since the potential $\bfw_2$ does not depend on $a^*$ either,
we have $\tilde f_2=\hat f_2\circ\rho_2$ for some function $\hat f_2$ on $\widehat\frakM(W)$.
Thus Proposition \ref{prop:Thom} yields an isomorphism
$$K_A(\widehat\frakM(W),\hat f_2)=
K_A(\widetilde\frakM(W),\tilde f_2).$$
We also define the function $\hat f_2^\bullet:\widehat\frakM^\bullet(W)\to\bbC$ as above
using $\tilde f^\bullet_2$.
Let $Q$ be a Dynkin quiver. 
Let $\U_F^{-w}(L\frakg)$ be the $(0,-w)$-shifted quantum loop group defined in \cite{FT19}.
See \S\ref{sec:QG} for details.

\begin{theorem}\label{thm: Nakajima shifted}
Assume that $Q$ is a Dynkin quiver.
\hfill
\begin{enumerate}[label=$\mathrm{(\alph*)}$,leftmargin=8mm]
\item
There is an $F$-algebra map
$\U_F^{-w}(L\frakg)\to 
K_A\big(\widehat\frakM(W)^2,(\hat f_2)^{(2)}\big)_{\widehat\calZ(W)}\otimes_RF$
which takes the central element $\psi^+_{i,0}\,\psi^-_{i,-w_i}$ to $(-q)^{-w_i}\det(W_i)^{-1}$ for each $i\in I$. Hence the $F$-algebra $\U_F^{-w}(L\frakg)$ acts on the $F$-vector spaces
$K_A(\widehat\frakM(W),\hat f_2)_{\widehat\frakL(W)}\otimes_RF$ and
$K_A(\widehat\frakM(W),\hat f_2)\otimes_RF$.
\item
The map in Part $\mathrm{(a)}$ restricts to an $R$-algebra homomorphism
$\U_R^{-w}(L\frakg)\to 
K_A\big(\widehat\frakM(W)^2,(\hat f_2)^{(2)}\big)_{\widehat\calZ(W)}\,/\,\tor$.
Hence $\U_R^{-w}(L\frakg)$ acts on 
$K_A(\widehat\frakM(W),\hat f_2)_{\widehat\frakL(W)}\,/\,\tor$ and
$K_A(\widehat\frakM(W),\hat f_2)\,/\,\tor$.
\item
Let $W\in\bfC^\bullet$ be as in $\S\ref{sec:graded quiver}$.
The map in Part $\mathrm{(b)}$ specializes to an algebra homomorphism
$\U_\zeta^{-w}(L\frakg)\to 
K\big(\widehat\frakM^\bullet(W)^2,(\hat f^\bullet_2)^{(2)}\big)_{\widehat\calZ^\bullet(W)}$.
Hence $\U_\zeta^{-w}(L\frakg)$ acts on 
$K(\widehat\frakM^\bullet(W),\hat f^\bullet_2)_{\widehat\frakL^\bullet(W)}$ and
$K(\widehat\frakM^\bullet(W),\hat f^\bullet_2)$.
\end{enumerate}

\end{theorem}

\begin{proof}
The proof of the theorem is based on the following ingredients :
the compatibility KCA/KHA proved in \S\ref{sec:double},
a reduction to the $Q=A_1$ case as in \cite{N00}, a fixed point computation in the $Q=A_1$ case as in
\cite{VV99}.
We first concentrate on the first claim of Part (a).
We consider the $F$-algebra
$$\calU_F^0=F[\psi_{i,n}^+\,,\,\psi_{i,-w_i-n}^-\,;\,i\in I\,,\,n\in\bbN].$$
The triangular decomposition of the shifted quantum loop group
yields an isomorphism
\begin{align*}
\U_F^{-w}(L\frakg)=\calU_F^+\otimes_F
\calU_F^0\otimes_F\calU_F^-.
\end{align*}
We'll define an $F$-algebra homomorphism
\begin{align}\label{map8}
\U_F^{-w}(L\frakg)\to 
K_A\big(\widehat\frakM(W)^2,(\hat f_2)^{(2)}\big)_{\widehat\calZ(W)}\otimes_RF
\end{align}
The $F$-algebra $\U_F^{-w}(L\frakg)$ is generated by the Fourier coefficients of 
\begin{align*}
x^\pm_i(u)=\sum_{n\in\bbZ}x^\pm_{i,n}\,u^{-n}
 ,\quad
\psi^+_i(u)=\sum_{n\in\bbN}\psi^+_{i,n}\,u^{-n}
 ,\quad
\psi^-_i(u)=\sum_{n\geqslant w_i}\psi^-_{i,-n}\,u^{n}
\end{align*}
modulo the defining relations (A.2) to (A.7) in \S\ref{sec:QG}.
Let $\psi^m$ be the Adams operation in 
$K_A\big(\widehat\frakM(W)\big).$
We consider the classes in $K_A\big(\widehat\frakM(W)\big)$ given by
$$\calH_{i,1}=\calW_i-\sum_j[c_{ij}]_q\calV_j
 ,\quad
\calH_{i,-1}=\calW_i^\vee-\sum_j[c_{ij}]_q\calV_j^\vee
 ,\quad
\calH_{i,\pm m}=\frac{[m]_q}{m}\,\psi^m(\calH_{i,\pm 1})
 $$
Composing the pushforward by the diagonal embedding with the algebra homomorphism 
\begin{align}\label{U2}
\Upsilon:K^A(\widehat\calZ(W))\to K_A\big(\widehat\frakM(W)^2,(\hat f_2)^{(2)}\big)_{\widehat\calZ(W)}
\end{align}
in Corollary \ref{cor:critalg1}, yields the map
\begin{align}\label{delta}
\Delta:K_A\big(\widehat\frakM(W)\big)\to
K_A\big(\widehat\frakM(W)^2,(\hat f_2)^{(2)}\big)_{\widehat\calZ(W)}
\end{align}
We consider the formal series with coefficients in
$K_A\big(\widehat\frakM(W)\big)$ given by
\begin{align}\label{psi+-}
q^{-w_i\pm(\alpha_i,w-\bfc v)}\,\Lambda_{-u^{-1}}(q^{-1}\calW_i)^{-1}
\,\exp\Big(\pm (q-q^{-1})\sum_{m>0}\calH_{i,\pm m}u^{\mp m}\Big).
\end{align}
We assign to the element $\psi_{i,n}^\pm$ in $\U_F^{-w}(L\frakg)$ the image by the map $\Delta$
of the coefficient of $u^{-n}$ 
in the formal series \eqref{psi+-}.
Composing \eqref{twist3} with the map \eqref{UPM} yields the algebra homomorphism
\begin{align*}
\Omega^\pm\Upsilon^\pm:K^{\bbC^\times}(\Rep^0)
\to
K_A(\widehat\frakM(W)^2,(\hat f_2)^{(2)})_{\widehat\calZ(W)}
\end{align*}
Recall the line bundle $\calL_i$ on $\Rep^0_{\delta_i}$ introduced in \S\ref{sec:KHA}.
We define 
$$A_{i,n}^\pm=\Omega^\pm\Upsilon^\pm\big(\calL_i^{\otimes n}\big)
,\quad
A_i^\pm(u)=\sum_{n\in\bbZ}A_{i,n}^\pm u^{-n}.$$
We assign to $x_{i,n}^\pm$ the following element in
$K_A(\widehat\frakM(W)^2,(\hat f_2)^{(2)})_{\widehat\calZ(W)}$
\begin{align}\label{xpm}
\begin{split}
x_{i,n}^+\mapsto A^+_{i,n}\star\det(\calV_{\circ i})(-1)^{v_{+i}}
,\quad
x_{i,n}^-\mapsto
(-1)^{v_{-i}}q^{-1}\,\calL_i^{-v_{\circ i}}\star A^-_{i,n}
\end{split}
\end{align}

To prove that the images of $x_{i,n}^\pm$ and $\psi_{i,n}^\pm$ in
$K_A(\widehat\frakM(W)^2,(\hat f_2)^{(2)})_{\widehat\calZ(W)}$
defined by the assignments \eqref{psi+-} and
\eqref{xpm} give a well-defined morphism \eqref{map8}, we must check that they satisfy the relations 
$\mathrm{(A.2)}$ to $\mathrm{(A.7)}$.
The relations $\mathrm{(A.5)}$ and $\mathrm{(A.7)}$ are already satisfied in $\calU_F^\pm$ 
by Proposition \ref{prop:Omega1}.
The relations $\mathrm{(A.2)}$ and $\mathrm{(A.3)}$ are straightforward.
The relation $\mathrm{(A.4)}$ is easy to check using the formulas \eqref{psi+-}.
We now concentrate on $\mathrm{(A.6)}$.

First, we assume that $Q=A_1$. 
Then $I=\{i\}$ and $\hat f_2=0$. 
Hence, we have 
\begin{align*}
K^A(\widehat\calZ(W))=
K_A(\widehat\frakM(W)^2,(\hat f_2)^{(2)})_{\widehat\calZ(W)}
\end{align*}
We'll omit the vertex $i$ in the notation, e.g., we abbreviate 
$$\calL=\calL_i
 ,\quad
A^\pm_n=A^\pm_{i,n}
 ,\quad
g(u)=g_{ii}(u)
 ,\quad
w=w_i
 ,\quad
v=v_i.$$
Given a variety $X$ with an action of an affine group $G$, we'll say that $X$ satisfies the property ($T$) if 
\begin{itemize}[leftmargin=3mm]
\item[-]
$K^G(X)$ is a free $R_G$-module, 
\item[-] the forgetful morphism 
$K^G(X)\otimes_{R_G}R_H\to K^H(X)$
is an isomorphism for all closed subgroup $H\subset G$.
\end{itemize}

\begin{lemma}\label{lem:(T)} The $G_W\times\bbC^\times$-varieties
$\widehat\frakM(W)$ and $\widehat\calZ(W)$ satisfy the property $(T)$.
\end{lemma}

\begin{proof}
The variety $\widehat\frakM(v,W)$ parametrizes
the conjugacy classes of pairs consisting of a $(v,v)$-matrix $\varepsilon$ and a $w$-tuple of generators
of $\bbC^v$ for the $\varepsilon$-action. 
In other words, $\widehat\frakM(v,W)$ is isomorphic to the Quot scheme
$\Quot_\bbC(W\otimes\calO, v)$ parametrizing length $v$-quotients of the trivial vector bundle
$W\otimes\calO$ over $\bbC$.
The group $G_W$ acts on $W$ in the obvious way, and 
$\bbC^\times$ dilates both the framing and $\varepsilon$.
The variety $\widehat\frakM(v,W)$ is smooth. Fix a basis of $W$.
Let $W=\bigoplus_{r=1}^wW_r$ be the corresponding decomposition of $W$ as a sum of lines.
Let $T_W\subset G_W$ be the diagonal maximal torus.
Let $\lambda:\bbC^\times\to T_W$ be the cocharacter $z\mapsto(z,z^2,\dots,z^w)$.
The $T_W$-fixed point locus is the disjoint union of the varieties
$$\Quot_\bbC(W\otimes\calO, \bfv)=
\prod_{r=1}^w\Quot_\bbC(W_r\otimes\calO, v_r)=\prod_{r=1}^w\bbC^{[v_r]}=\bbC^v$$
where $\bfv=(v_1,v_2,\cdots,v_w)$ runs into the set of tuples in $\bbN^w$ with sum $v$,
and $\bbC^{[v_r]}$ is the $v_r$-fold symmetric product of $\bbC$.
The closed embedding $\Quot_\bbC(W\otimes\calO, \bfv)\subset\Quot_\bbC(W\otimes\calO, v)$
is the direct sum of $\calO$-modules.
The Byalinicki-Birula theorem yields a $T_W\times\bbC^\times$-equivariant stratification
\begin{align}\label{cells}
\Quot_\bbC(W\otimes\calO, v)=\bigsqcup_\bfv\Quot_\bbC(W\otimes\calO, \bfv)^+
\end{align}
where $\Quot_\bbC(W\otimes\calO, \bfv)^+$ is an affine fiber bundle over 
$\Quot_\bbC(W\otimes\calO, \bfv)$ of relative dimension $\sum_{r=1}^w(r-1)v_r$.
See \cite[prop.~3.4]{MR22} for more details.
This yields a $T_W\times\bbC^\times$-equivariant stratification 
of $\widehat\frakM(v,W)$ by affine cells,
and the first claim of the lemma follows using \cite[thm.~6.1.22]{CG}.

The proof of the second claim is similar. Recall that 
$$\widehat\calZ(W)=\widehat\frakM(W)\times_{\widehat\frakM_0(W)}\widehat\frakM(W)$$
and that $\widehat\frakL(W)$ is the central fiber of the map
$\pi:\widehat\frakM(W)\to\widehat\frakM_0(W)$.
The isomorphism $\widehat\frakM(v,W)=\Quot_\bbC(W\otimes\calO, v)$ identifies
$\widehat\frakL(v,W)$ with the punctual Quot scheme $\Quot_\bbC(W\otimes\calO, v)_0$
consisting of the sheaves supported at $0$.
Intersecting the cell decomposition \eqref{cells} with $\Quot_\bbC(W\otimes\calO, v)_0$
yields an affine cell decomposition 
\begin{align*}
\Quot_\bbC(W\otimes\calO, v)_0=\bigsqcup_\bfv\Quot_\bbC(W\otimes\calO, \bfv)^+_0
\end{align*}
such that $\Quot_\bbC(W\otimes\calO, \bfv)^+$ is an affine fiber bundle over
$\Quot_\bbC(W\otimes\calO, \bfv)^+_0$ for each tuple $\bfv$.
We deduce that $\widehat\calZ(W)$ has a $T_W\times\bbC^\times$-equivariant 
affine cell decomposition whose cells
are affine fiber bundles over the cells
$$\Quot_\bbC(W\otimes\calO, \bfv_1)^+_0\times\Quot_\bbC(W\otimes\calO, \bfv_2)^+_0
\subset\widehat\frakL(W)\times\widehat\frakL(W)$$
for each pair of tuples $(\bfv_1,\bfv_2)$ as above.
\end{proof}

Next, to compute the relations between $A_n^+$ and $A_n^-$ we claim that it is enough
to compute their actions on
$K^{G_W\times\bbC^\times}(\widehat\frakM(W))\otimes_{R_{G_W\times\bbC^\times}}F_{G_W\times\bbC^\times}$ because Lemma \ref{lem:(T)} and the localization theorem in K-theory yield
the following commutative diagram of algebras
\begin{align*}
\xymatrix{
K^{G_W\times\bbC^\times}(\widehat\calZ(W))\ar@{^{(}->}[r]\ar@{->>}[d]&
K^{G_W\times\bbC^\times}(\widehat\calZ(W))\otimes_{R_{G_W\times\bbC^\times}}F_{G_W\times\bbC^\times}
\ar@{=}[d]\\
K^A(\widehat\calZ(W))&
\End_{F_{G_W\times\bbC^\times}}(K^{G_W\times\bbC^\times}(\frakM(W))\otimes_{R_{G_W\times\bbC^\times}}F_{G_W\times\bbC^\times})}
\end{align*}
To do that, let $T_W\subset G_W$ be a maximal torus.
The $T_W\times\bbC^\times$-fixed points locus is
$$\widehat\frakM(v,W)^{T_W\times\bbC^\times}=
\{\underline x_\lambda\,;\,\lambda\in\bbN^w\,,\,|\lambda|=v\}$$ 
where $|\lambda|=\sum_{s=1}^w\lambda_s$ is the weight of the $w$-tuple $\lambda=(\lambda_1,\lambda_2,\dots,\lambda_w)$.
Let 
$[\lambda]$ be the fundamental class
of $\{\underline x_\lambda\}$.
For any linear operator $A$, 
let $\langle \lambda|A|\mu\rangle$ be the coefficient of the basis element $[\lambda]$ 
in the expansion of $A[\mu]$ in the
basis $\{[\lambda]\,;\,\lambda\in\bbN^w\}$.
Recall the tautological vector bundles
$\calV^+$, $\calV^-$ and $\calL=\calV^+/\,\calV^-$ on the Hecke correspondences
and on $\widehat\calZ(W)$.
Let $\lambda$ and $\mu$ be $w$-tuples of weight $v$ and $v+1$.
We abbreviate
$\calV_\lambda=\calV|_{\{\underline x_\lambda\}}$, 
$\calL_{\lambda,\mu}=\calL|_{\{(\underline x_\lambda,\,\underline x_\mu)\}}$, etc.
By \cite[\S4.5]{VV99} we have
\begin{align*}
\langle \lambda|A^-_n|\mu\rangle
&=(\calL_{\mu,\lambda})^{\otimes n}\otimes
\Lambda_{-1}\Big(T_{\mu}\widehat\frakM(v+1,W)-T_{\mu,\lambda}\widehat\frakP(v+1,v,W)\Big)\\
\langle \mu|A^+_m|\lambda\rangle
&=(\calL_{\lambda,\mu})^{\otimes m}\otimes
\Lambda_{-1}\Big(T_\lambda\widehat\frakM(v,W)-T_{\lambda,\mu}\widehat\frakP(v,v+1,W)\Big)
\end{align*}
The class of $T\widehat\frakM(W)$ in the Grothendieck group of $\widehat\frakM(W)$ is
\begin{align}\label{C1}
T\widehat\frakM(W)=(q^{-2}-1)\End(\calV)+q\Hom(\calV,\calW).
\end{align}
The class of $T\widehat\frakP(W)$ in the Grothendieck group of $\widehat\frakP(W)$ is
\begin{align}\label{C2}
T\widehat\frakP(W)=(q^{-2}-1)\calP+q\Hom(\calV^+,\calW)
\end{align}
with
$
\calP=\End(\calV^-)+\Hom(\calL,\calV^+)=\End(\calV^+)-\Hom(\calV^-,\calL).$
We write 
\begin{align}\label{form10}
\calV=\sum_{r=1}^vz_r\in R_{G_v}
 ,\quad
\calW=\sum_{s=1}^w\chi_s\in 
R_{G_W\times\bbC^\times}
\end{align} 
where $z_1,\dots, z_v$ and $\chi_1,\dots,\chi_w$ are the fundamental characters of maximal tori
in $G_v$ and $G_W$.
Specializing at the points $\ux_\lambda$ and $\ux_\mu$, we get the following classes in $R_{G_W\times\bbC^\times}$
\begin{align}\label{form7}
\begin{split}
\calV_\lambda=\sum_{s=1}^w\sum_{r=1}^{\lambda_s}\chi_sq^{3-2r}
 ,\quad
\calL_{\lambda,\mu}=z_{v+1}.
\end{split}
\end{align} 
Fix positive integers  $s_0,$ $r_0$ with $s_0\leqslant w$ such that 
$\calV_\mu-\calV_\lambda=z_{v+1}=\chi_{s_0}q^{3-2r_0}.$
The matrix coefficients $\langle \lambda |A^-_n|\mu\rangle$ and $\langle \mu|A^+_m|\lambda\rangle$
in $F_{G_W\times\bbC^\times}$ are given by
\begin{align*}
\langle \lambda |A^-_n|\mu\rangle
&=(\calL_{\lambda,\mu})^{\otimes n}\otimes
\Lambda_{-1}\Big((q^{-2}-1)\calV^\vee_\lambda\otimes\calL_{\lambda,\mu}\Big)\\
&=\ev_{u=z_{v+1}}\Big(u^{n}\prod_{r=1}^v\frac{uq^{-2}-z_r}{u-z_r}\Big)\\
\langle \mu|A^+_m|\lambda\rangle
&=(\calL_{\lambda,\mu})^{\otimes m}\otimes\Lambda_{-1}
\Big((1-q^{-2})\otimes\calL^\vee_{\lambda,\mu}\otimes\calV_\mu-q\calL^\vee_{\lambda,\mu}\otimes\calW\Big)\\
&=(1-q^{-2})^{-1}\Res_{u=z_{v+1}}\Big(\frac{u^{m+w-1}}{\prod_{s=1}^w(u-\chi_sq)}\prod_{r=1}^{v}\frac{u-z_r}{u-z_rq^{-2}}\Big)
\end{align*}
We consider the rational function
$\phi_\lambda(u)\in F_{G_W\times\bbC^\times}$ such that
\begin{align}\label{h}
\begin{split}
\phi_\lambda(u)=u^w\frac{\prod_{s=1}^w\prod_{r=1}^{\lambda_s}g(u/\chi_sq^{3-2r})}{\prod_{s=1}^w(u-\chi_s q)}.
\end{split}
\end{align}
We have
$$\phi_\lambda(u)=
q^{-2v}u^w\prod_{s=1}^w\frac{u-\chi_sq^3}{(u-\chi_sq^{1-2\lambda_s})(u-\chi_sq^{3-2\lambda_s})}
.$$
The poles of 
$\phi_\lambda(u)$ belong to the set 
$\{\chi_sq^{1-2\lambda_s}\,,\,\chi_sq^{3-2\lambda_s}\,;\,s\in[1,w]\}$.
So the residue theorem yields
\begin{align*}
(q-q^{-1})\langle\lambda|[A^+_m,A^-_n|\lambda\rangle
&=
-q\sum_{s=1}^w\Res_{u=\chi_sq^{3-2\lambda_s}}\Big(u^{m+n-1}\phi_\lambda(u)\Big)+\\
&\qquad+\Res_{u=\chi_sq^{1-2\lambda_s}}\Big(u^{m+n-1}\phi_\lambda(u)\Big)\\
&=
q\Res_{u=0}\Big(u^{m+n-1}\phi_\lambda(u)\Big)+
q\Res_{u=\infty}\Big(u^{m+n-1}\phi_\lambda(u)\Big)
\end{align*}
Let $\phi^\pm_\lambda(u)$ be the expansion of  $\phi_\lambda(u)$
in non negative powers of $u^{\mp 1}$. 
The matrix coefficient
$(q-q^{-1})\,\langle \lambda|[A^+_m,A^-_n]|\mu\rangle$ 
is equal to the Kronecker symbol $\delta_{\lambda,\mu}$ 
times the coefficient of $u^{-m-n}$ in the formal series 
$-q\phi^+_\lambda(u)+q\phi^-_\lambda(u)$.
Now, let $\psi^\pm(u)$ be the formal series of operators on
$K^{G_W\times\bbC^\times}(\widehat\frakM(v,W))$ which act
by multiplication by the Fourier coefficients of the
expansions in non negative powers of $u^{\mp 1}$ of the following rational function in
$F_{G_v\times G_W\times\bbC^\times}$
\begin{align*}
\begin{split}
\psi(u)=u^w\frac{\prod_{r=1}^vg(u/z_r)}{\prod_{s=1}^w(u-\chi_s q)}.
\end{split}
\end{align*}
The upperscript $\pm$ holds for the expansion in non negative powers of $u^{\mp 1}$.
We have
$(q-q^{-1})[x^+(u)\,,\,x^-(v)]=\delta(u/v)\,(\psi^+(u)-\psi^-(u))$
with
\begin{align}\label{psipmi}
\begin{split}
\psi^+(u)&=q^{-w_i}q^{(\alpha_i,w-\bfc v)}
\Lambda_{-u^{-1}}\big((q^2-q^{-2})\calV_i-q\calW_i\big)^+,
\\
\psi^-(u)&=(-u)^{w_i}q^{-(\alpha_i,w-\bfc v)}\det(\calW_i)^{-1}
\Lambda_{-u}\big((q^{-2}-q^2)\calV_i^\vee-q^{-1}\calW_i^\vee\big)^-
\end{split}
\end{align}

Now, let $Q$ be any Dynkin quiver.
First, we prove the relation (A.6) for $i=j$. 
To do this, we'll use a reduction to the case $A_1$,
which is proved above, similar to the one used in proof of \cite[\S11.3]{N00}.
Fix a vertex $i\in I$. Consider the subquiver 
$\widehat Q_{f,\neq i}$ of $\widehat Q_f$ such that
$$(\widehat Q_{f,\neq i})_0=(\widehat Q_f)_0\,\setminus\,\{i,i'\}
 ,\quad
(\widehat Q_{f,\neq i})_1=\{h\in (\widehat Q_f)_1\,;\,s(h),t(h)\neq i\}.$$
The representation variety of $\widehat Q_f$ decomposes as
$$
\widehat\X(V,W)=\widehat\X(V_i,W_i\oplus V_{\circ i})\times
\Hom(V_{\circ i},V_i)\times\X_{\widehat Q_{f,\neq i}}(V_{\neq i},W_{\neq i})
$$
where 
$$V=\bbC^v
 ,\quad
V_{\circ i}=\bigoplus_{j\neq i}(V_j)^{\oplus (-c_{ij})}
 ,\quad
V_{\neq i}=\bigoplus_{j\neq i}V_j
 ,\quad
W_{\neq i}=\bigoplus_{j\neq i}W_j.$$
We define
$\frakM(v,W)_\diamondsuit=\X(V,W)_\diamondsuit\,/\,G_{V_i},$
$\frakM(v,W)_\heartsuit=\X(V,W)_\heartsuit\,/\,G_{V_i}$ and
$\frakM(v,W)_\spadesuit=\widehat\X(V,W)_s\,/\,G_{V_i}$
where
\begin{align*}
\X(V,W)_\diamondsuit&=\widehat\X(V_i,W_i\oplus V_{\circ i})_s\times
\Hom(V_{\circ i},V_i),\\
\X(V,W)_\heartsuit&=\X(V,W)_\diamondsuit\times\X_{\widehat Q_{f,\neq i}}(V_{\neq i},W_{\neq i})
\end{align*}
We consider the diagram
$$\xymatrix{\widehat\frakM(v_i,W_i\oplus V_{\circ i})&\ar[l]_-\rho
\frakM(v,W)_\diamondsuit&\ar[l]_p\frakM(v,W)_\heartsuit&\ar[l]_-{\iota}\frakM(v,W)_\spadesuit\ar[r]^\pi
&\widehat\frakM(v,W)}$$
where $\widehat\frakM(v_i,W_i\oplus V_{\circ i})$ is the quiver variety of type $A_1$.
The map $\rho$ is the first projection. 
It is the vector bundle given by forgetting the arrow $\alpha_{ij}$ for all $j\in I$. 
The map $p$ is the first projection,
$\iota$ is an open embedding, and $\pi$ is a principal bundle.
Let $v=(v_i,v_{\neq i})$ with $v_{\neq i}$ fixed and $v_i$ running in $\bbN\{i\}\simeq\bbN$. Set
$$\widehat\frakM(W_i\oplus V_{\circ i})=\bigsqcup_{v_i\in\bbN}\widehat\frakM(v_i,W_i\oplus V_{\circ i})
 ,\quad
\frakM(W)_\flat=\bigsqcup_{v_i\in\bbN}\frakM(v,W)_\flat
,\quad
\flat=\diamondsuit,\heartsuit,\spadesuit.$$
The subvarieties $\frakP(W)_\diamondsuit,\calZ(W)_\diamondsuit\subset\frakM(W)_\diamondsuit^2$
are defined in the obvious way.
We define $\frakP(W)_\heartsuit\subset\frakM(W)_\heartsuit^2$ to be the product of
$\frakP(W)_\diamondsuit$ 
and the diagonal of the affine space 
$\X_{\widehat Q_{f,\neq i}}(V_{\neq i},W_{\neq i})$. 
The map $\iota$ satisfies the condition \cite[(11.2.1)]{N00},
and $\pi$ the condition \cite[(11.2.9)]{N00}.
Hence we can apply the argument in \cite[\S 11.3]{N00}.
We get an algebra homomorphism
$$K^{G_{W_i}\times G_{V_{\circ i}}\times\bbC^\times}\big(\calZ(W)_\diamondsuit\big)
\to
K^{G_W\times\bbC^\times}\big(\widehat\calZ(W)\big).$$
Composing it with \eqref{U2} yields an algebra homomorphism 
$$K^{G_{W_i}\times G_{V_{\circ i}}\times\bbC^\times}\big(\calZ(W)_\diamondsuit\big)
\to K_{G_W\times\bbC^\times}\big(\widehat\frakM(W)^2,(\hat f_2)^{(2)}\big)_{\widehat\calZ(W)}$$
Hence, we are reduced to prove the relation  (A.6) in the left hand side.
We'll prove it as above, using the action of
$K^{G_{W_i}\times G_{V_{\circ i}}\times\bbC^\times}(\calZ(W)_\diamondsuit)$
on $K^{G_{W_i}\times G_{V_{\circ i}}\times\bbC^\times}(\frakM(W)_\diamondsuit)$.
To do that, we use the following formulas in the Grothendieck groups, compare \eqref{C1} and \eqref{C2},
\begin{align*}
\begin{split}
T\frakM(W)_\diamondsuit&=(q^{-2}-1)\End(\calV_i)+q\Hom(\calV_i,\calW_i\oplus\calV_{\circ i})
+q\Hom(\calV_{\circ i},\calV_i)\\
T\frakP(W)_\diamondsuit&=(q^{-2}-1)\calP_i+q\Hom(\calV^+_i,\calW_i\oplus\calV_{\circ i})
+q\Hom(\calV_{\circ i},\calV_i^-)\\
\calP_i&=\End(\calV^-_i)+\Hom(\calL_i,\calV^+_i).
\end{split}
\end{align*}
Arguing as in the case of $Q=A_1$, we prove the following relations
\begin{align*}
\begin{split}
\langle \lambda |A^-_{i,n}|\mu\rangle
&=(\calL_{i,\lambda,\mu})^{\otimes n}\otimes
\Lambda_{-1}\Big((q^{-2}-1)\calV^\vee_{i,\lambda}\otimes\calL_{i,\lambda,\mu}+
q\calV_{\circ i}^\vee\otimes\calL_{i,\lambda,\mu}\Big)\\
\langle \mu|A^+_{i,m}|\lambda\rangle
&=(\calL_{i,\lambda,\mu})^{\otimes m}\otimes\Lambda_{-1}
\Big((1-q^{-2})\otimes\calL^\vee_{i,\lambda,\mu}\otimes\calV_{i,\mu}-
q\calL^\vee_{i,\lambda,\mu}\otimes(\calW_i\oplus\calV_{\circ i})\Big).
\end{split}
\end{align*}
We deduce that
Let $v_\lambda$ be the rank of $\calV_\lambda$. We deduce that
\begin{align*}
(1-q^{-2})\langle\lambda| A^-_{i,n}A^+_{i,m}|\lambda\rangle
&=(-1)^{v_{\circ i}}q^{-(\alpha_i,\bfc v_\lambda)}\det\big(\calV^\vee_{\circ i}\big)\sum_\mu
\Res_{u=\calV_\mu/\calV_\lambda}\Big(u^{m+n-1+v_{\circ i}}
\\
&\quad\Lambda_{-u^{-1}}\big((q-q^{-1})\sum_j[c_{ij}]_{q}\calV_{j,\lambda}
-q\calW_i\big)
\Big),\\
(1-q^{-2})\langle\lambda| A^+_{i,m}A^-_{i,n}|\lambda\rangle
&=(-1)^{1+v_{\circ i}}q^{-(\alpha_i,\bfc v_\lambda)}\det\big(\calV^\vee_{\circ i}\big)\sum_\mu
\Res_{u=\calV_\lambda/\calV_\mu}\Big(u^{m+n-1+v_{\circ i}}
\\
&\quad\Lambda_{-u^{-1}}\big((q-q^{-1})\sum_j[c_{ij}]_{q}\calV_{j,\lambda}-q\calW_i\big)
\Big).
\end{align*}
The sums are over all $\mu$'s such that $\ux_\lambda\subset\ux_\mu$ and 
$\ux_\mu\subset\ux_\lambda$
are of codimention $\delta_i$ respectively.
Using the residue theorem, we get
\begin{align*}
(q-q^{-1})\langle\lambda|[A^+_{i,m},A^-_{i,n}]|\lambda\rangle
=
-\Res_{u=0}\Big(u^{m+n-1}\phi_{i,\lambda}(u)\Big)
-\Res_{u=\infty}\Big(u^{m+n-1}\phi_{i,\lambda}(u)\Big)
\end{align*}
where
$$\phi_{i,\lambda}(u)=
(-u)^{v_{\circ i}}
q^{1-(\alpha_i^\vee,\bfc v_\lambda)}
\det\big(\calV_{\circ i}\big)^{-1}
\Lambda_{-u^{-1}}\Big((q-q^{-1})\sum_j[c_{ij}]_{q}\calV_{j,\lambda}-q\calW_i\Big)
$$
Similarly, given $\lambda$, $\lambda'$ such that
$\calV_\lambda\cap\calV_{\lambda'}$ is of codimension one in  
$\calV_\lambda$ and in $\calV_{\lambda'}$, we get
\begin{align*}
\langle\lambda'| A^-_{i,n}A^+_{i,m}|\lambda\rangle
&=\langle\lambda'| A^-_{i,n}|\mu\rangle\langle\mu|A^+_{i,m}|\lambda\rangle\\
&=(\calL_{i,\lambda',\mu})^{\otimes n}\otimes(\calL_{i,\lambda,\mu})^{\otimes m}\otimes
\Lambda_{-1}\big((1-q^{-2})(\calL_{i,\lambda,\mu}^\vee\otimes\calV_{i,\mu}
-\calL_{i,\lambda',\mu}\otimes\calV_{i,\lambda'}^\vee)-
q\calL_{i,\lambda,\mu}^\vee\otimes\calW_i\\
&\quad
+\sum_{c_{ij}<0}q^{-c_{ij}}(\calL_{i,\lambda',\mu}\otimes\calV_j^\vee-\calL_{i,\lambda,\mu}^\vee\otimes\calV_j)
\big)\\
\langle\lambda'| A^+_{i,m}A^-_{i,n}|\lambda\rangle
&=\langle\lambda'| A^+_{i,m}|\nu\rangle\langle\nu|A^-_{i,n}|\lambda\rangle\\
&=(\calL_{i,\nu,\lambda})^{\otimes n}\otimes(\calL_{i,\nu,\lambda'})^{\otimes m}\otimes
\Lambda_{-1}\big((1-q^{-2})(\calL_{i,\nu,\lambda'}^\vee\otimes\calV_{i,\lambda'}
-\calL_{i,\nu,\lambda}\otimes\calV_{i,\nu}^\vee)-
q\calL_{i,\nu,\lambda'}^\vee\otimes\calW_i\\
&\quad
+\sum_{c_{ij}<0}q^{-c_{ij}}(\calL_{i,\nu,\lambda}\otimes\calV_j^\vee-\calL_{i,\nu,\lambda'}^\vee\otimes\calV_j)
\big)
\end{align*}
where $\mu$, $\nu$ are such that 
$\calV_\mu=\calV_\lambda+\calV_{\lambda'}$ and $\calV_\nu=\calV_\lambda\cap\calV_{\lambda'}$.
Let $\phi^\pm_{i,\lambda}(u)$ be the expansion of  $\phi_{i,\lambda}(u)$
in non negative powers of $u^{\mp 1}$. 
We deduce that 
$$(q-q^{-1})\,\langle \lambda'|[A^+_i(u),A^-_i(v)]|\lambda\rangle=
\delta_{\lambda,\lambda'}\delta(u/v)(\phi^+_{i,\lambda}(u)-\phi^-_{i,\lambda}(u)).$$
Let $\phi^\pm_i(u)$ be the formal series of operators acting on
$K^{G_W\times\bbC^\times}(\widehat\frakM(v,W))$ 
by multiplication by the Fourier coefficients of the
expansions in non negative powers of $u^{\mp 1}$ of the following rational function
$$\phi_i(u)=
(-1)^{v_{\circ i}}q^{1-(\alpha_i^\vee,\bfc v)}
\det\big(\calL_i^\vee\otimes\calV_{\circ i}\big)^{-1}
\,\Lambda_{-u^{-1}}(q^{-1}\calW_i)^{-1}
\Lambda_{-u^{-1}}\big(-(q-q^{-1})\calH_{i,1}\big)$$
We have
\begin{align}\label{Apm}
(q-q^{-1})\,[A^+_i(u),A^-_i(v)]=
\delta(u/v)(\phi^+_i(u)-\phi^-_i(u)).
\end{align}
We define similarly 
$$\psi_i(u)=
q^{-(\alpha_i^\vee,\bfc v)}
\,\Lambda_{-u^{-1}}(q^{-1}\calW_i)^{-1}
\Lambda_{-u^{-1}}\big(-(q-q^{-1})\calH_{i,1}\big)$$
Then, we have
$$(q-q^{-1})[x^+_i(u)\,,\,x^-_i(v)]=\delta(u/v)\,(\psi^+_i(u)-\psi^-_i(u)).$$
Note that
\begin{align*}
\begin{split}
\psi_i^\pm(u)&=q^{- w_i}q^{\pm(\alpha_i^\vee,w-\bfc v)}
\,\Lambda_{-u^{-1}}(q^{-1}\calW_i)^{-1}
\,\Lambda_{-u^{\mp 1}}\big(\mp(q-q^{-1})\calH_{i,\pm 1}\big)^\pm.
\end{split}
\end{align*}
Further, we have the following relation between wedges and Adams operations 
\begin{align}\label{WA}
\Lambda_{-u}(\calE)=\exp\Big(-\sum_{m>0}\psi^m(\calE)u^m/m\Big).
\end{align}
We deduce that the series $\psi_i^\pm(u)$ above coincide with the series in \eqref{psi+-},
proving the relation (A.6) with $i=j$.
Note that 
\begin{align}\label{CT}\psi^+_{i,0}=q^{-w_i+(\alpha_i^\vee,w-\bfc v)}
 ,\quad
\psi^-_{i,-w_i}=(-1)^{w_i}q^{-(\alpha_i^\vee,w-\bfc v)}\det(\calW_i)^{-1}.
\end{align}

Finally, we prove the relation (A.6) for $i\neq j$.
The proof of Proposition \ref{prop:double} yields
the commutative diagram of $R$-algebra homomorphisms 
 \begin{align}\label{form4}
\begin{split}
\xymatrix{
K^{\bbC^\times}(\Rep^0)^\pm\otimes R_{G_W}\ar[d]_-{\omega^\pm}\ar[r]^-{\Upsilon^\pm\otimes 1}&
K_{\bbC^\times}(\bfw)^\pm\ar[d]^-{\omega^\pm}\otimes R_{G_W}\\
K^A(\widehat\calZ(W))\ar[r]^-\Upsilon&
K_A(\widehat\frakM(W)^2,(\hat f_2)^{(2)})_{\widehat\calZ(W)}
}
\end{split}
\end{align}
The function $h:\Rep\to\bbC$ in \S\ref{sec:potential} vanishes on the substacks $\Rep_{\delta_i}$ and $\Rep_{\delta_j}$.
Hence, by \eqref{form4}, the elements $x_{i,m}^+$ and $x_{j,n}^-$ defined in \eqref{xpm}
have obvious liftings in $K^A(\widehat\calZ(W))$ and
it is enough to check that these liftings 
commute with each other.
This follows from the transversality result in Lemma \ref{lem:transverse} below,
which is analogous to 
\cite[lem.~10.2.1]{N00} and \cite[lem.~9.8\,,\,9.9\,,\,9.10]{N98}.
Set $v_2=v_1+\delta_i=v_3+\delta_j$
and
$v_4=v_1-\delta_j=v_3-\delta_i.$
We consider the intersections
\begin{align*}
&I_{v_1,v_2,v_3}=\big(\widehat\frakP(v_1,v_2,W)\times\widehat\frakM(v_3,W)\big)\cap
\big(\widehat\frakM(v_1,W)\times\widehat\frakP(v_2,v_3,W)\big),\\
&I_{v_1,v_4,v_3}=\big(\widehat\frakP(v_1,v_4,W)\times\widehat\frakM(v_3,W)\big)\cap
\big(\widehat\frakM(v_1,W)\times\widehat\frakP(v_4,v_3,W)\big)
\end{align*}

\begin{lemma}\label{lem:transverse}
\hfill
\begin{enumerate}[label=$\mathrm{(\alph*)}$,leftmargin=8mm]
\item 
The intersections $I_{v_1,v_2,v_3}$ and $I_{v_1,v_4,v_3}$ are both transversal
in $\widehat\frakM(W)^3$.
\item
There is a $G_W\times\bbC^\times$-equivariant isomorphism 
$I_{v_1,v_2,v_3}\simeq I_{v_1,v_4,v_3}$ which intertwines the sheaves
$(\calL_i\otimes\calO)|_{I_{v_1,v_2,v_3}}$ and $(\calO\otimes\calL_i)|_{I_{v_1,v_4,v_3}}$,
and the sheaves
$(\calO\boxtimes\calL_j)|_{I_{v_1,v_2,v_3}}$ and $(\calL_j\boxtimes\calO)|_{I_{v_1,v_4,v_3}}$.
\end{enumerate}
\end{lemma}

\begin{proof}
We first prove that the intersection $I_{v_1,v_2,v_3}$
is transversal at any point $(\ux_1,\ux_2,\ux_3)$. 
Let $\pi_i$ be the projection of
$\widehat\frakM(v_1,W)\times\widehat\frakM(v_2,W)\times\widehat\frakM(v_3,W)$
to the $i$th factor along the other ones.
Set $\ux_{12}=(\ux_1,\ux_2)$ and $\ux_{23}=(\ux_2,\ux_3)$.
The Hecke correspondences
$\widehat\frakP(v_1,v_2,W)$ and $\widehat\frakP(v_2,v_3,W)$ are smooth.
Set
$$W_1=(d_{\ux_{12}}\pi_2)\big(\widehat\frakP(v_1,v_2,W)\big)
 ,\quad
W_3=(d_{\ux_{23}}\pi_2)\big(\widehat\frakP(v_2,v_3,W)\big).
$$
We claim that
$W_1+W_3=T_{\ux_2}\widehat\frakM(v_2,W).$
The tangent space of $\widehat\frakM(v_2,W)$ is
$$T_{\ux_2}\widehat\frakM(v_2,W)=\widehat\X(v_2,W)\,/\,\frakg_{v_2}\cdot \ux_2$$
and the tangent spaces of the Hecke correspondences are
\begin{align*}
T_{\ux_{12}}\widehat\frakP(v_1,v_2,W)=\widehat\X(v_1,v_2,W)\,/\,\frakp_{v_1,v_2}\cdot \ux_{12}
 ,\quad
T_{\ux_{23}}\widehat\frakP(v_2,v_3,W)=\widehat\X(v_2,v_3,W)\,/\,\frakp_{v_2,v_3}\cdot \ux_{23}
\end{align*}
where $\widehat\X(v_1,v_2,W)$ is the subspace of $\widehat\X(v_2,W)$ given by
\begin{align*}
\widehat\X(v_1,v_2,W)=\{y\in\widehat\X(v_2,W)\,;\,y(\bbC^{v_1}\oplus W)
\subseteq\bbC^{v_1}\oplus W\},
\end{align*}
and $\widehat\X(v_2,v_3,W)\subset\widehat\X(v_2,W)$ is defined similarly.
It is enough to prove that
$$\pi_2(\widehat\X(v_1,v_2,W))+\pi_2(\widehat\X(v_2,v_3,W))=\widehat\X(v_2,W).$$
To prove this recall that $i\neq j$. 
Hence we have 
$$\bbC^{v_2}=\bbC^{\delta_j}\oplus(\bbC^{v_1}\cap\bbC^{v_3})\oplus\bbC^{\delta_i}.$$
Let $p_1,p_3\in\End(\bbC^{v_2})$ be the projection along $\bbC^{\delta_i}$ and $\bbC^{\delta_j}$ respectively,
onto the other summands.
Fix any tuple $x_2=(\alpha_2,a_2,\varepsilon_2)\in\widehat\X(v_2,W)$.
We define $x_1=(\alpha_1,a_1,\varepsilon_1)$ and $x_3=x_2-x_1$ with 
$$\alpha_1=p_1\alpha_2+(1-p_1)\alpha_2(1-p_1)
,\quad
a_1=a_2p_1
,\quad
\varepsilon_1=p_1\varepsilon_2+(1-p_1)\varepsilon_2(1-p_1)$$
It is not difficult to see that
$$x_1\in\pi_2(\widehat\X(v_1,v_2,W))
,\quad
x_3\in\pi_2(\widehat\X(v_2,v_3,W)).$$
The transversality of
$I_{v_1,v_4,v_3}$
can be proved in a similar way.
Next, we concentrate on Part (b). Let $\Gr(\delta_i,V)$ be the grassmannian of codimension
$\delta_i$ $I$-graded subspaces. We have
\begin{align*}
I_{v_1,v_2,v_3}=&\{(S_1,S_3,x)\,;\,x(S_1)\subset S_1\,,\,x(S_3)\subset S_3\}\,/\,G_{v_2}\\
I_{v_1,v_4,v_3}=&\{(T_1,T_3,x_1,x_3,\phi)\,;\,
x_1(T_1)\subset T_1\,,\,x_3(T_3)\subset T_3\,,\,\phi\circ x_1|_{T_1}=x_3|_{T_3}\circ\phi\}\,/\,
G_{v_1}\times G_{v_3}.
\end{align*}
where $(S_1,S_3,x)\in\Gr(\delta_i,V)\times\Gr(\delta_j,V)\times
\widehat\X(V_2,W)_s$ and
\begin{align*}
(T_1,T_3,x_1,x_3,\phi)&\in\Gr(\delta_j,V_1)\times\Gr(\delta_i,V_3)\times
\widehat\X(v_1,W)_s\times\widehat\X(v_3,W)_s\times\Isom(T_1,T_3).
\end{align*}
The isomorphism $I_{v_1,v_2,v_3}\simeq I_{v_1,v_4,v_3}$ is given by 
\begin{align*}
(S_1\,,\,S_3\,,\,x)&\mapsto(S_1\cap S_3\,,\,S_1\cap S_3\,,\,x|_{S_1}\,,\,x|_{S_3}\,,\,\id_{S_1\cap S_3})\\
(T_1\,,\,T_3\,,\,x_1\,,\,x_3\,,\,\phi)&\mapsto(V'_1\,,\,V'_3\,,\,x')
\end{align*}
where $V'_2=V_1\oplus V_3/(\id\times\phi)(T_1)$, the subspaces
$V'_1, V'_3\subset V'_2$ are the images of $V_1$, $V_3$ in
$V'_2$, and
$x'$ is the image of $x_1\oplus x_3$ in $\widehat\X(V'_2,W)$.
Note that $x|_{S_1}$, $x|_{S_3}$ and $x'$ are stable.
\end{proof}

We have proved the relation (A.6).
The second claim of the part (a) of the theorem follows from the formula \eqref{CT}. 
To prove the part (b) we must check that the morphism \eqref{map8} restricts to a map
$$\U_R^{-w}(L\frakg)\to 
K_A\big(\widehat\frakM(W)^2,(\hat f_2)^{(2)}\big)_{\widehat\calZ(W)}\,/\,\tor$$
By \eqref{LGR} the $R$-subalgebra $\U_R^{-w}(L\frakg)$ of $\U_F^{-w}(L\frakg)$ is generated by 
$$\psi_{i,\mp w^\pm_i}^\pm 
 ,\quad
(\psi^\pm_{i,\mp w_i^\pm})^{-1}
 ,\quad
h_{i,\pm m}/[m]_q
 ,\quad
(x^\pm_{i,n})^{[m]}$$
with $
i\in I$, $n\in\bbZ$, $m\in\bbN^\times$, and, by
\eqref{psi+-}, \eqref{xpm} and Proposition \ref{prop:Omega1},
the map \eqref{map8} takes these elements into 
$K_{G_W\times\bbC^\times}\big(\widehat\frakM(W)^2,(\hat f_2)^{(2)}\big)_{\widehat\calZ(W)}\,/\,\tor$.
The part (c) of the theorem follows from the part (b) and Proposition \ref{prop:Thomason}.
\end{proof}

\begin{remark}
\hfill
\begin{enumerate}[label=$\mathrm{(\alph*)}$,leftmargin=8mm]
\item The proof of the theorem yields also an $F$-algebra homomorphism
$$\U_F^{-w}(L\frakg)\to 
K_{G_W\times\bbC^\times}\big(\widehat\frakM(W)^2,(\hat f_2)^{(2)}\big)_{\widehat\calZ(W)}\otimes_RF.$$
\item In the particular case where $Q=A_1$ the theorem implies that the shifted quantum group
$\U_F^{-w}(L\frakg)$ of $\fraks\frakl_2$ acts on the equivariant K-theory of the Quot scheme
parametrizing all finite length quotients of the trivial bundle $W\otimes\calO_{\bbA^1}$.
\end{enumerate}
\end{remark}

\section{CCA's and representations of (shifted) quantum loop groups}

\subsection{Admissible triples}\label{sec:admissible}
Let $Q$ be a Dynkin quiver.
Fix $W\in\bfC$. Fix a nilpotent element $\gamma$ in $\frakg_W^\nil$ and fix
a cocharacter $\sigma:\bbC^\times\to G_W$ such that 
\begin{align}\label{sigma2}
\Ad_{\sigma(z)}(\gamma)=z^2\gamma.
\end{align}
We equip $W$ with the $I^\bullet$-grading \eqref{sigma1}, for which we have $\gamma\in\frakg^2_W.$
We define
$$a=(\sigma,\xi)
 ,\quad
A=a(\bbC^\times)\subset\{(g,\xi(z))\,;\,g\in G_W\,,\,z\in\bbC^\times\,,\,\Ad_g(\gamma)=z^2\gamma\}
.$$
We'll call $(W,A,\gamma)$ an admissible triple.
The triple $(W,A,\gamma)$ is called regular if $\gamma$ is a regular nilpotent element of $\frakg_W$.
Let
$\langle-,-\rangle:\frakg_W\times\frakg_W^\vee\to\bbC$
be the canonical pairing.
We have the
$A$-invariant function
\begin{align}\label{fu}
f_\gamma:\frakM(W)\to\bbC
 ,\quad
\ux\mapsto\langle\gamma,\mu_W(\ux)\rangle.
\end{align}
The restriction of the function $f_\gamma$ to the
$A$-fixed points locus is the function
$f_\gamma^\bullet:\frakM^\bullet(W)\to\bbC$.
We have  $f_\gamma=f_0\circ\pi$ and $f_\gamma^\bullet=f_0^\bullet\circ\pi^\bullet$ with
\begin{align}\label{f0}
f_0:\frakM_0(W)\to\bbC
 ,\quad
f_0^\bullet:\frakM_0^\bullet(W)\to\bbC
 ,\quad
\ux\mapsto\langle\gamma,\mu_0(\ux)\rangle.
\end{align}

\begin{proposition}\label{prop:crit2}
\hfill
\begin{enumerate}[label=$\mathrm{(\alph*)}$,leftmargin=8mm]
\item
For $V\in\bfC$, the assignment $(x,\varepsilon)\mapsto x$ yields an isomorphism
\begin{align*}
\big\{(x,\varepsilon)\in\mu_V^{-1}(0)_s\times\frakg^\nil_V\,;\,
[\gamma\oplus\varepsilon,x]=0\big\}/G_V
= \crit(f_\gamma)\cap\frakM(v,W).
\end{align*}
\item
For $V\in\bfC^\bullet$, the assignment $(x,\varepsilon)\mapsto x$ yields an isomorphism
\begin{align*}
\big\{(x,\varepsilon)\in\mu_V^{-1}(0)_s^\bullet\times\frakg^2_V\,;\,
[\gamma\oplus\varepsilon,x]=0\big\}/G_V^0
= \crit(f_\gamma^\bullet)\cap\frakM^\bullet(v,W).
\end{align*}
\end{enumerate}
\end{proposition}

\begin{proof}
The infinitesimal action of $\gamma$  yields a vector field on $\frakM(W)$.
Let $\frakM(W)^\gamma$ be the reduced zero locus of this vector field
in $\frakM(W)$.
Since $\mu_W$ is the moment map for the $G_W$-action on
$\frakM(W)$, from \eqref{fu} we have
$\frakM(W)^\gamma=\crit(f_\gamma).$ 
Recall that the group $G_V$ acts properly and freely on the set of stable representations,
i.e., the map $G_V\times\overline\X(V,W)_s\to\overline\X(V,W)_s\times\overline\X(V,W)_s$ 
defined by $(g,x)\mapsto(gx,x)$ is a closed embedding.
Using this, a standard argument implies that
\begin{align}\label{above}
\frakM(v,W)^\gamma=\big\{x\in\mu_V^{-1}(0)_s\,;\,
\exists\varepsilon\in\frakg_V\,,\,[\gamma\oplus\varepsilon,x]=0\big\}/G_V.
\end{align}
For any $x$ as in \eqref{above}
the stability condition implies that 
there is at most one element $\varepsilon\in\frakg_V$ such that $[\gamma\oplus\varepsilon,x]=0$, 
because 
$$[\gamma\oplus\varepsilon_1,x]=[\gamma\oplus\varepsilon_2,x]
\Rightarrow [\varepsilon_1-\varepsilon_2,x]=0
\Rightarrow \varepsilon_1-\varepsilon_2=0.$$
Thus, the assignment $(x,\varepsilon)\mapsto x$ yields an isomorphism
\begin{align*}
\crit(f_\gamma)\cap\frakM(v,W)=\frakM(v,W)^\gamma=\big\{(x,\varepsilon)\in\mu_V^{-1}(0)_s\times\frakg_V\,;\,
[\gamma\oplus\varepsilon,x]=0\big\}/G_V.
\end{align*}
Finally, given a positive integer $l$ such that $\gamma^l=0$,
for each pair $(x,\varepsilon)$ as above we have $[\varepsilon^l,\alpha]=a\varepsilon^l=0$,
hence $\varepsilon^l=0$ because $x$ is stable.
Part (a) is proved.
Part (b) follows from (a).
Indeed, the group $A$ acts on  $\frakM(v,W)^\gamma$ and we have
\begin{align*}
\crit(f_\gamma^\bullet)\cap\frakM^\bullet(v,W)
=\crit(f_\gamma)\cap\frakM^\bullet(v,W)
=(\frakM(v,W)^\gamma)^A.
\end{align*}
\end{proof}

\begin{proposition}\label{prop:LG}
\hfill
\begin{enumerate}[label=$\mathrm{(\alph*)}$,leftmargin=8mm]
\item
$(\frakM(W),f_\gamma)$ is a smooth $A$-invariant LG-model and
$(f_\gamma)^{-1}(0)$ is homotopic to $\frakM(W)$.
\item
$(\frakM^\bullet(W),f_\gamma^\bullet)$ is a smooth LG-model and
$(f_\gamma^\bullet)^{-1}(0)$ is homotopic to $\frakM^\bullet(W)$.
\end{enumerate}
\end{proposition}

\begin{proof}
We'll prove the part (a). The proof of (b) is similar.
To prove that the function $f_\gamma$ is regular
it is enough to check that it does not vanish identically on any connected component $\frakM(v,W)$.
We may assume that $Q$ is of type $A_1$. 
It is easy to see that $f_\gamma\neq 0$, since each nilpotent matrix in
$\frakg_W^\nil$ is of the form $aa^*$ for some tuple $x=(a,a^*)$ in $\overline\X(V,W)$ with $V\neq 0$.
Next, we must check that 
$\crit(f_\gamma)$
is contained in $(f_\gamma)^{-1}(0)$.
Proposition \ref{prop:crit2} yields
\begin{align*}
\ux\in\crit(f_\gamma)\Rightarrow
f_\gamma(\ux)=\Tr_W(\gamma aa^*)
=\Tr_V(\varepsilon a^*a)
=-\Tr_V(\varepsilon[\alpha,\alpha^*])
=0.
\end{align*}
The function $f_\gamma$ is $A$-invariant.
Thus the first claim in (a) is proved. 
To prove the second one, recall that $\frakM(W)$ is homotopic to $\frakL(W)$ 
by \cite[cor.~5.5]{N94}, and that $f_\gamma=f_0\circ\pi$ by \eqref{f0}.
The function $f_0$ is homogeneous of degree $2$ relatively to the $\bbC^\times$-action $\xi$ in 
\S\ref{sec:quiver variety}. 
Hence the zero locus of $f_0$ is homotopic to $\{0\}$. 
So, by \cite[\S 4.3]{S80}, the zero locus of  $f_\gamma$ is also homotopic to 
$\frakL(W)$. 
\end{proof}

Now, we state a version of Theorem \ref{thm:Nakajima}
for the deformed potential $f_\gamma$.
We'll write it in K-theory, topological K-theory and cohomology.
See \eqref{Ktop2} for a definition the topological critical equivariant K-theory
and a definition of $K_G^\top(X,f)_Z$.
Recall that  $R=R_A$ and $F=F_A$.
In the non graded case we set
$$G=A
 ,\quad
X=\frakM(W)
 ,\quad
X_0=\frakM_0(W)
 ,\quad
L=\frakL(W)
 ,\quad
f=f_\gamma.$$
In the graded case we set
$$G=\{1\}
 ,\quad
X=\frakM^\bullet(W)
 ,\quad
X_0=\frakM_0^\bullet(W)
 ,\quad
L=\frakL^\bullet(W)
 ,\quad
f=f_\gamma^\bullet.$$
From \eqref{nak1} we get $R$-algebra structures on $K^A(\calZ(W))$ and
$K_\top^A(\calZ(W))$ with
\begin{itemize}[leftmargin=3mm]
\item[-]  the topologization map yields an homomorphism
$K^A(\calZ(W))\to K_\top^A(\calZ(W))$ and isomorphisms
$K_\top^A(\frakL(W))=K^A(\frakL(W))$,
$K_\top^A(\frakM(W))=K^A(\frakM(W)),$
\item[-] $K_\top^A(\frakL(W))$ and $K_\top^A(\frakM(W))$
are modules over $K_\top^A(\calZ(W))/\tor$.
\end{itemize}

\begin{theorem}\label{thm:moduleK}
\hfill
\begin{enumerate}[label=$\mathrm{(\alph*)}$,leftmargin=8mm]
\item
$K_A(\frakM(W)^2,(f_\gamma)^{(2)})_{\calZ(W)}$ is an $R$-algebra wich acts on 
$K_A(\frakM(W),f_\gamma)$ and $K_A(\frakM(W),f_\gamma)_{\frakL(W)}$.
\item
$K_A^\top(\frakM(W)^2,(f_\gamma)^{(2)})_{\calZ(W)}$ 
is an $R$-algebra wich acts on 
$K_A^\top(\frakM(W),f_\gamma)$ and $K_A^\top(\frakM(W),f_\gamma)_{\frakL(W)}.$
The topologization map is an intertwiner.
\item
There are $R$-algebra homomorphisms with $\Upsilon$ surjective
$$\xymatrix{
\U_R(L\frakg)\ar[r]&K^A(\calZ(W))/\tor
\ar[r]^-\Upsilon& K_A(\frakM(W)^2,(f_\gamma)^{(2)})_{\calZ(W)}/\tor}$$
\item
Idem in the graded case with $\frakM^\bullet(W)$, $f^\bullet_\gamma$ and
$$\xymatrix{
\U_\zeta(L\frakg)\ar[r]&K(\calZ^\bullet(W))\ar[r]^-\Upsilon& 
K(\frakM^\bullet(W)^2,(f_\gamma^\bullet)^{(2)})_{\calZ^\bullet(W)}}$$
\item
$H^\bullet(\frakM^\bullet(W)^2,(f_\gamma^\bullet)^{(2)})_{\calZ^\bullet(W)}$ 
is an algebra which acts on 
$H^\bullet(\frakM^\bullet(W),f_\gamma^\bullet)$ and
$H^\bullet(\frakM^\bullet(W),f_\gamma^\bullet)_{\frakL^\bullet(W)}$, and there are
algebra homomorphisms 
$$\xymatrix{\U_\zeta(L\frakg)\ar[r]& 
H_\bullet(\calZ^\bullet(W))\ar[r]& H^\bullet(\frakM^\bullet(W)^2,(f_\gamma^\bullet)^{(2)})_{\calZ^\bullet(W)}.}
$$
\end{enumerate}
\end{theorem}

\begin{proof}
Part (a) follows from Corollary \ref{cor:critalg1}, 
Part (b) from Proposition \ref{prop:critalg2},
Parts (c) and (d) from Corollary \ref{cor:critalg1}, \eqref{nak1} and \eqref{nak2}
and  Part (e) from Proposition \ref{prop:critalg3}.
\end{proof}

\subsection{CCA's and quantum loop groups}

\subsubsection{Quiver Grassmannians}
Let $Q$ be a quiver of Dynkin type.
Let  $\overline\Pi$ be the preprojective algebra of the quiver $Q$.
The generalized preprojective algebra $\widetilde\Pi$ is the quotient of the path algebra $\bbC\widetilde Q$
 by the two-sided ideal generated by the elements
$[\alpha,\alpha^*]$ and $[\varepsilon,\alpha]$.
We have $\widetilde\Pi=\overline\Pi\otimes\bbC[\varepsilon]$.
For every positive integer $l$ we set $\widetilde\Pi^l=\widetilde\Pi/(\varepsilon^l)$.
We equip the quiver $\widetilde Q$ with the degree function
$\deg:\widetilde Q_1\to\bbZ$ in \eqref{degree2}.
We equip the algebras $\overline\Pi$, $\widetilde\Pi$ and $\widetilde\Pi^l$ with the 
corresponding $\bbZ$-gradings.
For each vertex $i$ let $e_i$ be the length 0 path supported on $i$
and let $S_i$ be the $\widetilde\Pi$-module of dimension $\delta_i$.
Let $\bfP$ and $\bfP^\bullet$ be the categories of finite dimensional $\widetilde\Pi$-modules
and graded $\widetilde\Pi$-modules.
Let $\bfP_l$ and $\bfP^\bullet_l$ the subcategories of finite dimensional $\widetilde\Pi^l$-modules
and graded $\widetilde\Pi^l$-modules.
We equip the categories $\bfP^\bullet$ and $\bfP^\bullet_l$ 
with the grading shift functor $[1]$ and the duality functor $D$ such that
$D(M)_k=(M_{-k})^\vee$ for each  graded module $M=\bigoplus_{k\in\bbZ}M_k$.
The $\widetilde\Pi$-action on $D(M)$ is the transpose of the $\widetilde\Pi$-action on $M$.
A (graded) module over $\widetilde\Pi$ or $\overline\Pi$
is nilpotent if it is killed by a power of the augmentation ideal.
We consider the following graded $\widetilde\Pi$-modules
$$I_{i,k}^l=D(\widetilde\Pi^le_i)[-k-l]
,\quad
I_i=D(\widetilde\Pi e_i)
=\bigcup_{l>0}I_{i,-l}^l
,\quad
I_{i,k}=I_i[-k].$$
The Jacobi algebra of the quiver with potential $(\widetilde Q^\bullet,\bfw^\bullet_2)$ is the quotient 
$$\widetilde\Pi^\bullet=\bbC \widetilde Q^\bullet\,\big/\,\big(\partial\bfw_2^\bullet/\partial\varepsilon_{i,k}
\,,\,\partial\bfw_2^\bullet/\partial\alpha_{ij,k}\,;\,i,j\in I\,,\,k\in\bbZ\big).$$
By \cite[prop.~4.4, 5.1]{FM21}, a graded $\widetilde\Pi$-module is the same as a 
$\widetilde\Pi^\bullet$-module, and,
under this equivalence, the graded $\widetilde\Pi$-module $I_{i,k}^l$
is the same as the generic kernel associated in \cite{HL16} with the Kirillov-Reshetikhin module
$KR^l_{i,k}$.
Given a module $M\in\bfP$ and $v\in\bbN I$, let $\Gr_v(M)$ and $\widetilde\Gr_v(M)$
be the Grassmannians of all $\overline\Pi$-submodules and $\widetilde\Pi$-submodules 
of dimension $v$. 
Given a graded module $M\in\bfP^\bullet$  and $v\in\bbN I^\bullet$, 
let $\widetilde\Gr^\bullet_v(M)$ be the Grassmannian of all graded $\widetilde\Pi$-submodules 
of dimension $v$. 
Set
$$\widetilde\Gr(M)=\bigsqcup_{v\in\bbN I}\widetilde\Gr_v(M)
 ,\quad
\widetilde\Gr^\bullet(M)=
\bigsqcup_{v\in\bbN {I^\bullet}}\widetilde\Gr^\bullet_v(M).$$

\subsubsection{Finite dimensional representations of quantum loop groups}\label{sec:HL1}
We define
\begin{align}\label{KRw}
w^l_{i,k}=\delta_{i,k-l+1}+\delta_{i,k-l+3}+\cdots+\delta_{i,k+l-1},\quad
(i,k)\in I^\bullet,\quad
l\in\bbN^\times.
\end{align}
Fix a graded vector space $W^l_{i,k}\in\bfC^\bullet$
of dimension $w^l_{i,k}$
and a regular element
$\gamma_{i,k}^l$ in $\frakg^2_{W_{i,k}^l}$.
Fix an admissible triple $(W,A,\gamma)$.
Given any tuple $i_1,k_1,l_1,\dots,i_s,k_s,l_s$, we write
\begin{align}\label{WgI}
W=\bigoplus_{r=1}^sW_{i_r,k_r}^{l_r}
 ,\quad
\gamma=\bigoplus_{r=1}^s\gamma_{i_r,k_r}^{l_r}
 ,\quad
I_\gamma=\bigoplus_{r=1}^sI_{i_r,k_r}^{l_r}.
\end{align}

\begin{proposition}\label{prop:HL1} 
Let $(W,A,\gamma)$ be an admissible triple.
We have an homeomorphism 
$\crit(f_\gamma^\bullet)\cap\frakL^\bullet(W)=\widetilde\Gr^\bullet\!(I_\gamma).$
\end{proposition}
 
\begin{proof} 
Let $\widehat\Gr_v(I_\gamma)$
be the set of all injective $I$-graded linear maps $f:\bbC^v\to I_\gamma$
whose image is a $\overline\Pi$-submodule of $I_\gamma$.
The quotient by the $G_v$-action is a $G_v$-torsor
$\widehat\Gr_v(I_\gamma)\to\Gr_v(I_\gamma)$, $f\mapsto\Im(f)$. 
For each positive integer $l$, since $\widetilde\Pi^l=\overline\Pi\otimes\bbC[\varepsilon]/(\varepsilon^l)$
as a $\bbC[\varepsilon]/(\varepsilon^l)$-module,
the top of the $\overline\Pi$-module $\widetilde\Pi^le_i$ is 
$$\top(\widetilde\Pi^le_i)=
S_i\otimes\bbC[\varepsilon]/(\varepsilon^l).$$
We deduce that the socle of $I_\gamma$, 
viewed as a $\overline\Pi$-module, is $I^\bullet$-graded of dimension $w$. 
Further, the action of $\varepsilon$ on $I_\gamma$ preserves the socle and is given by
an homogeneous operator of degree 2.
We identify the $I^\bullet$-graded vector spaces $W=\soc(I_\gamma)$
so that the action of $\gamma$ on $W$ coincides with the action of
$\varepsilon$ on $\soc(I_\gamma)$.
Fix an $I^\bullet$-graded $\bbC[\varepsilon]$-linear map 
$a:I_\gamma\to W$
which is the identity on  $W\subset I_\gamma$.
The set of nilpotent representations in $\overline\X(v,W)_s$ is
$$\overline\X(v,W)_s^\nil=\{x=(\alpha,a,a^*)\in\overline\X(v,W)_s\,;\,\underline x\in\frakL(v,W)\}.$$
We abbreviate
$$\mu_v^{-1}(0)_s^\nil=\mu_v^{-1}(0)\cap\overline\X(v,W)_s^\nil
=\{x\in\mu_v^{-1}(0)_s\,;\,a^*=0\,,\,\alpha\ \text{is\ nilpotent}\},$$ 
where $\alpha$ is nilpotent if it is nilpotent as a representation of $\overline\Pi$.
Let $\Aut_I(\overline\Pi)$ be the group of all algebra automorphisms of $\overline\Pi$ that
fix the idempotents $e_i$'s.
By \cite[\S 5A]{ST11}, the group $G_v\times G_W\times\Aut_I(\overline\Pi)$
acts on $\widehat\Gr_v(I_\gamma)$, and
by \cite[thm.~4.4, prop.~5.1]{ST11} there is a
$G_v\times G_W\times\Aut_I(\overline\Pi)$-equivariant homeomorphism 
\begin{align}\label{homeo1}
\widehat\Gr_v(I_\gamma)\to\mu_v^{-1}(0)_s^\nil
 ,\quad
f\mapsto (f^{-1}\circ\alpha\circ f\,,\,a\circ f\,,\,0)\end{align}
where the map $a:I_\gamma\to W$ is as above and $\alpha$ denotes the $\overline\Pi$-action on 
$I_\gamma$. 
Now we consider the nilpotent operator $\varepsilon$ on $I_\gamma$ given by the $\widetilde\Pi$-action.
For any $f\in\widehat\Gr_v(I_\gamma)$ such that $\Im(f)\in\widetilde\Gr(I_\gamma)$ 
we also have a nilpotent operator $\varepsilon$ on $\Im(f)$ commuting with $\alpha$ and such that
$[\gamma\oplus\varepsilon,a]=0$.
On the other hand, Proposition \ref{prop:crit2} yields an isomorphism
\begin{align}\label{crit2}
\crit(f_\gamma)\cap\frakL(v,W)=\big\{(x,\varepsilon)\in\mu_v^{-1}(0)_s^\nil\times\frakg^\nil_v\,;\,
[\gamma\oplus\varepsilon,x]=0\big\}/G_v
\end{align}
Comparing \eqref{homeo1} and \eqref{crit2}, 
we get a $G_W\times\Aut_I(\overline\Pi)$-equivariant homeomorphism 
$\widetilde\Gr(I_\gamma)=\crit(f_\gamma)\cap\frakL(W)$.
To prove the proposition, we view $A$ as a subgroup of 
$G_W\times\Aut_I(\overline\Pi)$ in the obvious way.
Since $\frakM^\bullet(W)=\frakM(W)^A$,
the homeomorphism in the proposition follows  by taking the $A$-fixed points.

\end{proof}

\begin{theorem}\label{thm:HL1} 
Let $(W,A,\gamma)$ be a regular admissible triple with
$W=W^l_{i,k}$. The $\U_\zeta(L\frakg)$-modules
$K(\frakM^\bullet(W),f_\gamma^\bullet)$,
$K(\frakM^\bullet(W),f_\gamma^\bullet)_{\frakL^\bullet(W)}$
and their cohomological analogues are simple and are isomorphic to $KR_{i,k}^l$.
\end{theorem}

\begin{proof}
Theorem \ref{thm:moduleK} yield a representation of 
$\U_\zeta(L\frakg)$ on each of the vector spaces above.
Let first prove that
$H^\bullet(\frakM^\bullet(W),f_\gamma^\bullet)_{\frakL^\bullet(W)}$
is isomorphic to $KR_{i,k}^l$.
We'll give an algebraic proof of the claim.
See  \S\ref{sec:proof2} for a geometric proof using microlocal geometry.
Proposition \ref{prop:HL1}  implies that the cohomology space
$H^\bullet(\frakM^\bullet(v,W),f_\gamma^\bullet)_{\frakL^\bullet(v,W)}$ vanishes whenever 
$\widetilde\Gr^\bullet_v(I_\gamma)=\emptyset$.
We have
$$e^{w-\bfc v}=m_{i,k}^l\prod_{j,r}A_{j,r}^{-v_{j,r}},\quad
m_{i,k}^l=Y_{i,k-l+1}\cdots Y_{i,k+l-1}.$$
The socle of the $\widetilde\Pi$-module $I^l_{i,k}$ has dimension $\delta_{i,k+l}$.
Each non-zero $\widetilde\Pi$-submodule of $I_\gamma$ contains the socle of $I_\gamma$.
Hence, by  \eqref{AY}, given $v\neq 0$ such that
the quiver Grassmannian $\widetilde\Gr_v^\bullet(I_\gamma)$ is non empty, we have
$$e^{w-\bfc v}\in m_{i,k}^l\,A_{i,k+l}^{-1}\,\bbZ[A_{j,r}^{-1}\,;\,(j,r)\in I^\bullet].$$
Hence, we have
$$q\ch(H^\bullet(\frakM^\bullet(W),f_\gamma^\bullet)_{\frakL^\bullet(W)})\in m_{i,k}^l\,\big(1+A_{i,k+l}^{-1}\,\bbZ[A_{j,r}^{-1}\,;\,(j,r)\in I^\bullet]\big).$$
The monomial $m_{i,k}^l\,A_{i,k+l}^{-1}$ is right-negative by \cite[lem.~4.4]{H06}. 
From \cite{FM01} we deduce that
the $q$-character of $H^\bullet(\frakM^\bullet(W),f_\gamma^\bullet)_{\frakL^\bullet(W)}$ 
contains a unique $\ell$-dominant monomial.
See \S\ref{sec:qg} for more details on $q$-characters.
Hence the $\U_\zeta(L\frakg)$-module 
$H^\bullet(\frakM^\bullet(W),f_\gamma^\bullet)_{\frakL^\bullet(W)}$ is simple.

The argument above implies that 
\begin{align}\label{IRR1}
\crit(f_\gamma^\bullet)\cap\frakL^\bullet(v,W)\neq\emptyset\ ,\ v\neq 0
\Rightarrow \text{ $e^{w-\bfc v}$\ is not $\ell$-dominant.}
\end{align}
On the other hand, since $\crit(f_\gamma^\bullet)$ is a closed conic subset, we have
$$\crit(f_\gamma^\bullet)\neq\emptyset
\Rightarrow\crit(f_\gamma^\bullet)\cap\frakL^\bullet(v,W)\neq\emptyset.$$
We deduce that
\begin{align}\label{IRR2}
\crit(f_\gamma^\bullet)\neq\emptyset\ ,\ v\neq 0
\Rightarrow \text{the\ monomial\  $e^{w-\bfc v}$\ is not $\ell$-dominant.}
\end{align}
Hence the $q$-character of $H^\bullet(\frakM^\bullet(W),f_\gamma^\bullet)$ 
contains a unique $\ell$-dominant monomial as well.

The proof in K-theory is similar. 
More precisely, by \eqref{IRR2} 
if the monomial $e^{w-\bfc v}$ is $\ell$-dominant and 
$v\neq 0$ then $\crit(f_\gamma^\bullet)=\emptyset$, hence,
since any matrix factorization is supported on the critical set of the potential by
\cite[cor.~3.18]{PV11}, we have
$$
K(\frakM^\bullet(v,W),f_\gamma^\bullet)
=K(\frakM^\bullet(v,W),f_\gamma^\bullet)_{\frakL^\bullet(v,W)}=0
$$
Further, the definition of the representation of $\U_\zeta(L\frakg)$ in Theorem \ref{thm:moduleK}
implies that 
$K(\frakM^\bullet(v,W),f_\gamma^\bullet)$ and
$K(\frakM^\bullet(v,W),f_\gamma^\bullet)_{\frakL^\bullet(v,W)}$
are $\ell$-weight subspaces of $\ell$-weight $\Psi_{w-\bfc v}$.
Hence, since $\frakM^\bullet(0,W)$ is a point, the $q$-characters of the $\U_\zeta(L\frakg)$-modules
$$K(\frakM^\bullet(W),f^\bullet_\gamma)
 ,\quad
K(\frakM^\bullet(W),f^\bullet_\gamma)_{\frakL^\bullet(W)}$$
contain a unique $\ell$-dominant monomial.
Thus, both modules are simple and isomorphic to $KR_{i,k}^l$.
\end{proof}

A similar result holds for some irreducible tensor products of Kirillov-Reshetikhin modules.

\begin{proposition}\label{prop:TPKR}
Fix $(i_r,k_r,l_r)\in I^\bullet\times\bbN^\times$ for all $r=1,\dots s$ such
that  either the condition $\mathrm{(a)}$ or the condition $\mathrm{(b)}$ below holds
for some integer $l$
\hfill
\begin{enumerate}[label=$\mathrm{(\alph*)}$,leftmargin=8mm]
\item
$k_r\geqslant l$ and $[k_r+2-2l_r,k_r]=(k_r-2\bbN)\cap[l,k_r]$ for all $r$, and
$$W=\bigoplus_{r=1}^sW_{i_r,1+k_r-l_r}^{l_r}
 ,\quad
\gamma=\bigoplus_{r=1}^s\gamma_{i_r,1+k_r-l_r}^{l_r}
,\quad
KR_W=\bigotimes_{r=1}^sKR_{i_r,1+k_r-l_r}^{l_r}
$$
\item
$k_r\leqslant l$ and $[k_r,k_r-2+2l_r]=(k_r+2\bbN)\cap[k_r,l]$ for all  $r$, and
$$W=\bigoplus_{r=1}^sW_{i_r,-1+k_r+l_r}^{l_r}
 ,\quad
\gamma=\bigoplus_{r=1}^s\gamma_{i_r,-1+k_r+l_r}^{l_r}
,\quad
KR_W=\bigotimes_{r=1}^sKR_{i_r,-1+k_r+l_r}^{l_r}
$$
\end{enumerate}
The $\U_\zeta(L\frakg)$-modules
$K(\frakM^\bullet(W),f_\gamma^\bullet)$ and
$K(\frakM^\bullet(W),f_\gamma^\bullet)_{\frakL^\bullet(W)}$
and their cohomological analogues are simple and isomorphic to $KR_W$.
\end{proposition}

\begin{proof}
In both cases the $\U_\zeta(L\frakg)$-module $KR_W$ is irreducible by
\cite[thm.~4.11]{FH15}. 
Let $M$ denote either
$K(\frakM^\bullet(W),f_\gamma^\bullet)$ or
$K(\frakM^\bullet(W),f_\gamma^\bullet)_{\frakL^\bullet(W)}$
or their cohomological analogues.
We define accordingly
$M_v=K(\frakM^\bullet(v,W),f_\gamma^\bullet)$ or
$K(\frakM^\bullet(v,W),f_\gamma^\bullet)_{\frakL^\bullet(v,W)}$
or their cohomological analogues.
The definition of the $\U_\zeta(L\frakg)$-action on $M$ in Theorem \ref{thm:moduleK}
implies that $M_v$ 
is an $\ell$-weight space of $\ell$-weight $\Psi_{w-\bfc v}$.
Further, the homeomorphism 
$\crit(f_\gamma^\bullet)\cap\frakL^\bullet(W)=\widetilde\Gr^\bullet\!(I_\gamma)$
in Proposition \ref{prop:HL1} yields
\begin{align}\label{bound}\crit(f_\gamma^\bullet)\cap\frakM^\bullet(v,W)\neq\emptyset
\Rightarrow 
\widetilde\Gr_v^\bullet(I_\gamma)\neq\emptyset,\end{align}
Now, we consider the cases (a) and (b) separately.
We'll abreviate $\calA=\bbZ[A_{j,r}^{-1}\,;\,(j,r)\in I^\bullet]$. 

Let us prove (b).
Any non zero graded $\widetilde\Pi$-submodule of $I_\gamma$ intersects the socle of $I_\gamma$.
We have
\begin {align}\label{dimsocle}
\dim\soc(I_\gamma)=\sum_{r=1}^s\delta_{i_r,k_r-1+2l_r}=
\sum_{\substack{1\leqslant r\leqslant s\\k_r\in l+2\bbZ}}\delta_{i_r,l+1}+
\sum_{\substack{1\leqslant r\leqslant s\\k_r\in l-1+2\bbZ}}\delta_{i_r,l}.
\end{align}
From \eqref{bound}, \eqref{dimsocle} and \eqref{AY} we deduce that
$$q\ch(M)\in m\Big(1+\sum_i(A_{i,l+1}^{-1}\calA+A_{i,l}^{-1}\calA)\Big)
,\quad
m=\prod_{r=1}^s\prod_{k=kr}^{k_r+2l_r-2}Y_{i_r,k}.$$
Hence, all monomials in $q\ch(M)$ are right-negative except $m$ by  \cite{FM01}, \cite{H06}.
Thus, the $\U_\zeta(L\frakg)$-module $M$ is irreducible and is isomorphic to $KR_W$.

Now we prove (a). 
We equip the categories $\bfC$ and $\bfC^\bullet$ with the duality functors such that
$D(W)_i=(W_i)^\vee$ and $D(W)_{i,r}=(W_{i,-r})^\vee$ respectively.
By \cite[\S 4.6]{VV03}, 
for each $W\in\bfC$ there is an isomorphism of algebraic varieties
$\omega:\frakM(W)\to\frakM(D(W))$ which intertwines the action of the element
$(g,z)\in G_W\times\bbC^\times$ with the action of the element $({}^tg^{-1},z)\in G_{D(W)}\times\bbC^\times$.
Taking the fixed points locus of some one parameter subgroups of
$G_W\times\bbC^\times$ and $G_{D(W)}\times\bbC^\times$
acting on the quiver varieties, we get for each $W\in\bfC^\bullet$
an isomorphism of algebraic varieties
$\omega:\frakM^\bullet(W)\to\frakM^\bullet(D(W))$
which intertwines the functions $f^\bullet_\gamma$ and $f^\bullet_{{}^t\gamma}$ for each element
$\gamma\in\frakg_W^2$. 
Here, the transpose ${}^t\gamma$ is viewed as an element in $\frakg^2_{D(W)}$.
Set
$\overline{M}=K(\frakM^\bullet(D(W)),f_{{}^t\gamma}^\bullet)$,
$K(\frakM^\bullet(D(W)),f_{{}^t\gamma}^\bullet)_{\frakL^\bullet(D(W))}$
or their cohomological analogue.
The map $\omega$ yields a vector space isomorphism $M\to \overline{M}$.
Both spaces $M$ and $\overline{M}$ are equipped with a representation of $\U_\zeta(L\frakg)$.
Let $i\mapsto i^*$ be the involution of the set $I$ such that $w_0\alpha_i=-\alpha_{i^*}$,
where $\alpha_i$ is the simple root corresponding to the vertex $i$.
By \cite[lem.~4.6]{VV03} we have
$q\ch(\overline{M})=\overline{q\ch(M)},$
where $f\mapsto \overline f$ is the involution of the ring $\bbZ[Y_{i,r}^{\pm 1}]$ such that
$\overline{Y_{i,r}}=Y_{i^*,h-2-r}$ and $h$ is the Coxeter number.
Now, we  apply the argument in the proof of case (b) with $M$ replaced by $\overline M$.
We deduce that the $q$-character $q\ch(\overline{M})$ admits at most one $\ell$-dominant monomial.
Hence $q\ch(M)$ admits also at most one $\ell$-dominant monomial.
Thus the $\U_\zeta(L\frakg)$-module $M$ is irreducible and the isomorphism $M=KR_W$ follows.
\end{proof}

\begin{remark}
\hfill
\begin{enumerate}[label=$\mathrm{(\alph*)}$,leftmargin=8mm]
\item
H. Nakajima informed us that, in an unpublished work with A. Okounkov,
they prove a statement
similar to Theorem \ref{thm:HL1}.
\item
Let $\chi(X,\calL)$ be the Euler characteristic of  $H^\bullet_c(X,\calL)$.
By \cite[thm.~4.8]{HL16} we have
$\chi\big(\widetilde\Gr^\bullet\!(I_{i,k}^l),\bbC\big)=\dim KR_{i,k}^l$.
By Proposition \ref{prop:HL1} and Theorem \ref{thm:HL1} we have
an $\U_\zeta(L\frakg)$-module isomorphism
$$H^\bullet\big(\widetilde\Gr^\bullet\!(I_{i,k}^l),\calL_{\gamma_{i,k}^l}\big)=KR_{i,k}^l.$$
\item
Shipman's work in \cite[prop.~7.2]{ST11} implies that the 
homeomorphism in Theorem \ref{thm:HL1}
is an isomorphism of algebraic varieties. 
\item
Let $W\in\bfC^\bullet$ and $A\subset G_W\times\bbC^\times$ be as in $\S\ref{sec:graded quiver}$.
Proposition \ref{prop:Thomason} yields
\begin{align*}
K_A(\frakM(W),f_\gamma)\otimes_RF&=K(\frakM^\bullet(W),f_\gamma^\bullet)\otimes F
\\
K_A(\frakM(W),f_\gamma)_{\frakL(W)}\otimes_RF&=
K(\frakM^\bullet(W),f_\gamma^\bullet)_{\frakL^\bullet(W)}\otimes F.
\end{align*} 
\item If $\gamma=0$, then $\widetilde\Pi^0=\overline\Pi$ and
Proposition \ref{prop:HL1} reduces to Lusztig's realization of the 
nilpotent graded quiver variety $\frakL^\bullet(W)$ as the graded quiver Grassmanian
$\Gr^\bullet(I_0)$ of the injective $\bar\Pi$-module $I_0$.
\end{enumerate}
\end{remark}

\subsection{CCA's and shifted quantum loop groups} 

\subsubsection{The critical cohomology of triple quiver varieties as a limit of critical cohomology of Nakajima's quiver varieties}\label{sec:HL2}

Fix $W\in \bfC^\bullet$ and $\gamma\in\frakg^2_W$.
We consider the function 
$\tilde f^\bullet_{\!\gamma}:\widetilde\frakM^\bullet(W)\to\bbC$
given by
$\tilde f^\bullet_{\!\gamma}=
\tilde f^\bullet_{\!1}-f_\gamma^\bullet$.

\begin{lemma}\label{lem:crit4} We have the following isomorphisms\hfill
\begin{enumerate}[label=$\mathrm{(\alph*)}$,leftmargin=8mm]
\item $H^\bullet(\widetilde\frakM^\bullet(W),\tilde f_{\!\gamma}^\bullet)=
H^\bullet(\frakM^\bullet(W),f_\gamma^\bullet)$,
\item
$K(\widetilde\frakM^\bullet(W),\tilde f_{\!\gamma}^\bullet)=
K(\frakM^\bullet(W),f_\gamma^\bullet)$,
\item
$\crit(\tilde f^\bullet_{\gamma})=\crit(f^\bullet_\gamma)$.
\end{enumerate}
\end{lemma}

\begin{proof}
Let $\tilde f^\bullet_{\!\gamma,\circ}$ be the restriction  of $\tilde f^\bullet_{\!\gamma}$
to the open subset $\widetilde\frakM^\bullet(W)_\circ\subset\widetilde\frakM^\bullet(W)$ introduced in 
\S\ref{sec:triple quiver}.
We first claim that the sets 
$\crit(\tilde f^\bullet_{\gamma,\circ}),\crit(\tilde f^\bullet_\gamma)\subset\widetilde\frakM^\bullet(W)$ coincide.
Hence
$$H^\bullet(\widetilde\frakM^\bullet(W)_\circ,\tilde f_{\!\gamma,\circ}^\bullet)=
H^\bullet(\widetilde\frakM^\bullet(W),\tilde f_{\!\gamma}^\bullet),\quad
K(\widetilde\frakM^\bullet(W)_\circ,\tilde f_{\!\gamma,\circ}^\bullet)=
K(\widetilde\frakM^\bullet(W),\tilde f_{\!\gamma}^\bullet).$$
Forgetting the variable $\varepsilon$ yields a vector bundle
\begin{align}\label{rho1}
\rho_1:\widetilde\frakM^\bullet(W)_\circ\to\big\{\ux\in\widetilde\frakM^\bullet(W)_\circ\,;\,
\varepsilon=0\big\}
\end{align}
We have 
$\tilde f^\bullet_{\!\gamma,\circ}=\tilde f^\bullet_{\!1,\circ}-\rho_1^*\Tr_W(\gamma a a^*)$
and
\begin{align*}
\frakM^\bullet(W)&=\big\{\ux\in\widetilde\frakM^\bullet(W)_\circ\,;\,
\varepsilon=0\,,\,\partial\tilde f^\bullet_{1,\circ}/\partial\varepsilon(\ux)=0\big\}
\end{align*}
Hence, the deformed dimensional reduction along the variable $\varepsilon$ 
 in cohomology \cite[thm.~1.2]{DP22} yields Part (a),
and the deformed dimensional reduction in K-theory \cite[thm.~1.2]{H17b} 
yields Part (b). Part (c) follows from Proposition \ref{prop:crit2}.

To prove the claim we must check that if $x\in\crit(\tilde f^\bullet_\gamma)$ is stable,
then it is $\circ$-stable. By hypothesis, we have
$[\gamma\oplus\varepsilon,x]=\mu_V(x)=0.$
We equip the vector space $V$ with the representation of $\widetilde Q$ given by $x$.
For each subrepresentation $V'\subset V$ of $\overline Q$ contained in $\Ker(a)$,
the subrepresentation of $\widetilde Q$ in $V$ generated by $V'$ is also contained in $\Ker(a)$
because $\alpha\circ\varepsilon=\varepsilon\circ\alpha$ and $\gamma\circ a=a\circ\varepsilon$.
Hence it is zero because $x$ is stable.
\end{proof}

Now, fix $W\in\bfC^\bullet$ and fix tuples $(i_r,k_r)\in I^\bullet$ with 
$r=1,2,\dots,s$ such that $\dim W=\sum_{r=1}^s\delta_{i_r,k_r}.$  
For each positive integers $l_1,\dots,l_s$, 
let $W_l\in\bfC^\bullet$ and $\gamma_l\in\frakg_{W_l}^2$ be such that
\begin{align}\label{WWg}
W_l=\bigoplus_{r=1}^sW_{i_r,1+k_r-l_r}^{l_r}
 ,\quad
\gamma_l=\bigoplus_{r=1}^s\gamma_{i_r,1+k_r-l_r}^{l_r}.
\end{align}
Note that $W=\Ker(\gamma_l)$ and that the socle of $I_W$ has the same dimension as $W$
in $\bbN I^\bullet$.
We'll need the following result, which can be viewed as a geometric analogue of
the limit procedure on normalized $q$-characters in \cite{HJ12}.

\begin{theorem}\label{thm:limH}
Fix $v\in\bbN I^\bullet$.
If $l_r\gg 0$ for each $r$, we have
\begin{enumerate}[label=$\mathrm{(\alph*)}$,leftmargin=8mm]
\item 
$H^\bullet(\widetilde\frakM^\bullet(v,W)\,,\,\tilde f^\bullet_2)=
H^\bullet(\widehat\frakM^\bullet(v,W)\,,\,\hat f^\bullet_2)=
H^\bullet(\frakM^\bullet(v,W_l)\,,\,f^\bullet_{\gamma_l})$,
\item
$K(\widetilde\frakM^\bullet(v,W)\,,\,\tilde f^\bullet_2)=
K(\widehat\frakM^\bullet(v,W)\,,\,\hat f^\bullet_2)=
K(\frakM^\bullet(v,W_l)\,,\,f^\bullet_{\gamma_l}).$
\end{enumerate}
\end{theorem}

\begin{proof}
Recall that $\tilde f^\bullet_{\!\gamma,\circ}$ is the restriction  of $\tilde f^\bullet_{\!\gamma}$
to $\widetilde\frakM^\bullet(W)_\circ$.
By Lemma \ref{lem:crit4}, 
we have
\begin{align}\label{form5}
H^\bullet(\widetilde\frakM^\bullet(v,W_l)_\circ\,,\,\tilde f^\bullet_{\gamma_l,\circ})
=H^\bullet(\frakM^\bullet(v,W_l)\,,\,f^\bullet_{\gamma_l})
\end{align}
Forgetting the variable $a^*$ yields a vector bundle
\begin{align}\label{rho2}\rho_2:\widetilde\frakM^\bullet(v,W_l)_\circ\to
\widehat\frakM^\bullet(v,W_l)_\circ.\end{align}
Set
$\overline\frakM^\bullet(v,W_l)
=\big\{\ux\in\widehat\frakM^\bullet(v,W_l)_\circ\,;\,
[a,\varepsilon]=0\big\}$.
We may view $W_l$ as a finite dimensional graded $\bbC[\varepsilon]$-module
with socle $W$ such that $\varepsilon$ acts as $\gamma_l$.
We have
\begin{align}\label{split}
\tilde f^\bullet_{\gamma_l,\circ}=\Tr_{W_l}([a,\varepsilon] a^*)+\rho_2^*\hat f^\bullet_2.
\end{align}
The deformed dimensional reduction
\cite[thm~1.2]{DP22} along the variable $a^*$ yields
\begin{align}\label{form6}
H^\bullet(\widetilde\frakM^\bullet(v,W_l)_\circ\,,\,\tilde f^\bullet_{\gamma_l,\circ})=
H^\bullet(\overline\frakM^\bullet(v,W_l)\,,\,\hat f^\bullet_2)
\end{align}
Composing \eqref{form5} and \eqref{form6} we get an isomorphism
\begin {align}\label{form7}
H^\bullet(\frakM^\bullet(v,W_l)\,,\,f^\bullet_{\gamma_l})=
H^\bullet(\overline\frakM^\bullet(v,W_l)\,,\,\hat f^\bullet_2).
\end{align}
To prove the claim (a), we must prove that there is an isomorphism
\begin{align}\label{form8}
H^\bullet(\widehat\frakM^\bullet(v,W)\,,\,\hat f^\bullet_2)=
H^\bullet(\overline\frakM^\bullet(v,W_l)\,,\,\hat f^\bullet_2)\quad\text{if}\quad l_1,\dots,l_s\gg 0
\end{align}
Let $i\in\Hom_{\bfC^\bullet}(W,W_l)$ be the obvious inclusion.
Fix $p\in\Hom_{\bfC^\bullet}(W_l,W)$ such that $p\circ i=\id$.
For each $\bbC[\varepsilon]$-module $V$ in $\bfC^\bullet$, 
the map $\Hom_{\bbC[\varepsilon]}(V,W_l)\to\Hom(V,W)$, $a\mapsto p\circ a$ is injective 
because the $\bbC[\varepsilon]$-module $W_l$ is cogenerated by $W$.
Further, it is invertible if $l_1,\dots,l_s$ are large enough, because $\varepsilon$ acts nilpotently on $V$.
Thus the assignment 
$(\alpha,a,0,\varepsilon)\mapsto (\alpha,p\circ a,0,\varepsilon)$ yields a closed embedding
$\overline\frakM^\bullet(v,W_l)\subset\widehat\frakM^\bullet(v,W)$ which
is an isomorphism if $l_1,\dots,l_s$ are large enough.

Next, we prove the claim (b).
By  Lemma \ref{lem:crit4}, we have
\begin{align*}
K(\widetilde\frakM^\bullet(v,W_l)_\circ\,,\,\tilde f^\bullet_{\gamma_l,\circ})=
K(\frakM^\bullet(v,W_l)\,,\,f^\bullet_{\gamma_l}).
\end{align*}
Forgetting the variable $a^*$ yields the vector bundle \eqref{rho2} such that
\eqref{split} holds.
Thus, the deformed dimensional reduction \cite[thm.~1.2]{H17b} along the variable $a^*$  yields
the isomorphism
\begin{align*}
K(\widetilde\frakM^\bullet(v,W_l)_\circ\,,\,\tilde f^\bullet_{\gamma_l,\circ})=
K(\overline\frakM^\bullet(v,W_l)\,,\,\hat f^\bullet_2)
\quad\text{if}\quad
l_1,\dots,l_s\gg 0.
\end{align*}
To apply the dimensional reduction, we need the map
$\hat f^\bullet_2$ on $\overline\frakM^\bullet(v,W_l)$ to be regular.
The variety $\overline\frakM^\bullet(v,W_l)$ may be not smooth, but
$\overline\frakM^\bullet(v,W_l)=\widehat\frakM^\bullet(v,W)$
if $l_1,\dots l_s$ are large enough, and $\widehat\frakM^\bullet(v,W)$ is smooth. 
So, we have proved that
$K(\widehat\frakM^\bullet(v,W)\,,\,\hat f^\bullet_2)=
K(\overline\frakM^\bullet(v,W_l)\,,\,\hat f^\bullet_2)$
if $l_1,\dots,l_s$ are large enough.
\end{proof}

\subsubsection{Representations of shifted quantum loop groups}\label{sec:HL2}

We now explain an analogue of Proposition \ref{prop:HL1} and
Theorem \ref{thm:HL1} for shifted quantum loop groups.
Fix $W\in\bfC^\bullet$ and fix tuples $(i_r,k_r)\in I^\bullet$ with 
$r=1,2,\dots,s$ such that $\dim W=\sum_{r=1}^s\delta_{i_r,k_r}.$  
We set
$I_W=\bigoplus_{r=1}^sI_{i_r,k_r}.$

\begin{proposition}\label{prop:HL3} For any $W\in\bfC^\bullet$
we have an homeomorphism 
$\crit(\tilde f_2^\bullet)\cap\widetilde\frakL^\bullet(W)=\widetilde\Gr^\bullet\!(I_W).$
\end{proposition}

\begin{proof}
The set of stable nilpotent representations is
$$\widetilde\X^\bullet(V,W)_s^\nil=\{x\in\widetilde\X^\bullet(V,W)_s\,;\,
\ux\in\widetilde\frakL^\bullet(v,W)\}$$
Let $\widehat\X^\bullet(V,W)_s$ be the set of stable tuples in $\widehat\X^\bullet(V,W)$.
Note that 
$\widetilde\X^\bullet(V,W)_s^\nil\subset\widehat\X^\bullet(V,W)_s.$
We have 
\begin{align*}
\crit(\tilde f^\bullet_2)\cap\widetilde\frakL^\bullet(v,W)&=
\{x\in\widetilde\X^\bullet(V,W)_s^\nil\,;\,
[\alpha,\alpha^*]=[\varepsilon,\alpha]=0\}\,/\,G^0_V\\
&=\{x\in\overline\X^\bullet(V,W)^\nil_s\times\frakg_V^2\,;\,
\mu_V(x)^\bullet=[\varepsilon,\alpha]=0\}\,/\,G^0_V.
\end{align*}
Using this isomorphism, the proof is similar to the proof of Proposition \ref{prop:HL1}.
More precisely, let $\underline\alpha$ and $\underline\varepsilon$ denote the action of the
elements 
$\alpha,\varepsilon\in\widetilde\Pi$
on the module $I_W$. 
We identify $W$ with the socle of the $\widetilde\Pi$-module $I_W$ as an $I^\bullet$-graded vector space.
Fix an $I^\bullet$-graded linear map $\ua:I_W\to W$ such that $\ua|_W=\id_W$. 
Let $\widehat\Gr_v^\bullet(I_W)$ be the set of injective $I^\bullet$-graded linear maps $f:\bbC^v\to I_W$
whose image is a $\widetilde\Pi$-submodule of $I_W$.
There is a $G^0_v$-equivariant map $\widehat\Gr_v^\bullet(I_W)\to\widehat\X^\bullet(v,W)_s$ such that
 \begin{align}\label{homeo2}
f\mapsto x=
(f^{-1}\circ\underline\alpha\circ f\,,\,\ua\circ f\,,\,0\,,\,f^{-1}\circ\underline\varepsilon\circ f)\end{align}
The tuple $x=(\alpha,a,a^*,\varepsilon)$ is stable because $W$ is the socle of $I_W$.
The map \eqref{homeo2} factorizes to a map 
$\widetilde\Gr_v^\bullet(I_W)\to
\crit(\tilde f^\bullet_2)\cap\widetilde\frakL^\bullet(v,W)$, because
$[\alpha,\alpha^*]=[\varepsilon,\alpha]=0$
and $x$ is nilpotent because $I_{i,k}=\bigcup_{l>0}I_{i,k}^l$.
This map is an homeomorphism by \cite[prop.~4.1]{ST11},
proving the proposition.
\end{proof}

Recall that for any $W\in\bfC^\bullet$ of dimension $w=(w_{i,k})$ in $\bbN I^\bullet$ , the symbol
$L^-(w)$ denotes the simple module in $\bfO_w$ with $\ell$-highest weight 
$\Psi^-_w=(\prod_{k\in\bbZ}(1-\zeta^k/u)^{- w_{i,k}})_{i\in I}$.

\begin{theorem}\label{thm:PF1}
Fix any $W\in\bfC^\bullet$.
The representations of $\U_\zeta^{-w}(L\frakg)$ in 
$K(\widetilde\frakM^\bullet(W),\tilde f^\bullet_2)$ and
$K(\widetilde\frakM^\bullet(W),\tilde f^\bullet_2)_{\widetilde\frakL^\bullet(W)}$
are both isomorphic to the simple module $L^-(w)$. 
\end{theorem}

For the cohomological analogue of the theorem
we need the following analogue of Theorem \ref{thm: Nakajima shifted}
whose proof will be given elsewhere.

\begin{proposition}  Fix any $W\in\bfC^\bullet$.
\hfill
\begin{enumerate}[label=$\mathrm{(\alph*)}$,leftmargin=8mm]
\item
There is an algebra homomorphism
$\U_\zeta^{-w}(L\frakg)\to 
H^\bullet\big(\widehat\frakM^\bullet(W)^2,(\tilde f_2^\bullet)^{(2)}\big)_{\widetilde\calZ^\bullet(W)}.$
\item
The algebra $\U_\zeta^{-w}(L\frakg)$ acts on 
$H^\bullet(\widehat\frakM^\bullet(W),\hat f^\bullet_2)_{\widehat\frakL^\bullet(W)}$ and
$H^\bullet(\widehat\frakM^\bullet(W),\hat f^\bullet_2)$ so that the subspaces
$H^\bullet(\widehat\frakM^\bullet(v,W),\hat f^\bullet_2)_{\widehat\frakL^\bullet(W)}$ and
$H^\bullet(\widehat\frakM^\bullet(v,W),\hat f^\bullet_2)$ are $\ell$-weight subspaces for each $v\in\bbN I^\bullet$.
\qed
\end{enumerate}
\end{proposition}

We can now prove the following.

\begin{theorem}\label{thm:PF2}
Fix any $W\in\bfC^\bullet$.
The $\U_\zeta^{-w}(L\frakg)$-modules
$H^\bullet(\widetilde\frakM^\bullet(W)\,,\,\tilde f^\bullet_2)$ and
$H^\bullet(\widetilde\frakM^\bullet(W)\,,\,\tilde f^\bullet_2)_{\widetilde\frakL^\bullet(W)}$
are both isomorphic to the simple module $L^-(w)$.
\end{theorem}

\begin{proof}[Proof of Theorems $\ref{thm:PF1}$ and $\ref{thm:PF2}$]
We first prove the isomorphism
$L^-(w)=H^\bullet(\widetilde\frakM^\bullet(W)\,,\,\tilde f^\bullet_2)$.
The case of
$H^\bullet(\widetilde\frakM^\bullet(W)\,,\,\tilde f^\bullet_2)_{\widetilde\frakL^\bullet(W)}$
is similar.
We have
$$
H^\bullet(\widetilde\frakM^\bullet(W)\,,\,\tilde f^\bullet_2)=
H^\bullet(\widehat\frakM^\bullet(W)\,,\,\hat f^\bullet_2).$$
The $\U_\zeta^{-w}(L\frakg)$-module
$H^\bullet(\widehat\frakM^\bullet(W)\,,\,\tilde f^\bullet_2)$
is of highest $\ell$-weight, with the same highest $\ell$-weight as the simple module $L^-(w)$.
Hence, it is enough to prove that both modules have the same character. 
Set $W_l=\dim W_l$ with $W_l$ as in \eqref{WWg}  and
$k_r$, $l_r,$ $l$ as in
Proposition \ref{prop:TPKR} (a). Then, we have
$$
L(W_l)=KR_{W_l}=H^\bullet(\frakM^\bullet(W_l)\,,\,f^\bullet_{\gamma_l}).$$
The proof of \cite[thm.~4.11]{FH15} implies that the normalized $q$-character of $L^-(w)$ is the limit of the 
normalized $q$-characters of the finite dimensional simple $\U_\zeta(L\frakg)$-modules $L(W_l)$ as $l\to\infty$.
Thus, it is enough to observe that Theorem \ref{thm:limH} implies
 that, for each $v\in\bbN I^\bullet$, for $l$ large enough we have
$$H^\bullet(\widetilde\frakM^\bullet(v,W)\,,\,\tilde f^\bullet_2)=
H^\bullet(\frakM^\bullet(v,W_l)\,,\,f^\bullet_{\gamma_l}).$$

The proof in K-theory is similar.
Let $W_l$ and $\gamma_l$ be as above.
The $\U^{-w}_F(L\frakg)$-action on 
$K(\widehat\frakM^\bullet(W)\,,\,\hat f^\bullet_2)$
given in Theorem \ref{thm: Nakajima shifted} is such that the subspace
$K(\widehat\frakM^\bullet(v,W)\,,\,\hat f^\bullet_2)$
is an $\ell$-weight subspace of $\ell$-weight given by the formula \eqref{psi+-}.
By Proposition \ref{prop:TPKR}, we have
$$L(W_l)=KR_{W_l}=K(\frakM^\bullet(W_l)\,,\,f^\bullet_{\gamma_l}).$$
By the same argument as in cohomology, it is enough to prove that, 
for each $v\in\bbN I^\bullet$,  if $l$ is large enough we have
$$K(\widehat\frakM^\bullet(v,W)\,,\,\hat f^\bullet_2)=
K(\frakM^\bullet(v,W_l)\,,\,f^\bullet_{\gamma_l}).$$
This follows from Theorem \ref{thm:limH}.
\end{proof}

\begin{remark}
Let $j_2$ be the embedding of $\widetilde\Gr^\bullet\!(I_W)$
 into $\widetilde\frakM^\bullet(W)$ given by Proposition \ref{prop:HL3}.
Set $\calL_2=j_2^!(\phi^p_{\!\tilde f_2^\bullet}\calC).$
The theorem above yields a representation of $\U_\zeta^{-w}(L\frakg)$ in the vector space
$H^\bullet(\widetilde\Gr^\bullet\!(I_W),\calL_2)$
which is isomorphic to $L^-(w)$.
\end{remark}

\appendix

\section{Representations of shifted quantum loop groups}\label{sec:QG}
\subsection{}
This section is a remind on shifted quantum loop groups of symmetric types.
We'll follow  \cite{FT19} and \cite{H22}.
Let $Q$ be a Dynkin quiver.
Fix $w^+,w^-\in\bbZ I$.
Let $c_{ij}$, $i,j\in I$, be the entries of the Cartan matrix $\bfc$ and
define
\begin{align}\label{gij}g_{ij}(u)=\frac{u-q^{c_{ij}}}{q^{c_{ij}}u-1}.\end{align}
Consider the formal series 
\begin{align*}
\delta(u)=\sum_{n\in\bbZ}u^n
,\quad
x^\pm_i(u)=\sum_{n\in\bbZ}x^\pm_{i,n}\,u^{-n}
 ,\quad
\psi^+_i(u)=\sum_{n\geqslant-w^+_i}\psi^+_{i,n}\,u^{-n}
 ,\quad
\psi^-_i(u)=\sum_{n\geqslant -w_i^-}\psi^-_{i,-n}\,u^{n}.
\end{align*}
Let $\U_F^{w^+,w^-}(L\frakg)$ be the $(w^+,w^-)$-shifted quantum loop group over $F$ 
with quantum parameter $q$. It is the $F$-algebra generated by 
$$x^\pm_{i,m}
 ,\quad
\psi^\pm_{i,\pm n}
 ,\quad
(\psi^\pm_{i,\mp w_i^\pm})^{-1}
\quad, \quad
i\in I
 ,\quad
m,n\in\bbZ
 ,\quad
n\geqslant-w^\pm_i$$
with the following defining relations  
where $a=+$ or $-$ and $i,j\in I$
\hfill
\begin{enumerate}[label=$\mathrm{(\alph*)}$,leftmargin=11mm,itemsep=1mm]
\item[(A.2)] $\psi^\pm_{i,\mp w_i^\pm}$ is invertible with inverse $(\psi^\pm_{i,\mp w_i^\pm})^{-1}$,
\item[(A.3)] $\psi^a_i(u)\,\psi^\pm_j(v)=\psi^\pm_j(v)\,\psi^a_i(u)$,
\item[(A.4)] $x^a_j(u)\,\psi_i^\pm(v)=\psi^\pm_i(v)\,x^a_j(u)\,g_{ij}(u/v)^a$,
\item[(A.5)] $x^\pm_i(u)\,x^\pm_j(v)=x^\pm_j(v)\,x^\pm_i(u)\,g_{ij}(u/v)^{\pm 1}$,
\item[(A.6)] $(q-q^{-1})[x^+_i(u)\,,\,x^-_j(v)]=\delta_{ij}\,\delta(u/v)\,(\psi^+_i(u)-\psi^-_j(u))$,
\item[(A.7)] the quantum Serre relations between 
$x^\pm_i(u_1),x^\pm_i(u_2),\dots,x^\pm_i(u_{1-c_{ij}})$ and $x^\pm_j(v)$ for $i\neq j$.
\end{enumerate}
\setcounter{equation}{7}
Here the rational function $g_{ij}(u/v)$ is expanded as a power series of $v^{\mp 1}$.
Let the element $h_{i,\pm m}$ in $\U_F^{w^+,w^-}(L\frakg)$ be such that
\begin{align*}
\psi_i^\pm(u)=\psi_{i,\mp w^\pm_i}^\pm u^{\pm w^\pm_i}
\exp\Big(\pm(q-q^{-1})\sum_{m>0}h_{i,\pm m}u^{\mp m}\Big)
 ,\quad
i\in I.
\end{align*}
Set
$[m]_q=(q^m-q^{-m})/(q-q^{-1})$
for each integer $m>0$.
The relation (A.4) is equivalent to the following relations
\begin{enumerate}[label=$\mathrm{(\alph*)}$,leftmargin=13mm,itemsep=1mm]
\item[(A.4a)] $x_{j,n}^a\,\psi_{i,\mp w^\pm_i}^\pm=q^{\pm ac_{ij}}\,\psi_{i,\mp w^\pm_i}^\pm\,x_{j,n}^a$,
\item[(A.4b)] $[h_{i,m},x^\pm_{j,n}]=\pm[m\,c_{ij}]_q\,x^\pm_{j,n+m}\,/\,m$ for $m\neq 0$.
\end{enumerate}
We have a triangular decomposition
$$\U_F^{w^+,w^-}(L\frakg)=\U_F^{w^+,w^-}(L\frakg)^+\otimes\U_F^{w^+,w^-}(L\frakg)^0\otimes
\U_F^{w^+,w^-}(L\frakg)^-$$
where $\U_F^{w_+,w_-}(L\frakg)^\pm$ 
is the subalgebra generated by the $x_{i,n}^\pm$'s
and $\U_F^{w_+,w_-}(L\frakg)^0$ is the subalgebra generated by 
the $\psi_{i,\pm n}^\pm$'s.
Set
\begin{align*}
[m]_q!=[m]_q[m-1]_q\cdots[1]_q
 ,\quad
(x^\pm_{i,n})^{[m]}=(x^\pm_{i,n})^{m}/[m]_q!.
\end{align*}
Let $\U_R^{w^+,w^-}(L\frakg)$ be the $R$-subalgebra of $\U_F^{w^+,w^-}(L\frakg)$ generated by 
\begin{align}\label{LGR}
\psi_{i,\mp w^\pm_i}^\pm 
 ,\quad
(\psi^\pm_{i,\mp w_i^\pm})^{-1}
 ,\quad
h_{i,\pm m}/[m]_q
 ,\quad
(x^\pm_{i,n})^{[m]}
\end{align}
with $
i\in I$,
$n\in\bbZ$
and 
$m\in\bbN^\times$.
We fix $\zeta\in\bbC^\times$ which is not a root of unity.
We define $\U_\zeta^{w^+,w^-}(L\frakg)=\U_R^{w^+,w^-}(L\frakg)|_\zeta$,
where $(-)|_\zeta$ is the specialization along the map $R\to\bbC$,
$q\mapsto\zeta$.
We'll concentrate on the module categories of the $\bbC$-algebra $\U_\zeta^{w^+,w^-}(L\frakg)$. 
The module categories of the $F$-algebra $\U_F^{w^+,w^-}(L\frakg)$ are similar.
Up to some isomorphism, the algebra $\U_\zeta^{w^+,w^-}(L\frakg)$ only depends on the sum 
$w=w^++w^-$ in $\bbZ I$.
Hence, we may assume that $w^+=0$ and we abbreviate
$\U_\zeta^{w}(L\frakg)=\U_\zeta^{0,w}(L\frakg)$.
We define $\U_F^{w}(L\frakg)$ and $\U_R^{w}(L\frakg)$ similarly.
The category $\bfO_w$ of $\U_\zeta^{w}(L\frakg)$-modules is defined as in 
\cite[def.~4.8]{H22}. 
A tuple $\Psi=(\Psi_i)_{i\in I}$ of rational functions over $\bbC$ 
such that $\Psi_i(u)$ is regular at 0 and of degree $w_i$
is called a $w$-dominant $\ell$-weight. Let
\begin{align*}
\Psi^+_i(u)=\sum_{n\in\bbN}\Psi_{i,n}^+\,u^{-n}
 ,\quad
\Psi^-_i(u)=\sum_{n\geqslant-w_i}\Psi_{i,-n}^-\,u^{n}
\end{align*}
be the expansions of the rational function $\Psi_i(u)$
in non negative powers of $u^{\mp 1}$.
A representation $V$ in the category $\bfO_w$ 
is of highest $\ell$-weight $\Psi(u)$ if it is generated by a vector $v$ such that
$$x_{i,n}^+\cdot v=0
 ,\quad
\psi^\pm_{i,n}\cdot v=\Psi^\pm_{i,n}v
 ,\quad
i\in I
 ,\quad
n\in\bbZ.$$
By \cite[thm.~4.11]{H22}
the simple objects in the category $\bfO_w$ are labelled by the $w$-dominant
$\ell$-weights.
Let $L(\Psi)$ be the unique simple object in $\bfO_w$ of highest $\ell$-weight $\Psi$.
For any module $V\in\bfO_w$ and for any tuple $\Psi=(\Psi_i(u))_{i\in I}$ of rational functions, 
the $\ell$-weight space of $V$ of $\ell$-weight $\Psi$ is
$$V_\Psi=\{v\in V\,;\,(\psi^\pm_{i,n}-\Psi^\pm_{i,n})^\infty\cdot v=0\,,\,i\in I\,,\,n\in\bbN\}.$$
The representation $V$ is a direct sum of its $\ell$-weight spaces.
The $q$-character of $V$ is the (possibly infinite) sum
$$q\!\ch(V)=\sum_\Psi\dim(V_\Psi)\,\Psi.$$
If the module $V$ admits an highest $\ell$-weight,
we may also consider the normalized $q$-character $q\widetilde\ch(V)$,
which is equal to the $q$-character $q\!\ch(V)$ divided by its highest weight monomial.
The map $q\!\ch$ is injective on the Grothendieck group $K_0(\bfO_w)$.
We'll abbreviate $I^\sharp=I\times\bbC^\times$.
For any tuple $w^\sharp=(w_{i,a})$ in $\bbN I^\sharp$, we consider the tuple
$\Psi^\pm_{w^\sharp}=(\Psi^\pm_{i})_{i\in I}$ such that
$$\Psi^\pm_{i}(u)=\prod_{a\in\bbC^\times}(1-a/u)^{\pm w_{i,a}}.$$
We'll write $L^\pm(w^\sharp)=L(\Psi^\pm_{w^\sharp})$.
We'll abbreviate $\Psi^\pm_{i,a}=\Psi^\pm_{\delta_{i,a}}$ and
$L^\pm_{i,a}=L^\pm(\delta_{i,a})$.
We call  $L^\pm_{i,a}$ the positive/negative
prefundamental representation. 
A positive prefundamental representation is one-dimensional,
a negative one is infinite dimensional.
We'll also abbreviate $L^\pm_{i,k}=L^\pm_{i,\zeta^k}$ for each integer $k$.
To avoid a cumbersome notation, we may use the symbol $w$ for the tuple $w^\sharp\in\bbN I^\sharp$
and we may write $L(w)$ for the corresponding simple module,
hoping it will not create any confusion.

\subsection{} \label{sec:qg}
Now, we consider non shifted quantum loop groups of symmetric types.
The quantum loop group $\U_F(L\frakg)$ is the quotient
of $\U_F^{0}(L\frakg)$ by the relations $\psi^+_{i,0}\,\psi^-_{i,0}=1$ for all $i\in I$.
We define the $R$-algebra $\U_R(L\frakg)$ and the $\bbC$-algebra $\U_\zeta(L\frakg)$ as above.
We have
$\U_F(L\frakg)=\U_R(L\frakg)\otimes_{R}F$.
We have a triangular decomposition
$$\U_F(L\frakg)=\U_F(L\frakg)^+\otimes\U_F(L\frakg)^0\otimes
\U_F(L\frakg)^-$$
and its analogues for the algebras
$\U_R(L\frakg)$ and $\U_\zeta(L\frakg)$ proved in \cite[prop.~6.1]{CP17}.
The $R$-algebra $\U_R(L\frakg)^\pm$ is generated by the quantum divided powers
$(x^\pm_{i,n})^{[m]}$ with
$i\in I$,
$n\in\bbZ$,
$m\in\bbN^\times$. 
Let
\begin{gather*}
\Big[\begin{matrix}\psi_{i,0}^+\,;\,n\\m\end{matrix}\Big]=
\prod_{r=1}^m\frac{q^{n-r+1}\,\psi_{i,0}^+-q^{-n+r-1}\,\psi_{i,0}^-}{q^r-q^{-r}}
 ,\quad
i\in I
 ,\quad
n\in\bbZ
 ,\quad
m\in\bbN^\times
\end{gather*}
The 
$R$-algebra $\U_R(L\frakg)^0$ is generated by the elements 
\begin{align*}
\psi^\pm_{i,0}
 ,\quad
h_{i,\pm m}/[m]_q
 ,\quad
\Big[\begin{matrix}\psi_{i,0}^+;n\\m\end{matrix}\Big]
 ,\quad
i\in I
 ,\quad
n\in\bbZ
 ,\quad
m\in\bbN^\times.
\end{align*}
A simple module $L(\Psi)$ in the category $\bfO_0$ is finite dimensional if and only if there is a tuple of polynomials
$P=(P_i)_{i\in I}$ with $P_i(0)=1$, called the Drinfeld polynomial, such that
the $\ell$-weight $\Psi$ is given by
\begin{align}\label{l-weight}
\Psi_i(u)=\zeta^{\deg P_i}\,P_i(1/\zeta u)\,P_i(\zeta/u)^{-1}.
\end{align}
For any tuple $w^\sharp=(w_{i,a})$ in $\bbN I^\sharp$, we consider the tuple of polynomials
$P_{w^\sharp}=(P_i(u))_{i\in I}$ given by
$P_{i}(u)=\prod_{a\in\bbC^\times}(1-au)^{w_{i,a}}.$
Let $\Psi_{w^\sharp}$ be the $\ell$-weight obtained by setting $P=P_{w^\sharp}$ in \eqref{l-weight}.
Let $L(w^\sharp)=L(\Psi_{w^\sharp})$ be the 
corresponding finite dimensional module.
The simple module 
$$KR_{i,a}^l=L(w_{i,a}^l),\quad
w^l_{i,a}=\delta_{i,a\zeta^{1-l}}+\delta_{i,a\zeta^{3-l}}+\cdots+\delta_{i,a\zeta^{l-1}}$$
is called a Kirillov-Reshetikhin module.
We may identify the $q$-character $q\!\ch(V)$ of a finite dimensional module $V\in\bfO_0$ 
with the sum of monomials $e^v$ such that
\begin{align*}
q\!\ch(V)=\sum_{v\in\bbZ I^\sharp}\dim(V_{\Psi_v})\,e^v
\end{align*}
where the $\ell$-weight $\Psi_v$ is given by
$\Psi_v=\Psi_{v_+}\cdot\Psi_{v_-}^{-1}$
with $v=v_+-v_-$ and
$v_+,v_-\in\bbN I^\sharp.$
The monomial $e^v$ is called $\ell$-dominant if $v\in\bbN I^\sharp$.
The module $V$ is simple whenever its $q$-character contains a unique $\ell$-dominant monomial,
see, e.g.,  \cite[\S 10]{N04}.
The following notation is standard
\begin{align}\label{AY}Y_{i,a}=e^{\delta_{i,a}}
 ,\quad
A_{i,a}=e^{\bfc\delta_{i,a}}=Y_{i,a\zeta}Y_{i,a\zeta^{-1}}\prod_{c_{ij}<0}Y_{j,a}^{-1}
,\quad
(i,a)\in I^\sharp.\end{align}
Thus $Y_{i,a}$ is the $\ell$-weight such that
$$(Y_{i,a})_i(u)=q\frac{1-a/qu}{1-q/u}
,\quad
(Y_{i,a})_j(u)=1
,\quad
j\neq i.$$
We'll view $I^\bullet$ as a subset of $I^\sharp$ such that
$(i,k)\mapsto(i,\zeta^k)$.
Hence we may write
$$Y_{i,k}=Y_{i,\zeta^k},\quad
A_{i,k}=A_{i,\zeta^k},\quad
KR_{i,k}^l=KR_{i,\zeta^k}^l,\quad
w_{i,k}^l=w_{i,\zeta^k}^l.$$
For each $v\in\bbZ I^\bullet$ we set $|v|=\max\{k\in\bbZ\,;\,\exists i\in I\,,\,v_{i,k}\neq 0\}$.
The monomial $e^v$ is called right-negative if we have
$v_{i,|v|}\leqslant 0$ for all $i\in I$. 
By \cite{FM01}, if the monomial $m$
 is right-negative then any monomial $m'$ in
$m\,\bbZ[A_{j,r}^{-1}\,;\,(j,r)\in I^\bullet]$ is also right-negative, hence $m'$ is not $\ell$-dominant.

\section{The shifted toroidal quantum group}\label{sec:toroidal}

In this section we give analogues of Theorems \ref{thm:Nakajima} and \ref{thm: Nakajima shifted}
for the Jordan quiver. In this case, our result is a K-theoretical analogue of 
\cite[thm.~1.3.2]{RSYZ20}.
Let $Q=A^{(1)}_0$.
We have $I=\{i\}$. 
We'll omit the subscript $i$ in the notation.
We equip the quiver $\widetilde Q_f$ with the potentials $\bfw_1$ or $\bfw_2$ as in \eqref{w12}.
Let $\bfw$ be the restriction of $\bfw_1$ or $\bfw_2$ to the quiver $\widetilde Q$.
We have $T=(\bbC^\times)^2$. We write
$R_T=\bbC[q^{\pm 1}, t^{\pm 1}]$ and
 $F_T=\bbC(q, t).$
The $T$-action on the representation space of $\widetilde Q_f$ is as in \eqref{action1}.
Let $\U_{F_T}^{0,-w}(L\widehat{\frakg\frakl}_1)$ be the $(0,-w)$-shifted 
toroidal quantum group of $\frakg\frakl_1$ for some integer $w\in\bbZ$.
Let $q_1=qt^{-1}$, $q_2=qt$ and $q_3=q^{-2}$.
Consider the rational function 
\begin{align}\label{alpha1}
g(u)=\prod_{i=1}^3(u-q_i^{-1})/(q_i^{-1}u-1).
\end{align}
Consider the formal series 
\begin{align*}
x^\pm(u)=\sum_{n\in\bbZ}x^\pm_{n}\,u^{\mp n}
 ,\quad
\psi^+(u)=\sum_{n\in\bbN}\psi^+_{n}\,u^{-n}
 ,\quad
\psi^-(u)=\sum_{n\geqslant w}\psi^-_{-n}\,u^{n}.
\end{align*}
The $F_T$-algebra $\U_{F_T}^{0,-w}(L\widehat{\frakg\frakl}_1)$ is generated by
$$x^\pm_{m}
 ,\quad
\psi^+_{n}
 ,\quad
\psi^-_{-w-n}
 ,\quad
(\psi^+_0)^{-1}
 ,\quad
(\psi^-_{-w})^{-1}
 ,\quad
i\in I
 ,\quad
m\in\bbZ
 ,\quad
n\in\bbN$$
with the following defining relations  
where $m\in\bbZ$ and $a=+$ or $-$ 
\hfill
\begin{enumerate}[label=$\mathrm{(\alph*)}$,leftmargin=11mm,itemsep=1mm]
\item[(B.2)]
$\psi^+_0$ and $\psi^-_{-w}$ are invertible with inverses 
$(\psi^+_0)^{-1}$ and $(\psi^-_{-w})^{-1}$,
\item[(B.3)]
$\psi^a(u)\,\psi^\pm(v)=\psi^\pm(v)\,\psi^a(u)$,
\item[(B.4)]
$x^a(u)\,\psi^\pm(v)=\psi^\pm(v)\,x^a(u)\,g(u/v)^a$,
\item[(B.5)]
$x^\pm(u)\,x^\pm(v)=x^\pm(v)\,x^\pm(u)\,g(u/v)^{\pm 1}$,
\item[(B.6)]
$[x_m^\pm\,,\,[x_{m+1}^\pm\,,\,x_{m-1}^\pm]]=0$,
\item[(B.7)]
$(1-q_1)(1-q_2)(1-q_3)\,[x^+(u)\,,\,x^-(v)]=\delta(u/v)\,(\psi^+(u)-\psi^-(u))$.
\end{enumerate}
\setcounter{equation}{7}
Note that $\psi^+_0\,\psi^-_{-w}$ is central in $\U_{F_T}^{0,-w}(L\widehat{\frakg\frakl}_1)$.
We'll abbreviate $\U_{F_T}^{-w}(L\widehat{\frakg\frakl}_1)=\U_{F_T}^{0,-w}(L\widehat{\frakg\frakl}_1)$.
The $F_T$-algebra $\U_{F_T}^{-w}(L\widehat{\frakg\frakl}_1)$ has a triangular decomposition
$$\U_{F_T}^{-w}(L\widehat{\frakg\frakl}_1)=\U_{F_T}^{-w}(L\widehat{\frakg\frakl}_1)^+\otimes_{F_T}
\U_{F_T}^{-w}(L\widehat{\frakg\frakl}_1)^0\otimes_{F_T}\U_{F_T}(L\widehat{\frakg\frakl}_1)^-$$
where $\U_{F_T}^{-w}(L\widehat{\frakg\frakl}_1)^\pm$ 
is the subalgebra generated by the $x_n^\pm$'s
and $\U_{F_T}^{-w}(L\widehat{\frakg\frakl}_1)^0$ is the subalgebra generated by the
$\psi_{\pm n}^\pm$'s.
The defining relations of $\U_{F_T}(L\widehat{\frakg\frakl}_1)^\pm$ are the relations
(B.5) and (B.6).
The proof is as in the non shifted case in \cite{T17}.

Forgetting the arrow $a^*$ yields a vector bundle
$\rho_2$
as in \eqref{rho} such that $\tilde f_2=\hat f_2\circ\rho_2$ for some function $\hat f_2$ on $\widehat\frakM(W)$.
Thus, Proposition \ref{prop:Thom} yields an isomorphism
$K_{G_W\times T}(\widehat\frakM(W),\hat f_2)= K_{G_W\times T}(\widetilde\frakM(W),\tilde f_2).$
Applying the results in \S\ref{sec:critalg1} as in \S\ref{sec:Jordan2}, we get an $F_T$-algebra
$$K_{G_W\times T}(\widehat\frakM(W)^2,(\hat f_2)^{(2)})_{\widehat\calZ(W)}\otimes_{R_{G_W\times T}}F_{G_W\times T}$$
and a representation in
$K_{G_W\times T}(\widehat\frakM(W),\hat f_2)\otimes_{R_{G_W\times T}}F_{G_W\times T}.$
The proof of the following theorem is similar to the proof of Theorem \ref{thm: Nakajima shifted}.

\begin{theorem}\label{thm:Jshifted}
Assume that $Q=A^{(1)}_0$. 
\hfill
\begin{enumerate}[label=$\mathrm{(\alph*)}$,leftmargin=8mm]
\item
There is an $F_T$-algebra map 
$$\U_{F_T}^{-w}(L\widehat{\frakg\frakl}_1)\to 
K_{G_W\times T}(\widehat\frakM(W)^2,(\hat f_2)^{(2)})_{\widehat\calZ(W)}\otimes_{R_{G_W\times T}}F_{G_W\times T}$$
which takes the central element $\psi^+_0\,\psi^-_{-w}$ to $(-q)^{-w}\det(W)^{-1}$.
\item
The $F_T$-algebra $\U_{F_T}^{-w}(L\widehat{\frakg\frakl}_1)$ acts on 
$K_{G_W\times T}(\widehat\frakM(W),\hat f_2)\otimes_{R_{G_W\times T}}F_{G_W\times T}$.
\end{enumerate}
\qed
\end{theorem}

\begin{remark}
\hfill
\begin{enumerate}[label=$\mathrm{(\alph*)}$,leftmargin=8mm]
\item
If $w=1$ then the critical locus of the function $\hat f_2$ 
in $\widehat\frakM(W)$ is the punctual Hilbert scheme of $\bbC^3$ with
$G_W$ acting on the framing and $T$ on the coordinates, and 
 $K_{G_W\times T}(\widehat\frakM(W),\hat f_2)$ is isomorphic to
the critical K-theory group of $\Hilb(\bbC^3)$ defined in \cite[\S4.2]{P21}.

\item
The toroidal quantum group $\U_{F_T}(L\widehat{\frakg\frakl}_1)$ is the quotient
of $\U_{F_T}^{0}(L\widehat{\frakg\frakl}_1)$ by the relation 
$\psi^+_0\,\psi^-_0=1$.
There is a triangular decomposition
$$\U_{F_T}(L\widehat{\frakg\frakl}_1)=\U_{F_T}(L\widehat{\frakg\frakl}_1)^+\otimes_{F_T}
\U_{F_T}(L\widehat{\frakg\frakl}_1)^0\otimes_{F_T}\U_{F_T}(L\widehat{\frakg\frakl}_1)^-$$
where $\U_{F_T}(L\widehat{\frakg\frakl}_1)^\pm$ is the subalgebra generated by the $x_n^\pm$'s
and $\U_{F_T}(L\widehat{\frakg\frakl}_1)^0$ the subalgebra generated by the $\psi_{\pm n}^\pm$'s.
The defining relations of $\U_{F_T}(L\widehat{\frakg\frakl}_1)^\pm$ are 
(B.5) and (B.6).
Given $W\in\bfC^\bullet$, $w=\dim W$, and
applying the results in \S\ref{sec:critalg1} as in \S\ref{sec:Jordan1}, we get an $F_T$-algebra homomorphism
$$\U_{F_T}(L\widehat{\frakg\frakl}_1)\to 
K_{G_W\times T}\big(\widetilde\frakM(W)^2,(\tilde f_1)^{(2)}\big)_{\widetilde\calZ(W)}
\otimes_{R_T}F_T=
K^{G_W\times T}(\calZ(W))\otimes_{R_T}F_T
$$
and a representation of $\U_{F_T}(L\widehat{\frakg\frakl}_1)$ on the $F_T$-vector space
$$K_{G_W\times T}(\widetilde\frakM(W),\tilde f_1)\otimes_{R_T}F_T=
K^{G_W\times T}(\frakM(W))\otimes_{R_T}F_T.$$
Let  $\calM_w$ be the moduli space of rank $w$ instantons over $\bbC^2$ with
the obvious $T$-action. 
The representation of $\U_{F_T}(L\widehat{\frakg\frakl}_1)$ above
is isomorphic to the representation in 
$K^{G_W\times T}(\calM_w)\otimes_{R_T}F_T$ 
given in \cite{SV13a}, and in \cite{FT11} if $w=1$.

\end{enumerate}
\end{remark}

\section{A second proof of Theorem \ref{thm:HL1}}\label{sec:proof2}

Let $(W,A,\gamma)$ be a regular admissible triple.
In this section we give a second proof of a version of Theorem \ref{thm:HL1} assuming
that  the set of closed points of $\frakL^\bullet(W)$ is finite.
More precisely, we'll prove that there is an isomorphism 
of $\U_\zeta(L\frakg)$-modules $H^\bullet(\frakM^\bullet(W),f_\gamma^\bullet)= L(w)$.
To do this, let
$\calE=\pi^\bullet_*\calC_{\frakM^\bullet(W)}.$
Since vanishing cycles commute with proper push-forwards, we have an isomorphism
$$H^\bullet(\frakM^\bullet(W),f_\gamma^\bullet)=
H^\bullet(\frakM_0^\bullet(W),\phi^p_{\!f^\bullet_0}\calE).$$
The complex $\calE$  is semi-simple.
We have a stratification $\calS$ by locally closed subsets
\begin{align*}
\frakM^\bullet_0(W)=\bigsqcup_{v\in\bbN I^\bullet}\frakM_0^{\bullet\reg}(v,W)
 ,\quad
\frakM^{\bullet\reg}(v,W)=(\pi^\bullet)^{-1}(\frakM_0^{\bullet\reg}(v,W))
\end{align*}
such that the following holds
\begin{itemize}[label=$\mathrm{(\alph*)}$,leftmargin=4mm]
\item[-]
$\frakM_0^{\bullet\reg}(v,W)=\frakM_0^{\reg}(v,W)^A$ under the isomorphism \eqref{GA}.
\item[-]
$\frakM_0^{\bullet\reg}(v',W)\subset\overline{\frakM_0^{\bullet\reg}(v,W)}$
if and only if $v'\leqslant v$.
\item[-] 
$\frakM_0^{\bullet\reg}(v,W)=\{\ux\,;\, $the $G^0_v$-orbit of $x$
is free and closed in ${}^f\overline\X^\bullet\}\,/\,G^0_v$.
\item[-] 
$\frakM_0^{\bullet\reg}(v,W)\neq\emptyset$ if and only if 
$\frakM^\bullet(v,W)\neq\emptyset$ and $(v,w)$ is $l$-dominant.
\item[-]
$\frakM^{\bullet\reg}(v,W)$ is open and dense in
$\frakM^\bullet(v,W)$. 
\item[-]
$\pi^\bullet$ is an isomorphism $\frakM^{\bullet\reg}(v,W)\to\frakM_0^{\bullet\reg}(v,W)$.
\end{itemize}
The strata may not be connected.
The connected components of the strata of $\calS$ form a
Whitney stratification.
Let $\IC_{\frakM_0^{\bullet\reg}(v,W)}$ be the intermediate extension of $\calC_{\frakM_0^{\bullet\reg}(v,W)}$.
The category $\D^\b_{G^0}(\frakM_0^\bullet(W))$ 
is $\bbZ$-graded by the cohomological shift functor. 
We have
$$\calE=\bigoplus_{v}\IC_{\frakM_0^{\bullet\reg}(v,W)}\otimes_\bbZ M_v,$$
where $M_v$ is a graded vector space, and
$\IC_{\frakM_0^{\bullet\reg}(v,W)}=0$ whenever $\frakM_0^{\bullet\reg}(v,W)=\emptyset$.
By \cite{N00} there is a vector space isomorphism
$M_v=L(w-\bfc v).$
We claim that
\begin{align}\label{vanishing1} v\neq 0\Rightarrow\phi^p_{\!f^\bullet_0}\IC_{\frakM_0^{\bullet\reg}(v,W)}=0.
\end{align}
Then, we have $H^\bullet(\frakM^\bullet(W),f_\gamma^\bullet)=M_0$, and,
setting $j$ and $\kappa$ to be the inclusions $\frakL^\bullet(W)\subset\frakM^\bullet(W)$
and $\{0\}\subset\frakM^\bullet_0(W)$, 
the base change theorem yields
\begin{align*}
H^\bullet(\frakM^\bullet(W),f^\bullet_\gamma)_{\frakL^\bullet(W)}&=
H^\bullet(\frakL^\bullet(W),j^!\phi^p_{\!f^\bullet}\calC_{\frakM^\bullet(W)})=\kappa^!\phi^p_{\!f^\bullet_0}\calE=M_0.
\end{align*}
This proves the theorem.

Now, we prove the claim.
It follows from the lemmas \ref{lem:micro1} and \ref{lem:micro2} below.
The generators of the coordinate ring of $\frakM_0(W)$ given in \cite{L98}
yield an $A$-invariant closed embedding of $\frakM_0(W)$
into a linear representation $E$ of $A$ equipped with an $A$-invariant linear function 
$f:E\to\bbC$ such that $f_0=f|_{\frakM_0(W)}$.
Taking the fixed points by the $A$-action, we get the inclusion $\frakM_0^\bullet(W)\subset E^A$.

\begin{lemma}\label{lem:micro1}
If $v\neq 0$, then
$d_xf\notin T^\vee_{\frakM_0^{\bullet\reg}(v,W)}E^A$
for each $x\in \frakM_0^{\bullet\reg}(v,W)$.
\end{lemma}

\begin{proof}
We must check that for any $x\in \frakM_0^{\bullet\reg}(v,W)$ the differential
$d_xf$ in $T_x^\vee E^A$
does not annihilates the subspace $T_x\frakM_0^{\bullet\reg}(v,W).$
Since the function $f$ on $E$ is 
$A$-invariant, the differential vanishes on the complementary $A$-module
$T_xE\ominus T_xE^A$.
Thus it vanishes on the complementary $A$-module 
$$T_x\frakM_0^{\reg}(v,W)\ominus T_x\frakM_0^{\bullet\reg}(v,W).$$
Hence, we must prove that
$d_xf(T_x\frakM_0^{\reg}(v,W))\neq 0$.
Let $G'_W$ be the derived subgroup of $G_W$ and $\frakg'_W$ be its Lie algebra.
We claim that 
\begin{align*}
\langle\frakg'_W\,,\,d_x\mu_0(T_x\frakM_0^{\reg}(v,W))\rangle\neq 0
\end{align*}
We deduce that $d_xf(T_x\frakM_0^{\reg}(v,W))\neq 0$, because
\begin{align*}
d_xf(T_x\frakM_0^{\reg}(v,W))=0
&\Rightarrow \langle\gamma\,,\,d_x\mu_0(T_x\frakM_0^{\reg}(v,W))\rangle=0\\
&\Rightarrow \langle[\frakg_W,\gamma]\,,\,d_x\mu_0(T_x\frakM_0^{\reg}(v,W))\rangle=0\\
&\Rightarrow \langle\frakg'_W\,,\,d_x\mu_0(T_x\frakM_0^{\reg}(v,W))\rangle=0.
\end{align*}
The first line is the definition of $f$, the second one the $G_W$-invariance of $\mu_0$,
and the third one the regularity of $\gamma$. 
Now we prove the claim.
Assume that for some $x\in\frakM_0^{\reg}(v,W)$ we have
$$\langle\frakg'_W\,,\,d_x\mu_0(T_x\frakM_0^{\reg}(v,W))\rangle=0.$$
Since the variety $\frakM_0^{\reg}(v,W)$ is smooth and 
$G_W$-Hamiltonian with moment map $\mu_0$,
the infinitesimal $\frakg'$-action on $\frakM_0^{\reg}(v,W)$ vanishes at the point $x$.
Let $\SG_W\subset G'_W$ be a maximal torus, and $\frak\SG_W$ be its Lie algebra.
The $\SG_W$-action on $\frakM_0^{\reg}(v,W)$ extends to a linear $\SG_W$-action on the vector space $E$.
Since the point $x$ is killed by the infinitesimal action of $\frak\SG_W$, it is also fixed by the action
of the torus $\SG_W$. Using the map $\pi$, we may identify $x$ with a point
in the fixed points locus
$\frakM^{\reg}(v,W)^{\SG_W}$.
The $I$-graded vector space $W$ splits as a direct sum of
one dimensional $I$-graded $\SG_W$-submodules $W^1,W^2,\dots,W^s$. 
The $\SG_W$-fixed points locus in $\frakM(v,W)$ decomposes as the Cartesian product
$$\frakM(v,W)^{\SG_W}=\bigsqcup_{v^1,\dots,v^s}\prod_{r=1}^s\frakM(v^r,W^r)$$
where $v^1,v^2,\dots,v^s$ run over the set of all tuples of dimension vectors in $\bbN I$ with
sum $v$.
Under this isomorphism we have
$$\frakM^{\reg}(v,W)^{\SG_W}
=\bigsqcup_{v^1,\dots,v^s}\prod_{r=1}^s\frakM^{\reg}(v^r,W^r)
$$
Since $W^r$ is of dimension 1 for all $r$, we have $\frakM^{\reg}(v^r,W^r)=\emptyset$ unless
$v^r=0$ by \cite[prop.~4.2.2]{N00}. Thus $v=0$, yielding a contradiction.
\end{proof}

We now prove  the claim \eqref{vanishing1}.
For each $\calS$-constructible complex $\calE\in \D^\b_{G^0}(\frakM_0^\bullet(W))$, 
let $\SS(\calE)$, $\CC(\calE)$ be the singular support and the characteristic cycle of 
the pushforward of $\calE$ to $E^A$.
We have
\begin{align}\label{micro1}
\CC(\calE)=\sum_{v\in \calS}c_v(\calE)\,[\overline{T^\vee_{\frakM_0^{\bullet\reg}(v,W)}E^A}]
 ,\quad
\SS(\calE)=\bigcup_{c_v(\calE)\neq 0}\overline{T^\vee_{\frakM_0^{\bullet\reg}(v,W)}E^A}
\end{align}
The integer
$c_v(\calE)$
is the microlocal multiplicity along $\frakM_0^{\bullet\reg}(v,W)$.
By \cite[(8.6.12)]{KS90} we have
\begin{align}\label{micro2}
\supp(\phi^p_{\!f^\bullet_0}\calE)\subset\{x\in \frakM_0^\bullet(W)\,;\,d_xf\in\SS(\calE)\},
\end{align}
Hence, to prove \eqref{vanishing1} it is enough to check that
\begin{align}\label{SS}d_xf\notin\SS(\IC_{\frakM_0^{\bullet\reg}(v,W)}) ,\quad v\neq 0 ,\quad x\in \frakM_0^\bullet(W).\end{align}
Since the stratification $\calS$ is Whitney, we have
$$\SS(\IC_{\frakM_0^{\bullet\reg}(v,W)})\subseteq
\bigcup_{v'\in\calS}T^\vee_{\frakM_0^{\bullet\reg}(v',W)}E^A,$$
see, e.g., \cite[rem.~4.3.16]{D04}.
Lemma \ref{lem:micro1} yields
\begin{align*}d_xf\notin T^\vee_{\frakM_0^{\bullet\reg}(v',W)}E^A
 ,\quad 
v'\neq 0 ,\quad x\in \frakM_0^{\bullet\reg}(v',W).\end{align*}
Thus, we are reduced to prove the following.

\begin{lemma}\label{lem:micro2}
If $v\neq 0$, then we have $c_0(\IC_{\frakM_0^{\bullet\reg}(v,W)})=0$.
\end{lemma}

\begin{proof}
To compute the microlocal multiplicity 
we fix a generic cocharacter $b:\bbC^\times\to T$ 
which contracts $\frakM_0^\bullet(W)$ to 0.
The cocharacter $b$ acts on $\frakM^\bullet(v,W)$.
The $b$-action contracts $\frakM^\bullet(v,W)$ to the central fiber $\frakL^\bullet(v,W)$,
which is a finite set.
We write 
$$\C_\tau=\{x\in\frakM^\bullet(v,W) \,;\,\lim_{t\to 0}b(t)\cdot x=\tau\}
 ,\quad
\tau\in \frakL^\bullet(v,W).$$
The Byalinicki-Birula theorem yields the paving
$\frakM^\bullet(v,W)=\bigsqcup_\tau \C_\tau$.
The cells $\C_\tau$ are affine spaces. 
They are closed in $\frakM^\bullet(v,W)$, 
because $\frakM^\bullet(v,W)$ is homeomorphic to arbitrary 
small neighborhoods $U$ of the central fiber
$\frakL^\bullet(v,W)$
and the intersection  $\C_\tau\cap U$ is closed in $U$ if the set $U$ is small enough.
Thus, the cells $\C_\tau$ are the connected components of $\frakM^\bullet(v,W)$ and
the map $\pi^\bullet$ is the sum of its restrictions to the cells $\C_\tau$. 
For each $\tau$, the map $\pi^\bullet|_{\C_\tau}$
is a closed embedding, because the map $\pi^\bullet$ is projective and 
$\C_\tau$ is affine.
Thus, we have
$\calE=\bigoplus_\tau\calC_{\C_{0,\tau}}$
where $\C_{0,\tau}=\pi^\bullet(\C_\tau)$ for each $\tau$.
Since $\C_{0,\tau}$ is smooth, we have
$c_0(\calC_{\C_{0,\tau}})=0$ whenever $\C_{0,\tau}\neq\{0\}$.
We deduce that $c_0(\IC_{\frakM_0^{\bullet\reg}(v,W)})=0$ whenever $v\neq 0$, because
$\IC_{\frakM_0^{\bullet\reg}(v,W)}$ is a direct summand of $\calE$.
\end{proof}

Using Maffei's isomorphism \cite{M05} it is easy to check the following.

\begin{proposition}\label{prop:type A}
The set of points of $\frakL^\bullet(W)$ is finite for $Q$ of type $A$.
\qed
\end{proposition}

\section{The algebraic and topological critical K-theory}\label{sec:Ktheory}

In this section we discuss some topological analogues of the Grothendieck groups following 
\cite{B16} and \cite{HP20}.
To do that, for $\flat=\alg$ or $\top$ we'll use the functor $\bfK^\flat$
from the category of all dg-categories over $\bbC$ 
to the category of spectra introduced in \cite{S06} and \cite{B16}.
Let $X^{an}$ be the underlying complex analytic space of a scheme $X$.
Given a closed subset $Y$ of $X$ we'll say that $Y^{an}$ is homotopic to $X^{an}$, or that
$Y$ is homotopic to $X$, if the inclusion $Y^{an}\subset X^{an}$ admits a deformation retraction 
$X^{an}\to Y^{an}$.
The following properties hold :
\begin{itemize}[leftmargin=3mm]
\item[-]
$\bfK^\alg(\calC)$ is the algebraic $K$-theory spectrum of the 
category $H^0(\overline\calC)$,
\item[-]
there is natural topologization map $\top:\bfK^\alg\to \bfK^\top$,
\item[-]
$\bfK^\flat$ takes localization sequences of dg-categories
to exact triangles.
\end{itemize}
Next, for any $G$-invariant closed immersion $Z\subset X$ we write
\begin{align*}
\bfK^G_\flat(X)_Z=\bfK^\flat(\Db\Coh_G(X)_Z)
 ,\quad
\bfK_G^\flat(X)=\bfK^\flat(\Perf_G(X))
\end{align*}
and  $\bfK_\flat^G(X)=\bfK_\flat^G(X)_X$. The following properties hold :
\begin{itemize}[leftmargin=3mm]
\item[-]
$\bfK_\flat^G$ is covariantly functorial for proper morphisms of
$G$-schemes, and contravariantly functorial for
 finite $G$-flat dimensional morphisms,
 \item[-]
$\bfK_\flat^G$ satisfies the flat base change and the projection formula,
\item[-]
$\bfK_\flat^G$ satisfies equivariant d\'evissage : there is a weak equivalence
$\bfK_\flat^G(Z)\to\bfK_\flat^G(X)_Z$,
\item[-]
$\bfK_\top^G(X)$ is the 
$G$-equivariant Borel-Moore K-homology spectrum of $X^{an}$, and
$\bfK^\top_G(X)$ is its 
$G$-equivariant K-theory spectrum, up to weak equivalences.
\end{itemize}

The Grothendieck groups $K^G(Z)$ and $K_G(Z)$
satisfy
$K_G(Z)=\pi_0\bfK_G^\alg(Z)\otimes\bbC$ and
$K^G(Z)=\pi_0\bfK^G_\alg(Z)\otimes\bbC$.
The $G$-equivariant Borel-Moore K-homology of $X$ 
and its $G$-equivariant K-theory are 
\begin{align}\label{Ktop1}
K_G^\top(Z)=\pi_0\bfK_G^\top(Z)\otimes\bbC
 ,\quad
K^G_\top(Z)=\pi_0\bfK^G_\top(Z)\otimes\bbC
\end{align}

Now, let $(X,\chi,f)$ be a $G$-equivariant LG-model.
Let $Y\subset X$ be the zero locus of $f$, 
$i$ be the closed embedding $Y\to X$, and
$Z\subset Y$ a closed $G$-invariant subset.
We define
\begin{align}\label{Ktop2}
\begin{split}
K_G(X,f)_Z&=K_0(\DCoh_G(X,f)_Z),\\
K_G^\top(X,f)_Z&=\pi_0\bfK^\top(\DCoh_G(X,f)_Z)\otimes\bbC,\\
K_G^\alg(X,f)_Z&=\pi_0\bfK^\alg(\DCoh_G(X,f)_Z)\otimes\bbC.
\end{split}
\end{align}
By \cite[cor.~2.3]{T97} there is an inclusion
$K_G(X,f)_Z\subset K_G^\alg(X,f)_Z.$
The functor \eqref{conv1} yields 
an associative $R_G$-algebra structure on 
$K^G_\flat(Z)$ and a representation on $K^G_\flat(L)$ and $K^G_\flat(X)$.
The functor \eqref{conv2} yields the following.

\begin{proposition}\label{prop:critalg2}
\hfill
\begin{enumerate}[label=$\mathrm{(\alph*)}$,leftmargin=8mm]
\item
$K_G^\flat(X^2,f^{(2)})_Z$ is an $R_G$-algebra which acts on
$K_G^\flat(X,f)_L$ and $K_G^\flat(X,f)$.

\item
The functor $\Upsilon$ yields an algebra homomorphism 
$K^G_\flat(Z)\to K_G^\flat(X^2,f^{(2)})_Z$
and an intertwiner
$K^G_\flat(L)\to K_G^\flat(X,f)_L.$

\end{enumerate}
\end{proposition}

\begin{proof}
Since the functor $K^G_\flat$ is localizing, satisfies equivariant d\'evissage and flat base change 
and \cite[\S 2.1]{HP20} for more details, the corollary is proved as in \S\ref{sec:Kcrit}. 
\end{proof}

\begin{remark}\label{rem:Xi}
Let $(X,\chi,f)$ be a smooth $G$-equivariant LG-model.
Let $Y$ be the zero locus of $f$, 
and $i$ the closed embedding $Y\subset X$.
Assume that $Y^{an}$ is homotopic to $X^{an}$.
Then, there is a map
$\Xi:K_G(X,f)\to K^G_\top(X)$ 
such that $\Xi\circ\Upsilon$ is the composition of the
pushforward $K^G(Y)\to K^G(X)$ and the topologization map.
We'll not need this map $\Xi$.

\end{remark}


\begin{thebibliography}{20}
 \bibitem{BDFIK16} Ballard, M., Deliu, D., Favero, D., Isik, U., Katzarkov, L., Resolutions in factorization categories, Adv. Math. 295 (2016), 195-249.
 \bibitem{BFK13} Ballard, M., Favero, D., Katzarkov, L., A Category of Kernels For Equivariant Factorizations and Its Implications For Hodge Theory, Publ. Math. IHES, 2014, 1-111. 
 \bibitem{BFK19} Ballard, M., Favero, D., Katzarkov, L., Variation of geometric invariant theory quotients and derived categories (2019), 235-303.
 \bibitem{BF08} Behrend, K., Fantechi, B., Symmetric obstruction theories and Hilbert schemes of points on threefolds, Algebra Number Theory, 2 (2008), 313-345.
\bibitem{B16} Blanc, A., Topological K-theory of complex noncommutative spaces, Compos. Math. 152 (2016),
489-555.
\bibitem{CP17} Chari, V., Pressley, A., Quantum affine algebras at roots of unity, Representation Theory 1
(1997), 280-328.
\bibitem{CG} Chriss, N., Ginzburg, V., Representation theory and complex geometry, Birkh\"auser, 1997.
\bibitem{D16} Davison, B., The integrality conjecture and the cohomology of preprojective stacks, arXiv:1602.02110.
\bibitem{D17} Davison, B., The critical CoHA of a quiver with potential,
Quart. J. Math. 68 (2017), 635-703.
\bibitem{DP22} Davison, B., Padurariu, T., Deformed dimensional reduction, Geom. Topol. 26 (2022), 
721-776.
\bibitem{D04} Dimca, A., Sheaves in topology, Springer-Verlag, 2004.
\bibitem{EP15} Efimov, A.I., Positselski, L., Coherent analogues of matrix factorizations and relative singularity categories, Algebra Number Theory 9 (2015), 1159-1292.
\bibitem{FT11} Feigin, B.L., Tsymbaliuk, A.I., Equivariant K-theory of Hilbert schemes via shuffle algebra, Kyoto J. Math. 51 (2011) 831-854.
\bibitem{FT19} Finkelberg, F., Tsymbaliuk, A., Multiplicative slices, relativistic Toda and shifted quantum affine algebras, in Representations and nilpotent orbits of Lie algebraic systems, vol. 330, Progr. Math., Birkh\"auser/Springer,  2019, 133-304.
\bibitem{FH15} Frenkel, E., Hernandez, D., Baxter's relations and spectra of quantum integrable models,
Duke Math. J. 164 (2015), 2407-2460.
\bibitem{FM01} Frenkel, E., Mukhin, E., Combinatorics of $q$-characters of finite dimensional representations of quantum affine algebras, Comm. Math. Phy. 216 (2001), 23-57.
\bibitem{FM21} Fujita, R., Murakami, K., Deformed Cartan matrices and generalized
preprojective algebras of finite type, arXiv:2109.07985.
\bibitem{GKV} Ginzburg, V.,  Kapranov, M., Vasserot, E., Langlands reciprocity for algebraic
surfaces, Math. Res. Lett. 2 (1995), 147-160.
\bibitem{HP20} Halpern-Leistner, D., Pomerleano, D., Equivariant Hodge theory and noncommutative geometry, Geom. Topol. 24 (2020), 2361-2433.
\bibitem{H06} Hernandez, D., Kirillov-Reshetikhin conjecture and solutions of T-systems, J. Reine angew. Math.
596 (2006), 63-87.
\bibitem{H22} Hernandez, D., Representations of shifted quantum affine algebras, arXiv:2010.06996.
\bibitem{HJ12} Hernandez, D., Jimbo, M., Asymptotic representations and Drinfeld rational fractions, Compos. Math. 148 (2012), 1593-1623.
\bibitem{HL16} Hernandez, D., Leclerc, B., A cluster algebra approach to $q$-characters of Kirillov-Reshetikhin
modules, J. Eur. Math. Soc. 18 (2016), 1113-1159.
\bibitem{HL16a} Hernandez, D., Leclerc, B., Cluster algebras and category $\calO$ for representations of Borel subalgebras of quantum affine algebras, Algebra and number theory 10 (2016), 2015-2052.
\bibitem{H17a} Hirano, Y., Equivalences of derived factorization categories of gauged Landau-Ginzburg models, Adv. Math. 306 (2017), 200-278.
\bibitem{H17b} Hirano, Y., Derived Kn\"orrer periodicity and Orlov's theorem for gauged Landau-Ginzburg models, Compos. Math. 153 (2017) 973-1007.
\bibitem{I12} Isik, U., Equivalence of the derived category of a variety with a singularity category,  
International Math. Research Notices 12 (2013), 2787-2808. 
\bibitem{K22}  Khan, A.A., K-theory and G-theory of derived algebraic stacks. Jpn. J. Math. 17 (2022), 1-61.
\bibitem{KKOP20} Kashiwara, M., Kim, M., Oh, S., Park, E., Monoidal categorification and quantum affine 
algebras, Compos. Math. 156 (2020), 1039-1077.
\bibitem{KS90} Kashiwara, M., Schapira, P., Sheaves on manifolds, Springer-Verlag, 1990.
\bibitem{KS11} Kontsevich, M., Soibelman, Y., Cohomological Hall algebra, exponential Hodge
structures and motivic Donaldson-Thomas invariants, Commun. Number Theory Phys. 5 (2011), 231-252.
\bibitem{L20} Liu, H., Asymptotic representations of shifted quantum affine algebras from critical K-theory,
PhD thesis, Columbia University, 2021.
\bibitem{L18} Lurie, J., Spectral algebraic geometry, version 2017-09-18, https://www.math.ias.edu/$\sim$lurie/papers/SAG-rootfile.pdf.
\bibitem{L98} Lusztig, G., On quiver varieties, Adv. Math. 136 (1998), 141-182.
\bibitem{M05} Maffei, A., Quiver varieties of type A, Comment. Math. Helv. 80 (2005), 1-27.
\bibitem{MR22} Monavari, S., Ricolfi, A.T., On the motive of the nested Quot scheme of points on a curve, J. Alg. 610 (2022), 99-118.
\bibitem{N94} Nakajima, H.,  Instantons on ALE spaces, quiver varieties, and Kac-Moody algebras,
Duke Math. J. 76 (1994), 365-416.
\bibitem{N98} Nakajima, H.,  Quiver varieties and Kac-Moody algebras, Duke Math. J. 91 (1998), 515-560.
\bibitem{N00} Nakajima, H., Quiver varieties and finite-dimensional representations of quantum affine algebras, 
J. Amer. Math. Soc. 14 (2001), 145-238.
\bibitem{N03} Nakajima, H., $t$-analogs of $q$-characters of Kirillov-Reshetikhin modules of quantum affine algebras, Represent. Th. 7 (2003), 259-274.
\bibitem{N04} Nakajima, H., Quiver varieties and $t$-analogs of $q$-characters of quantum affine algebras
Ann. Maths. 160 (2004), 1057-1097.
\bibitem{N11} Nakajima, H., Quiver varieties and cluster algebras, Kyoto J. Math 51 (2011), 71-126.
\bibitem{P21} Padurariu, T., Categorical and K-theoretic Hall algebras for quivers with potential, arXiv:2107.13642.
\bibitem{PS21} Pavic, N., Shinder, E., K-theory and the singularity category of quotient singularities,
Ann. K-Theory 6 (2021), 381-424. 
\bibitem{PV11} Polishchuk, A., Vaintrob, A., Matrix factorizations and singularity categories for stacks,
 Ann. Inst. Fourier 61 (2011), 2609-2642.
\bibitem{RSYZ20} Rapc\'ak, M., Soibelman, Y., Yang, Y., Zhao, G., Cohomological Hall algebras and perverse coherent sheaves on toric Calabi-Yau 3-folds,
 arXiv:2007.13365v2.
 \bibitem{RS17} Ren, J., Soibelman, Y., Cohomological Hall algebras, semicanonical bases and Donaldson-Thomas invariants for 2-dimensional Calabi-Yau categories (with an appendix by Ben Davison). (English summary) Algebra, geometry, and physics in the 21st century, 261-293, 
Progr. Math., 324, Birkhauser/Springer, Cham, 2017. 
\bibitem{SGA5} SGA5, 
Cohomologie $l$-adique et fonctions L, S\'eminaire de G\'eom\'etrie Alg\'ebrique du Bois-Marie
 1965-1966, Edit\'e par L. Illusie, Lectures Notes in Mathematics, vol. 589, 1977.
\bibitem{SV13a} Schiffmann, O., Vasserot, E., The elliptic Hall algebra and the K-theory of the Hilbert scheme of $\bbC^2$, Duke Math. J. 162 (2013), 279-366.
\bibitem{SV13b} Schiffmann, O., Vasserot, E., Cherednik algebras, W-algebras and the equivariant cohomology of the moduli space of instantons on A2, Publ. Math. Inst. Hautes Etudes Sci. 118 (2013), 213-342.
\bibitem{SV18} Schiffmann, O., Vasserot, E., On cohomological Hall algebras of quivers : Generators, J. reine angew. Math. 760 (2020), 59-132.
\bibitem{ST11} Savage, A., Tingley, P., Quiver Grassmannians, quiver varieties and the preprojective algebra, 
Pacific J. Math. 251 (2011), 393-429.
\bibitem{S06} Schlichting, M., Negative K-theory of derived categories, Math. Z. 253 (2006), 97-134.
\bibitem {S80} Slodowy, P., Four Lectures on Simple Groups and Singularities, Comm. Math. Inst., Rijksuniv. Utrecht 11, Math. Inst. Rijksuniv., Utrecht, 1980.
\bibitem{T97} Thomason, R. W.,  The classification of triangulated subcategories, Compos. Math. 105 (1997), 1-27.
\bibitem{T19} Toda, Y., Categorical Donaldson-Thomas theories for local surfaces,
arXiv:1907.09076v4.
\bibitem{T17} Tsymbaliuk, A., The affine Yangian of $\frakg\frakl_1$ revisited, 
Adv. Math. 304 (2017), 583?645.
 \bibitem{VV99} Varagnolo, M., Vasserot, E., On the K-theory of the cyclic quiver variety,
 Int. Math. Res. Not. (1999), 1005-1028
 \bibitem{VV03} Varagnolo, M., Vasserot, E., Canonical bases and quiver varieties, Represent. Theory 7 (2003), 227-258.
\bibitem{VV22}  Varagnolo, M., Vasserot, E.,  K-theoretic Hall algebras, quantum groups and super quantum
groups, Selecta Math. (N.S.) 28 (2022).
\bibitem{VV23}  Varagnolo, M., Vasserot, E., Non symmetric quantum loop groups and critical convolution algebras, arXiv:2308.01809.
\bibitem{YZ20} Yang, Y., Zhao, G., On two cohomogical Hall algebras,  Proc. Roy. Soc. Edinburgh Sect. A 150 (2020), 1581-1607.



\end{thebibliography}
\end{document}